\documentclass[12pt,reqno, english]{amsart}
\usepackage{tikz}
\usepackage[all]{xy}
\usepackage{color}
\usetikzlibrary{automata}
\usepackage{tkz-euclide}
\usetikzlibrary{matrix}
\usepackage{amsmath}
\usepackage{amsthm, textcomp}
\usepackage{bm}
\usepackage{centernot}
\usepackage{mathtools}
\usepackage{ stmaryrd }
\usepackage{amsfonts}
\usepackage{mathtools}
\usepackage{marginnote}
\usetikzlibrary{fit,arrows,calc,positioning, trees}
\usetikzlibrary{arrows.meta, quotes}
\usepackage{graphicx}
\usepackage{latexsym}
\usepackage[left=2cm, right=2cm, top = 1.5cm, bottom=2 cm]{geometry}
\usepackage[english, francais]{babel}
\usepackage{amssymb}
\usepackage[utf8]{inputenc}
 \usepackage{lipsum}
\makeatletter
\g@addto@macro{\endabstract}{\@setabstract}
\newcommand{\authorfootnotes}{\renewcommand\thefootnote{\@fnsymbol\c@footnote}}%
\makeatother

\usepackage{amsmath,amsthm,amsfonts,amssymb}

\usepackage{epsfig}
\usepackage{hyperref}
\usepackage{amssymb}
\usepackage{xcolor}
\usepackage{mathtools}
\usepackage[all]{xy}
\xyoption{poly}
\usepackage{fancyhdr}
\usepackage{wrapfig}
\usepackage{makecell}
\usepackage[T1]{fontenc}
\usepackage{amssymb,amsmath,amsthm}
\usepackage{amsmath,amssymb,amsthm}
\usepackage{graphicx}

\newcommand{\ncmd}{\newcommand}

\newtheorem{theorem}{Theorem}
\newtheorem{lemma}{Lemma}
\newtheorem{corollary}{Corollary}
\newtheorem{proposition}{Proposition}
\newtheorem{conjecture}{Conjecture}

\theoremstyle{definition}

\newtheorem{remarkimp}{Remark}
\newtheorem*{remark}{Note}

\newtheorem*{example}{Example}

\ncmd{\A}{\mathcal{A}}
\ncmd{\Ftau}{\mathcal{F}^{\tau}}
\ncmd{\SAF}{\mathrm{SAF}}
\ncmd{\Fone}{\mathcal{F}_1}
\ncmd{\aaa}{\boldsymbol{a}}
\ncmd{\AR}{\mathrm{AR}(\Sph^1)}
\ncmd{\Ftwo}{\mathcal{F}_2}
\ncmd{\IET}{\mathrm{IET}}
\ncmd{\FET}{\mathrm{FET}}
\ncmd{\IETFn}{\mathrm{IETF}^n}
\ncmd{\IETFplus}{\mathrm{IETF}^{n+1}}
\ncmd{\IETFtwo}{\mathrm{IETF}^2}
\ncmd{\IETFthree}{\mathrm{IETF}^3}
\ncmd{\IETFfour}{\mathrm{IETF}^4}
\ncmd{\IETFfive}{\mathrm{IETF}^5}
\ncmd{\FETthree}{\mathrm{FET}^3}
\ncmd{\FETn}{\mathrm{FET}^n}
\ncmd{\FETfive}{\mathrm{FET}^5}
\ncmd{\wind}{\mathrm{wd}}
\ncmd{\midd}{\mathrm{mid}}
\ncmd{\FETplus}{\mathrm{FET}^{n+1}}
\ncmd{\CETn}{\mathrm{CET}^n_{\tau}}
\ncmd{\gammain}{\gamma_{\mathrm{in}}}
\ncmd{\gammaout}{\gamma_{\mathrm{out}}}
\ncmd{\deltain}{\delta_{\mathrm{in}}}
\ncmd{\deltaout}{\delta_{\mathrm{out}}}
\ncmd{\CETthree}{\mathrm{CET}^3_{\tau}}
\ncmd{\CETfive}{\mathrm{CET}^5_{\tau}}
\ncmd{\CETthreehalf}{\mathrm{CET}^3_{\frac{1}{2}}}
\ncmd{\CETfourhalf}{\mathrm{CET}^4_{\frac{1}{2}}}
\ncmd{\CETnhalf}{\mathrm{CET}^n_{\frac{1}{2}}}
\ncmd{\Rauzy}{\mathcal{R}}

\ncmd{\CEThalf}{\mathrm{CET}^{3}_{\frac{1}{2}}}
\ncmd{\CETfour}{\mathrm{CET}^4_{\tau}}
\ncmd{\E}{\mathbb{E}}
\ncmd{\Oc}{\mathbb{O}}
\ncmd{\Ha}{\mathbb{H}}

\def\R{\mathbf{R}}
\def\Q{\mathbf{Q}}
\def\N{\mathbf{N}}
\ncmd{\Z}{\mathbb{Z}}
\ncmd{\Sph}{\mathbb{S}}
\ncmd{\T}{\mathbb{T}}
\ncmd{\D}{\mathbb{D}}
\ncmd{\re}{\mathrm{Re}}
\ncmd{\im}{\mathrm{Im}}
\ncmd{\sing}{\mathrm{sing}}
\ncmd{\reg}{\mathrm{reg}}
\ncmd{\red}{\mathrm{red}}
\ncmd{\ttop}{\mathrm{top}}
\ncmd{\bbot}{\mathrm{bot}}
\ncmd{\bs}{\backslash}
\ncmd{\ov}{\overline}
\ncmd{\noi}{\noindent}
\ncmd{\di}{\displaystyle}
\ncmd{\ra}{\rightarrow}
\ncmd{\lra}{\longrightarrow}
\newcommand{\Addresses}{{
  \bigskip
  \footnotesize
  Olga~Paris-Romaskevich, \textsc{Univ Rennes, CNRS, IRMAR - UMR 6625, F-35000 Rennes}\par\nopagebreak
  \textit{E-mail address}, O.~Paris-Romaskevich: \texttt{olga@pa-ro.net, olga.romaskevich@univ-rennes1.fr}
}}

\title{Trees and flowers on a billiard table}
\author{Olga Paris-Romaskevich}
\date{July 2019}

\begin{document}

\maketitle
\begin{center}
\emph{to Manya and Katya}
\end{center}

\smallskip

\begin{abstract}
In this work we completely describe the dynamics of triangle tiling billiards. In the first part of this work, we propose a geometric approach of dynamics by introducing natural foliations associated to it. In the second part, we exploit the relationship between triangle tiling billiards and a family of fully flipped $3$-interval exchange transformations on the circle. We give a combinatorial approach of dynamics via renormalization. By uniting the two approaches, we prove several conjectures on the dynamics of triangle tiling billiards. First, we prove the Tree Conjecture and the 4n+2 Conjecture, both concerning the symbolic dynamics of periodic trajectories, and both stated by Baird-Smith, Davis, Fromm and Iyer. Second, we study a family of exceptional trajectories which are closely related to the orbits of minimal Arnoux-Rauzy maps. We prove that all of these exceptional trajectories pass by all tiles, which confirms our own conjecture with P. Hubert on their non-linear escape. Moreover, we use tiling billiards to prove the convergence, up to rescaling, of arithmetic orbits of the Arnoux-Yoccoz map to the Rauzy fractal, conjectured by Hooper and Weiss. All of these conjectures have been stated in print in the last three years.
\end{abstract} 
\begin{center}
{\textbf{Introduction, motivation and overview of results.}}
\end{center}

\bigskip

A tiling billiard is a model of movement of light in a heterogeneous medium that is constructed as a union of homogeneous pieces, see \cite{DDRSL16} and \cite{G16} for the first mathematical approaches of the subject and definitions. The defintion of a tiling billiard is the following. Take any tiling of a plane by polygons and define a billiard on it such that a point particle moves in a straight line till the moment when it reaches a border of a tile. Then it passes to a neighboring tile, and its direction follows Snell's law with a fixed local refraction coefficient $k$. In this work, we only consider the case where $k \equiv -1$, see Figure \ref{fig:Snellslaw}. We are interested in the dynamics of particles in such a class of dynamical systems, the so-called  \textbf{tiling billiards} \cite{DDRSL16}. The dynamics of a tiling billiard depends very strongly on the underlying tiling, see Figure \ref{fig:examples_of_tiling_billiards} for examples.

\begin{figure}
\centering
\includegraphics[scale=0.05]{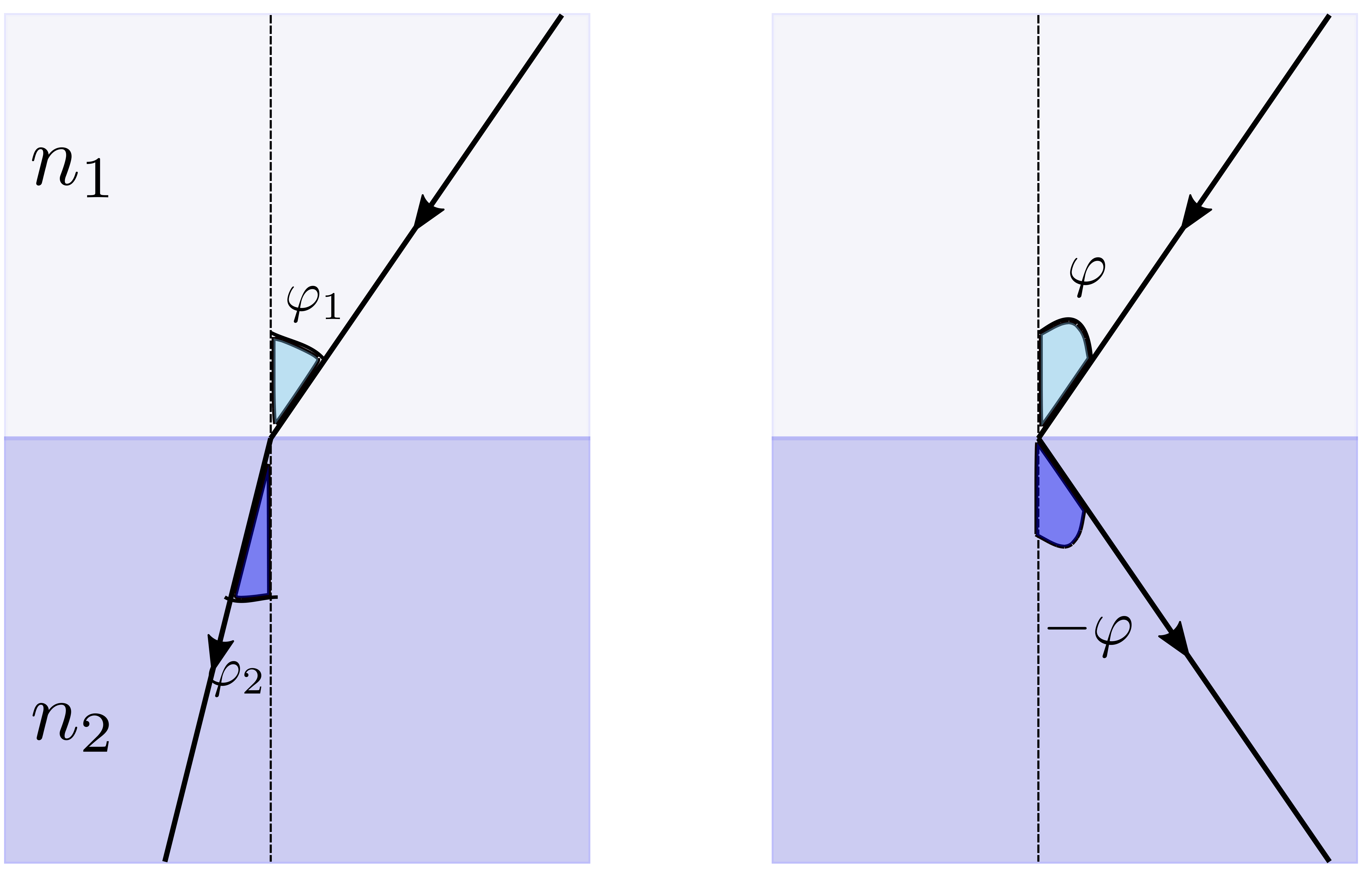}
\caption{\emph{Snell's law of refraction.} On the left: a ray of light crosses the boundary between two homogeneous materials with refraction indices $n_1, n_2 \in \R$. Then for any ray of light, the relationship between the angles $\varphi_1$ and $\varphi_2$ is defined by the Snell's law of refraction $\frac{\sin \varphi_1}{\sin \varphi_2}=\frac{n_2}{n_1}=:k$. For example, for the air and the water correspondingly, $n_1 \approx 1, n_2 \approx 1.333$. On the right: the behavior of a ray of light in the case where the refraction coefficient $k=-1$.}\label{fig:Snellslaw}
\end{figure}
\smallskip

The mathematical study of tiling billiards was proposed in \cite{DDRSL16} several years ago. The study of tiling billiards is quite a new subject in mathematics. Although tiling billiards have already proven their richness and interest from the point of view of dynamics, see \cite{BDFI18, DH18, Olga}. The study of tiling billiards stays for now a highly unexplored area even though its interest for mathematics is straightforward. Indeed, such a dynamics is related to the dynamics of geodesic flows on non-orientable flat surfaces which is an unexplored area of the general theory. The only non-trivial examples of tiling billiards for which the dynamics has been studied in some detail are that of a tiling billiard on a trihexagonal tiling \cite{DH18} and on a periodic triangle tiling \cite{BDFI18, Olga}.

\bigskip

Concerning physical relevance of tiling billiards, the materials having the refraction index equal to $-1$ can be quite easily constructed \footnote{Most of usual plastic or glass materials have indices of refraction bigger than $1$, and metamaterials with negative indices of refraction are usually artificially constructed.} (for example, as slabs of photonic crystals) even though it would necessarily imply for these materials to be strongly dispersive with frequency. This implies that an even more physically relevant (and more complicated...) model of a billiard in a tiling should include an additional parameter $f$ which corresponds to the light frequency, with the refraction coefficient $k=k(f)$ depending on it. 

\smallskip

There has been quiet a big body of research in physics of metamaterials related to tiling billiards.  In particular, we send our readers to the works \cite{Guenneau1, Liu, Guenneau2, Ramakrishna}.
The periodic trajectories in tiling billiards model (with $k(f)=-1$) correspond to the resonances in the full wave picture (where $k=k(f)$ is a function of the initial frequency), which are important for super resolution.  Negative refraction materials, as well as complementary media, remain active areas of research for modern physics, with numerous possible applications. One of such applications could be the construction of invisibility cloaks, see \cite{Wood} and references within. We hope that a subject of tiling billiards could potentially reunite mathematics and physics communities around this fascinating dynamics.

\bigskip

This work considers tiling billiards on two tilings that have many common features. These are a \textbf{periodic triangle tiling} and a \textbf{periodic cyclic quadrilateral tiling}, and are defined as follows.
Each of these two tilings consists of congruent triangles (or cyclic quadrilaterals\footnote{A \textbf{cyclic quadrilateral} is a quadrilateral inscribed into a circle.}) and has a property that each of two neighbouring tiles are centrally symmetric to each other with respect to the middle of their common side, see Figure  \ref{fig:triangle_tiling_and_cyclic_quadrilateral_tiling}.  Such a periodic tiling of a plane by quadrilaterals always exists, whatever the form of a quadrilateral. Although in this work we are interested only in the special case of cyclic quadrilateral tilings, since these are the only ones that admit the \emph{folding construction}. We discuss this construction in Section \ref{sec:locally foldable tilings and associated billiard foliations}.

\smallskip

Whenever we refer to a tiling billiard, we suppose that this is a tiling billiard in a periodic triangle or cyclic quadrilateral tiling. We call these two tilings simply \textbf{triangle} and \textbf{quadrilateral tilings} and corresponding dynamical systems \textbf{triangle (quadrilateral) tiling billiards}.

\begin{figure}
\centering
\includegraphics[scale=0.43]{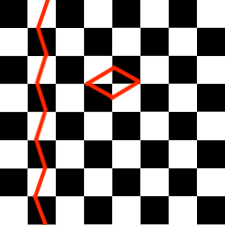}
\includegraphics[scale=0.18]{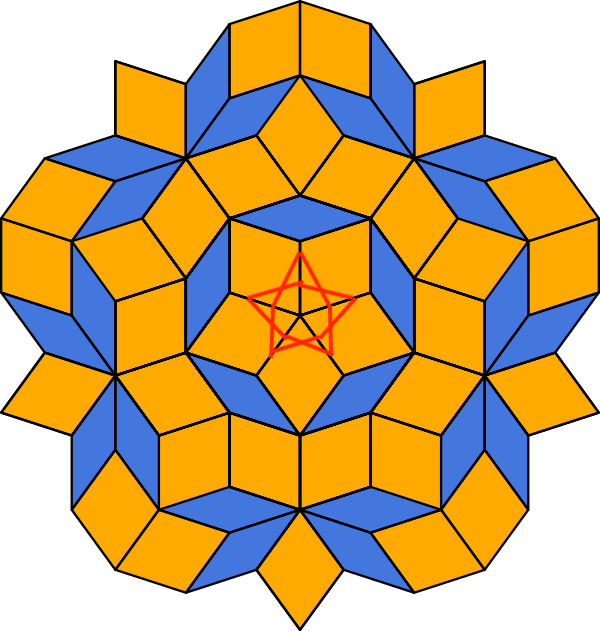}
\includegraphics[scale=0.13]{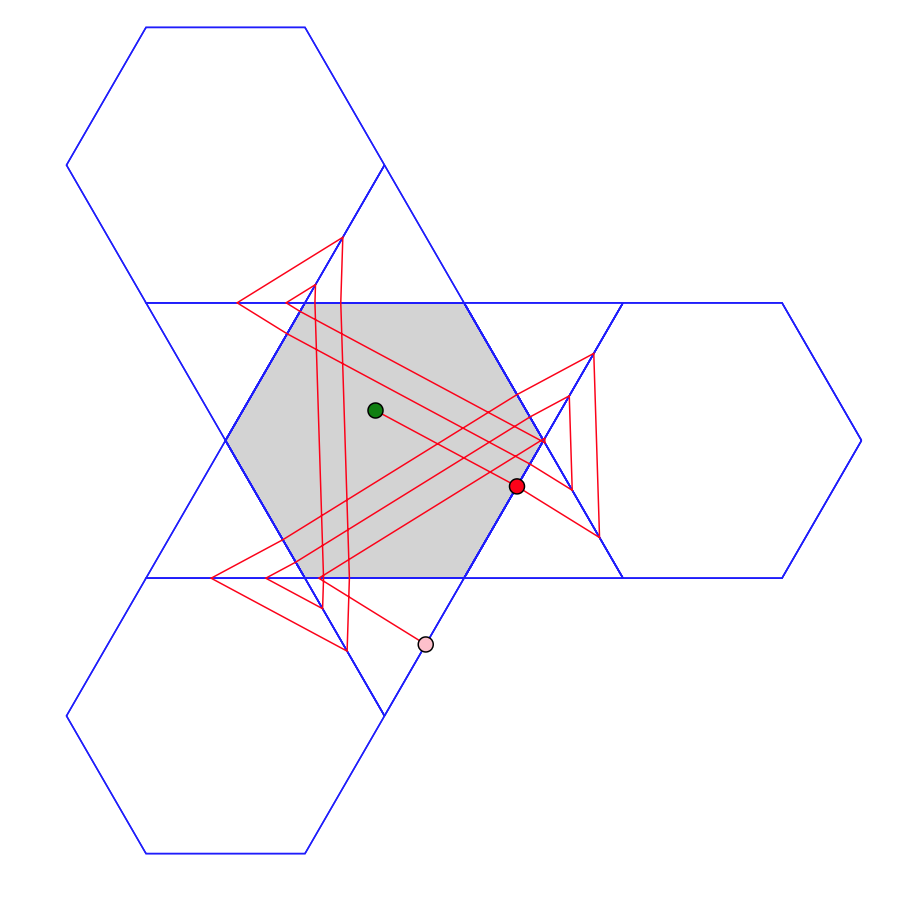}
\includegraphics[scale=0.2]{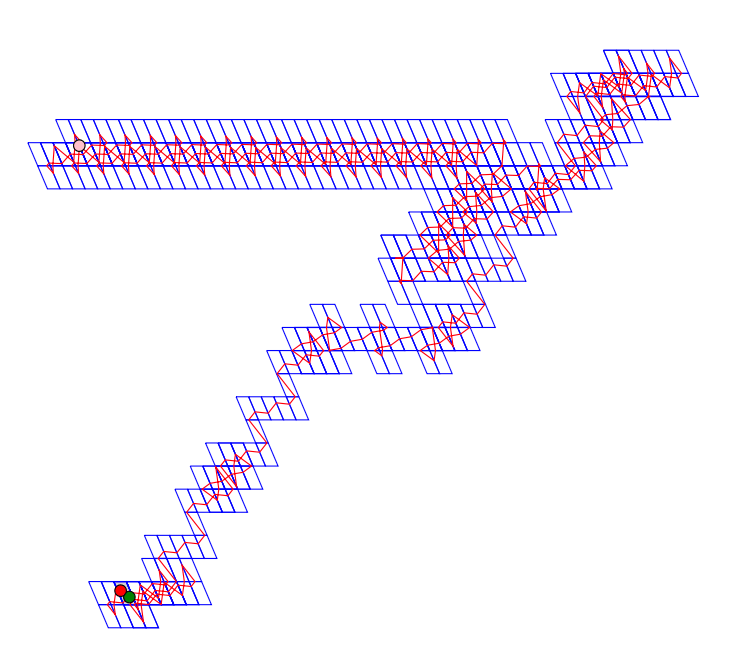}
\caption{Examples of tiling billiards and related questions. From left to right: 1. A \emph{square tiling} with trivial behavior of trajectories: all trajectories are either of period $4$ or are vertically (or horizontally)
shift $2$-periodic. All bounded trajectories are periodic. Does this property persist for a larger class of tilings? 2. \emph{Penrose tiling} and a periodic trajectory in it.  How often do periodic trajectories occur in Penrose tilings? 3. \emph{Trihexagonal tiling billiard}'s trajectories exhibit ergodic properties as was shown in \cite{DH18}. Are the ergodic properties preserved in the bifurcation to a periodic triangle tiling, where all of the positively oriented triangles grow bigger, all of the negatively oriented triangles grow smaller, and hexagons converge to triangles? 4. Do all of the trajectories in a \emph{parallelogram tiling} escape linearly to infinity?}
\label{fig:examples_of_tiling_billiards}
\end{figure}

\begin{figure}
\centering
\includegraphics[scale=0.8]{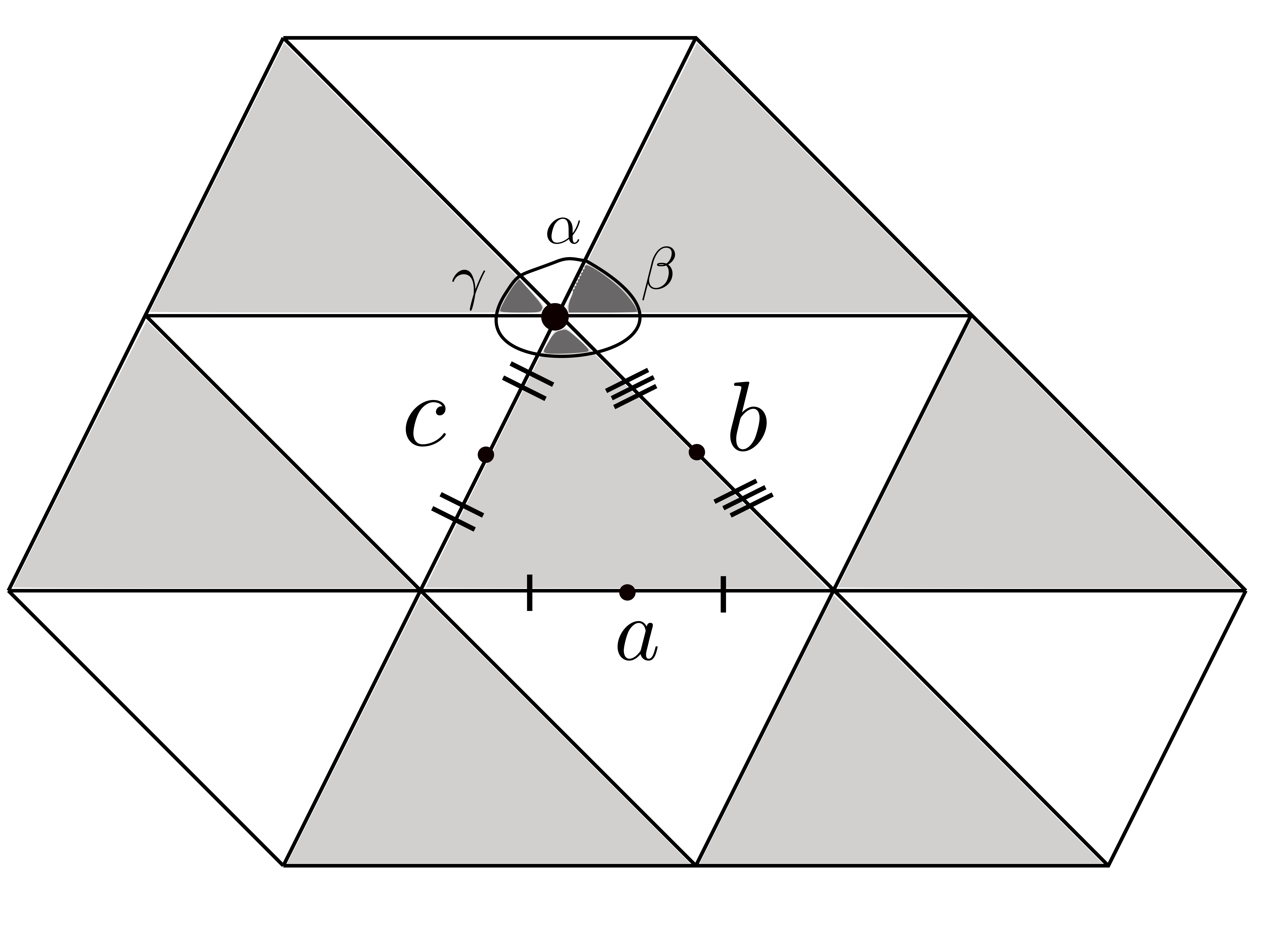}
\includegraphics[scale=0.04]{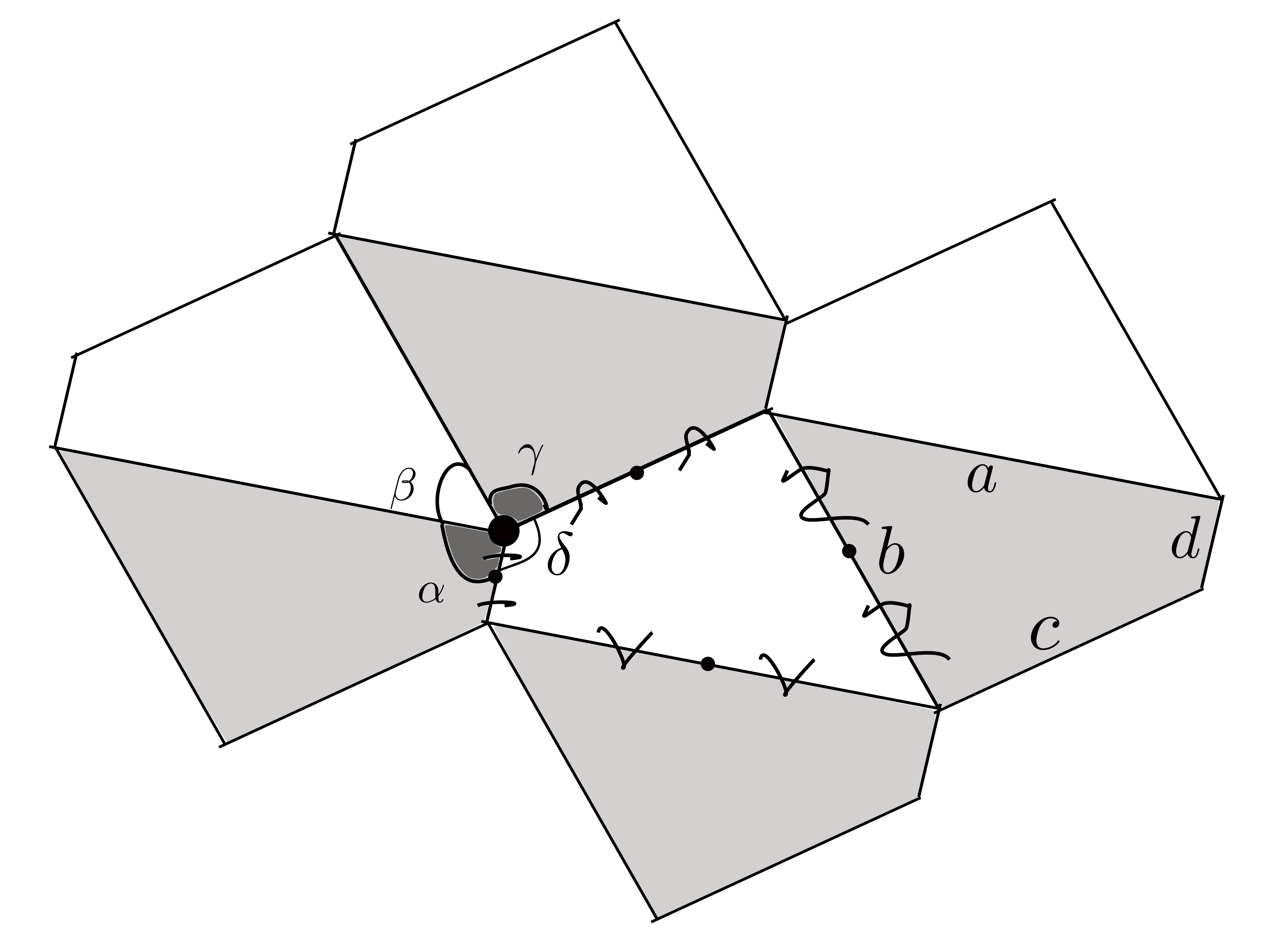}
\caption{On the left - a \emph{triangle tiling}, on the right - a \emph{cyclic quadrilateral tiling}. In both tilings, the neighbouring tiles are centrally symmetric with respect to the middle of their common side. Both tilings are colorable in two alternating colors and for each vertex a sum of angles of tiles of each color containing it is equal to $\pi$. For the triangle tiling such a relation on the angles is trivial, and for a general quadrilateral tiling this relationship is equivalent to the fact that the tile is cyclic.}
\label{fig:triangle_tiling_and_cyclic_quadrilateral_tiling}
\end{figure}

\bigskip

The dynamics of a triangle (quadrilateral) tiling billiard is equivariant under homothety of the plane. The parameters of such dynamics hence encode a form of a tile but not its size. Denote the angles of a tile by $\alpha, \beta$, $\gamma$ (and $\delta$, in case of a cyclic quadrilateral tiling). Let the corresponding sides be $a,b$ and $c$ (and $d$ \footnote{For a triangle tiling, the sides corresponding to the angles are opposite sides. For a quadrilateral tiling, the correspondance of sides and angles is reflected on Figure \ref{fig:triangle_tiling_and_cyclic_quadrilateral_tiling}.}). Moreover, suppose that any tile in the tiling is oriented in such a way that a counterclockwise tour of its boundary reads the letters in the alphabetical order. Both triangle and quadrilateral tilings can be colored  into two colors in such a way that neighbouring tiles have different colors and that tiles with the same color can be identified by a translation. We call the tiles of one of the colors \textbf{positively oriented}, and of another color \textbf{negatively oriented}, in an arbitrary way, see Figure \ref{fig:triangle_tiling_and_cyclic_quadrilateral_tiling}.

\smallskip

Triangle tiling billiards with the refraction index equal to $-1$ were introduced in \cite{DDRSL16} by D. Davis, K. DiPietro, J. Rustad and A. St Laurent. They were subsequently studied in much more detail by P. Baird-Smith, D. Davis, E. Fromm and S. Iyer in \cite{BDFI18}. In particular, the authors show the relationship between triangle tiling billiards and fully flipped $3$-interval exchange transformations on the circle. With P. Hubert we continued their study. In our work \cite{Olga} we have given a qualitative description of generic trajectories as well as have described the set of trajectories with non-generic behavior. Even though some understanding of the dynamics of triangle tiling billiards was achieved in \cite{BDFI18} and \cite{Olga}, a precise description of symbolic dynamics of trajectories was far from being complete. 

In the present work we describe completely the dynamics of triangle tiling billiards. A first ingredient in our description are \emph{tiling billiard foliations}, to which is dedicated the first part of this work. With the use of these foliations, we prove the Tree Conjecture (formulated in \cite{BDFI18}) on the symbolic dynamics of periodic trajectories as a main result of the first part.\footnote{A pinch of the second ingredient also appears in the first part, since $4n+2$ Conjecture is used to prove the Tree Conjecture.}

A second ingredient in the complete description of triangle tiling billiard dynamics is a renormalization process for fully flipped $3$-interval exchange transformations that we describe in the second part of this work. By putting these two ingredients together we give simpler proofs of the main results in \cite{Olga} and prove several additional results.

\smallskip

In the following two sections of this Introduction, we remind our reader on previously discovered results on triangle tiling billiards. We also provide the context by giving the definitions of classical objects that reveal themselves related to triangle tiling billiards. In these two sections we also present our main results, athough in the body of the article most of these results are formulated in a greater generality. Finally, in the Section \ref{sec:plan} of the Introduction, we give a detailed plan of this article.

\section{Symbolic dynamics of triangle tiling billiards.}\label{subs:symbolic_intro}

\subsection{Triangle tiling billiards: known results.}\label{subs:triangle tiling - known results}

The results of this paragraph come entirely from \cite{BDFI18} and \cite{Olga}.\footnote{We make a following bibliographical remark. Even though the article \cite{Olga} has been published (and even, appeared online) before \cite{BDFI18}, we have studied in detail an early draft by Baird-Smith, Davis, Fromm and Iyer from $2017$, and our work \cite{Olga} (as well, as this work) is based on their results and ideas.} 

\smallskip

A \textbf{symbolic code} of an oriented curve on the plane with respect to some triangle tiling is defined as a word in the alphabet of sides $\mathcal{A}_{\Delta}:=\{a,b,c\}$. This symbolic code corresponds to the sequence of sides, crossed by this curve.  For example, a code of a curve $\delta$ making a  clock-wise circular tour of a vertex in a tiling is $abcabc$. This code can be considered as an infinite word in  $\mathcal{A}_{\Delta}^{\N}$ (and in this case we write it as $\overline{abcabc}=\overline{abc}$\footnote{In our notations, a word under the bar in such a representation is a period of an infinite word in $\mathcal{A}_{\Delta}^{\N}$.}) or as a periodic cyclic word.

In the following, we also use another coding for an oriented curve in the triangle tiling which is defined in the alphabet of couples of sides $\mathcal{A}_{\Delta}^2:= 
\{ab,ba, bc, cb, ca, ac\}$. The \textbf{accelerated symbolic code} of an oriented curve on the plane with a tiling is defined by a sequence of couples of crossed edges. For example, the accelerated symbolic code of the circular curve $\delta$ is now $ab\; bc \; ca \; ab \; bc \; ca$. Of course, a symbolic code in $\mathcal{A}_{\Delta}$ and an accelerated symbolic code in $\mathcal{A}_{\Delta}^2$ obviously translate one into another. The accelerated symbolic code is a redundant notation for a symbolic code of a periodic trajectory but it happens to be more convenient in some situations as we will see in the future. 

\smallskip

One can now speak about the \textbf{symbolic dynamics of triangle tiling billiards} by defining a shift map on the subset of possible symbolic codes of trajectories. 

\begin{example}
For a periodic trajectory of a triangle tiling billiard depicted on Figure \ref{fig:example_traj_symb},\footnote{The Figures \ref{fig:example_traj_symb}, \ref{fig:tree_conjecture_picture} and \ref{fig:drifty} representing triangle tiling billiard trajectories are drawn with the help of the program by P. Hooper and A. St Laurent accessible online. Our reader can go and play with triangle tiling billiards by following the link \url{http://awstlaur.github.io/negsnel/}. This program doesn't only model triangle tiling billiards but also parallelogram, two-square, trihexagon and octagon-square tiling billiards.} its symbolic code in the alphabet $\mathcal{A}_{\Delta}$ is given by a periodic word $w=\overline{abacacacbacac}$.
\end{example}

One of our main interests in this work is the symbolic dynamics of triangle tiling billiards. We are interested in a following question.
What words in the alphabet $\mathcal{A}_{\Delta}$ (and $\mathcal{A}_{\Delta}^2$) correspond to triangle billiard trajectories? We answer this question in the second part of this work, see in particular Proposition \ref{prop:SYMBOLIC}. 

\begin{figure}
\centering
\includegraphics[scale=0.02]{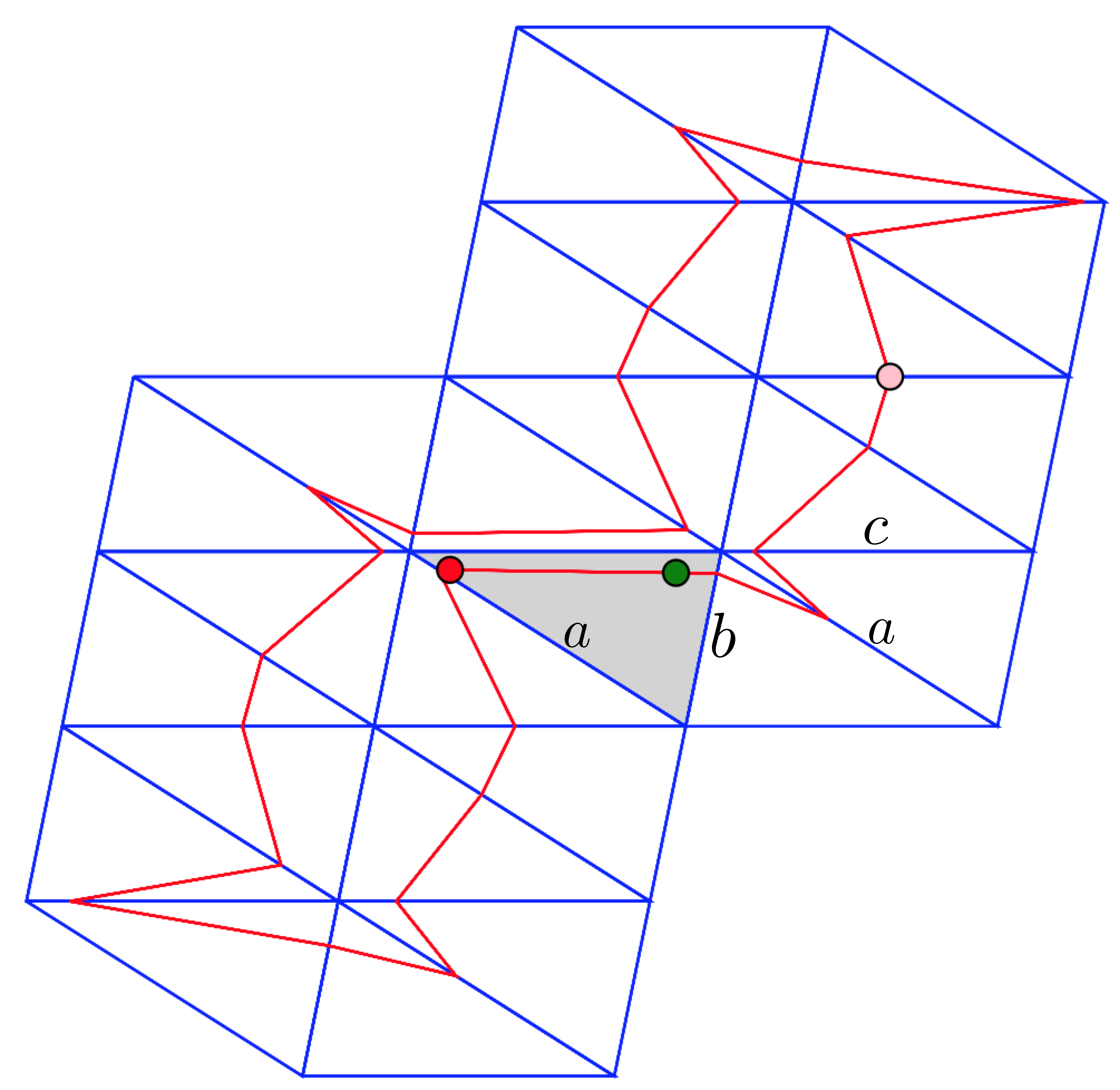}
\caption{\emph{Example of a periodic trajectory of a triangle tiling billiard. The first four letters $abac$ in a symbolic code label the corresponding crossed sides, one contines along a trajectory to form a complete symbolic code.}}\label{fig:example_traj_symb}
\end{figure}

\bigskip

The state of art on the symbolic behavior of triangle tiling billiard trajectories can be summarized in a following 

\begin{theorem}\cite{BDFI18, Olga}\label{thm:triangle_tiling_billiards_info}
Consider a triangle tiling billiard. Then the following holds:
\begin{itemize}
\item[1.] Every trajectory passes by each tile at most once. Additionally, the oriented distance between a segment of a trajectory in some tile and its circumcenter is preserved along all the trajectory;
\item[2.] all bounded trajectories are periodic and simple closed curves;
\item[3.] all bounded trajectories are stable under small perturbations (form of a tile, initial condition), i.e. they deform to bounded trajectories with the same symbolic code in $ \mathcal{A}_{\Delta}^{\N}$;
\item[4.*] the period of any periodic trajectory belongs to the set $\left\{4n+2 \left| \right. n \in \N^*\right\}$;
\item[5.*] the symbolic code $w \in \mathcal{A}_{\Delta}^{\N}$ of any periodic trajectory has its smallest period $s \in \mathcal{A}_{\Delta}$ of odd length. A complete period of a periodic trajectory is then described by the word $s^2$, and $w=\overline{s^2}$.
\end{itemize}
\end{theorem}

This Theorem implies that the periodic trajectory on Figure \ref{fig:example_traj_symb} is not exotic but generic and stable, since the periodicity of trajectories is an open property in triangle tiling billiards.

The statements 1.--3. have been proven and 4. has been conjectured in \cite{BDFI18}. The first three statements are consequences of an important folding idea, see Section \ref{sec:locally foldable tilings and associated billiard foliations}. The point 4. is a simple consequence of 5. 

The statements 4.--5. have been announced to be proven in \cite{Olga} by P. Hubert and myself.  Our proof of 4. and 5. presented in \cite{Olga} is based in a crucial way on the relation of triangle tiling billiards with interval exchange transformations with flips that was discovered in \cite{BDFI18}. This proof is quite technical (it uses the explicit construction of Nogueira-Rauzy graphs), and, unfortunately, incomplete as we have discovered while working on this paper. The proof could be easily completed and finished combinatorially, along the lines and methods of the initial article. In this work, we give an alternative and much simpler proof of 4. and 5. (see Theorem \ref{thm:one-more-time}) and hence give a first complete proof of these two statements that are known as $4n+2$ Conjecture and were initially formulated in \cite{BDFI18}. See the Appendix for more comments on the work \cite{Olga}.

\bigskip

We say that a triangle tiling billiard trajectory is \textbf{escaping to infinity} (or simply \textbf{escaping}) if it is not periodic. This name makes sense since by the point 1. in Theorem \ref{thm:triangle_tiling_billiards_info}, any trajectory which is not periodic, is not \emph{"spiraling"} in and out in a bounded domain of a plane\footnote{The absence of spiraling for tiling billiards is \emph{a priori} possible, see for example the dynamics of tiling billiards in trihexagonal tilings.} but genuinely escapes to infinity. A trajectory is \textbf{linearly escaping} if it escapes to infinity and stays in a bounded distance from a fixed straight line. Any triangle tiling billiard trajectory is either \textbf{periodic}, \textbf{linearly escaping} or \textbf{non-linearly escaping}, as follows from 1.--2. in Theorem \ref{thm:triangle_tiling_billiards_info} and was proven in \cite{BDFI18}. As proven in \cite{Olga}, \emph{almost any} trajectory of a tiling billiard in a fixed triangle tiling  \emph{is either periodic or linearly escaping}. In order to make this statement more precise, we need one more definition. We start by defining a set of measure $1$ of triangle tilings in which all trajectories are periodic or linearly escaping.

\smallskip

Let $\Delta_2:=\left\{(x_1, x_2, x_3)|x_i \geq 0, x_1+x_2+x_3=1\right\} \subset \R^3$. If $x_j>\frac{1}{2}$ for some $j$, one maps a triple $(x_1, x_2, x_3)$ to a new one where $x_j':=2x_j-1$ and the other two coordinates $x_i, i \neq j$, stay unchanged. Then we normalize by $x_j$ to get back to $\Delta_2$. In projective coordinates, this is equivalent to subtracting the sum of two smaller coordinates from the biggest one. We call this operation on $\Delta_2$ the \textbf{Rauzy subtractive algorithm}. The subset $\overline{\mathcal{R}}\subset \Delta_2$ of triples on which the Rauzy subtractive algorithm can be applied infinitely, was defined in \cite{AS13} by P. Arnoux and S. Starosta. They have also proven that the set $\overline{\mathcal{R}}$ is homeomorphic to the Sierpinsky triangle. The questions related to it were studied in many works, see for example \cite{AR91, AHS16, AHS16-1}. See Figure \ref{fig:rauzygasket} for the illustration of the set $\overline{\mathcal{R}}$. We define $\mathcal{R} \subset \overline{\mathcal{R}}$ as a set on which for the Rauzy subtractive algorithm $x_j \neq \frac{1}{2}$ at each step, in other words the inequality $x_j > \frac{1}{2}$ is strict. In the following we call this set $\mathcal{R}$ the \textbf{Rauzy gasket} (even though one usually calls $\overline{\mathcal{R}}$ the Rauzy gasket but in this work we exclude its boundary to define $\mathcal{R}$). 

The set $\mathcal{R}$, seemingly unnatural if introduced as above, appears to be a set of parameters for the set of \emph{interesting} maps in various dynamical contexts, see for example the works by Avila-Hubert-Skripchenko on systems of isometries \cite{AHS16, AHS16-1}, by Dynnikov-DeLeo \cite{DDL09} on sections of $3$-periodic surfaces, by Arnoux-Rauzy \cite{AR91} on $6-$interval exchange transformations on the circle. These works and many others show that the set $\mathcal{R}$ represents a great interest for modern dynamics. It is still not completely understood, for example it is an open question to calculate  its Hausdorff dimension. 

The set $\mathcal{R}$ is related to triangle tiling billiards, as shown in \cite{Olga}. Indeed, this set parametrizes the rare forms of tiles for which the corresponding triangle tiling billiards admit trajectories escaping in a non-linear way.
 
\begin{figure}
\centering
\includegraphics[scale=0.3]{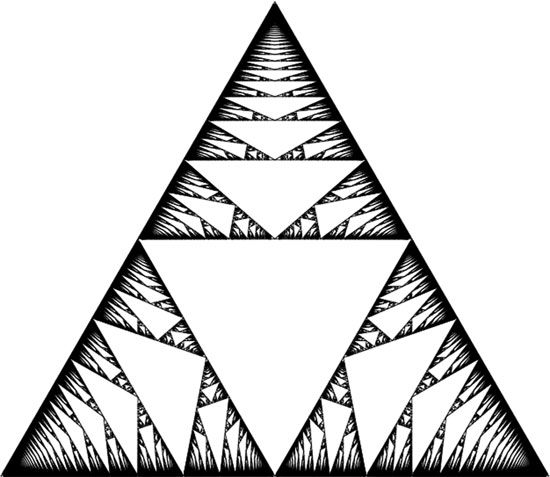}
\caption{The sett $\overline{\mathcal{R}} \subset \Delta_2$ is a fractal set homeomorphic to Sierpinsky triangle.}
\label{fig:rauzygasket}
\end{figure} 
 
\smallskip

The forms of tiles in triangle tilings are parametrized by their angles.
Consider the set of triangular tiles such that the point 

\begin{equation}\label{eq:rho_delta}
\rho_{\Delta}:=\left(1-\frac{2}{\pi}\alpha, 1-\frac{2}{\pi}\beta, 1-\frac{2}{\pi}\gamma \right) \in \Delta_2
\end{equation}

belongs to the Rauzy gasket $\mathcal{R}$, ${\rho}_{\Delta} \in \mathcal{R}$. Of course, this set  is just an affine re-parametrization of $\mathcal{R}$. A trajectory of a triangle tiling billiard is called \textbf{exceptional} if first, a corresponding $\rho_{\Delta} \in \mathcal{R}$ and second, this trajectory passes through the circumcenter of some tile (and hence, by point 1. of Theorem \ref{thm:triangle_tiling_billiards_info}, of \emph{any} tile it crosses). 

\begin{theorem}\cite{Olga}\label{thm:exceptional-zero measure}
Fix a triangle tiling. If the angles of the tiles are such that ${\rho}_{\Delta} \notin \mathcal{R}$ then all of the trajectories in such a tiling are either periodic or linearly escaping. On the contrary, if ${\rho}_{\Delta} \in \mathcal{R}$, a trajectory escapes to infinity non-linearly \emph{only if} it passes by the circumcenters of tiles. 
\end{theorem}

In the second part of this work, we give an alternative proof of a stronger version of this Theorem. We prove that in Theorem \ref{thm:exceptional-zero measure} the \emph{only if} can be replaced by \emph{ if and only if}. The \emph{only if} direction has already been proven in \cite{Olga} for almost all ${\rho}_{\Delta} \in \mathcal{R}$ with respect to the Avila-Hubert-Skripchenko measure on the Rauzy gasket defined in \cite{AHS16, AHS16-1}. Here we prove it for \emph{all} angle parameters ${\rho}_{\Delta} \in \mathcal{R}$, see points 1. and 2. in Theorem \ref{thm:complete_classification} and Theorem \ref{thm:exceptional_trajectories_intro}.

Exceptional trajectories of triangle tiling billiards are of great interest because of their relationship to arithmetic orbits of a famous Arnoux-Rauzy family of interval exchange transformations. The better understanding of the behavior of exceptional trajectories (and their density properties) is achieved in this work by apporaching these trajectories by bigger and bigger periodic trajectories. 

The next paragraph discusses a beautiful property of periodic trajectories of triangle tiling billards that revealed itself to be not only beautiful but useful for the global understanding of the dynamics.

\subsection{Tree Conjecture: formulation and motivation.}
The Tree Conjecture concerns the symbolic behavior of \emph{any} periodic trajectory of a triangle tiling billiard, see Figure \ref{fig:tree_conjecture_picture}.
\smallskip

First, for any periodic trajectory $\delta$ in a tiling billiard denote a \textbf{domain} of the plane \textbf{that it encloses} by $\Omega^{\delta} \subset \R^2$, $\partial \Omega^{\delta}=\delta$.

Consider a triangle tiling. Denote by $\Lambda_{\Delta}:=(V,E)$ an abstract graph such that the set $V$ consists of the vertices of tiles in the plane, two vertices in $V$ being connected by an edge in $E$ if they are connected in the tiling. The abstract graph $\Lambda_{\Delta}$  comes with its embedding in the plane, it is a graph we see when we look at the triangle tiling. 

\begin{conjecture}[Tree Conjecture for triangle tilings]\label{conj:tree}
Take any periodic trajectory $\delta$ of a triangle tiling billiard. Then the graph $G_{\Delta}^{\delta}:={\Omega}^{\delta} \cap \Lambda_{\Delta}$ (as a subgraph of $\Lambda_{\Delta}$) is a \emph{tree}. In other words, a trajectory $\delta$ passes by all the tiles that intersect its interior ${\Omega}^{\delta}$.
\end{conjecture}

\begin{figure}
\centering
\includegraphics[scale=0.3]{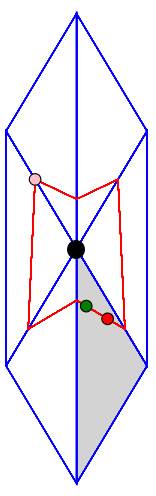}
\includegraphics[scale=0.3]{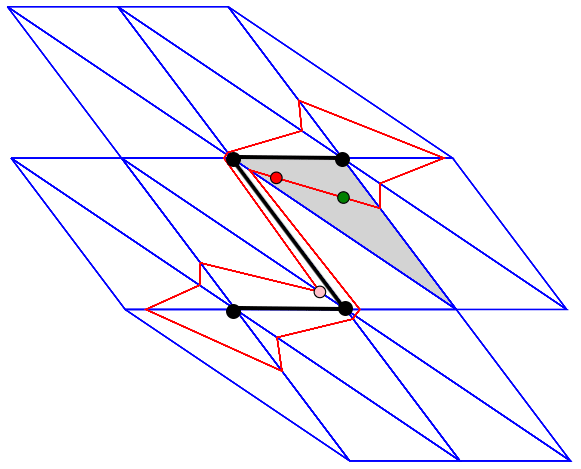}
\includegraphics[scale=0.32]{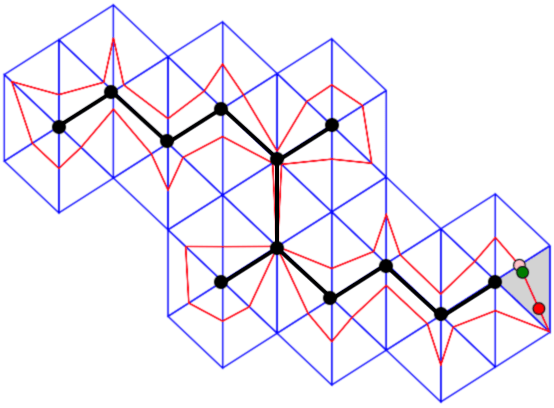}
\includegraphics[scale=0.2]{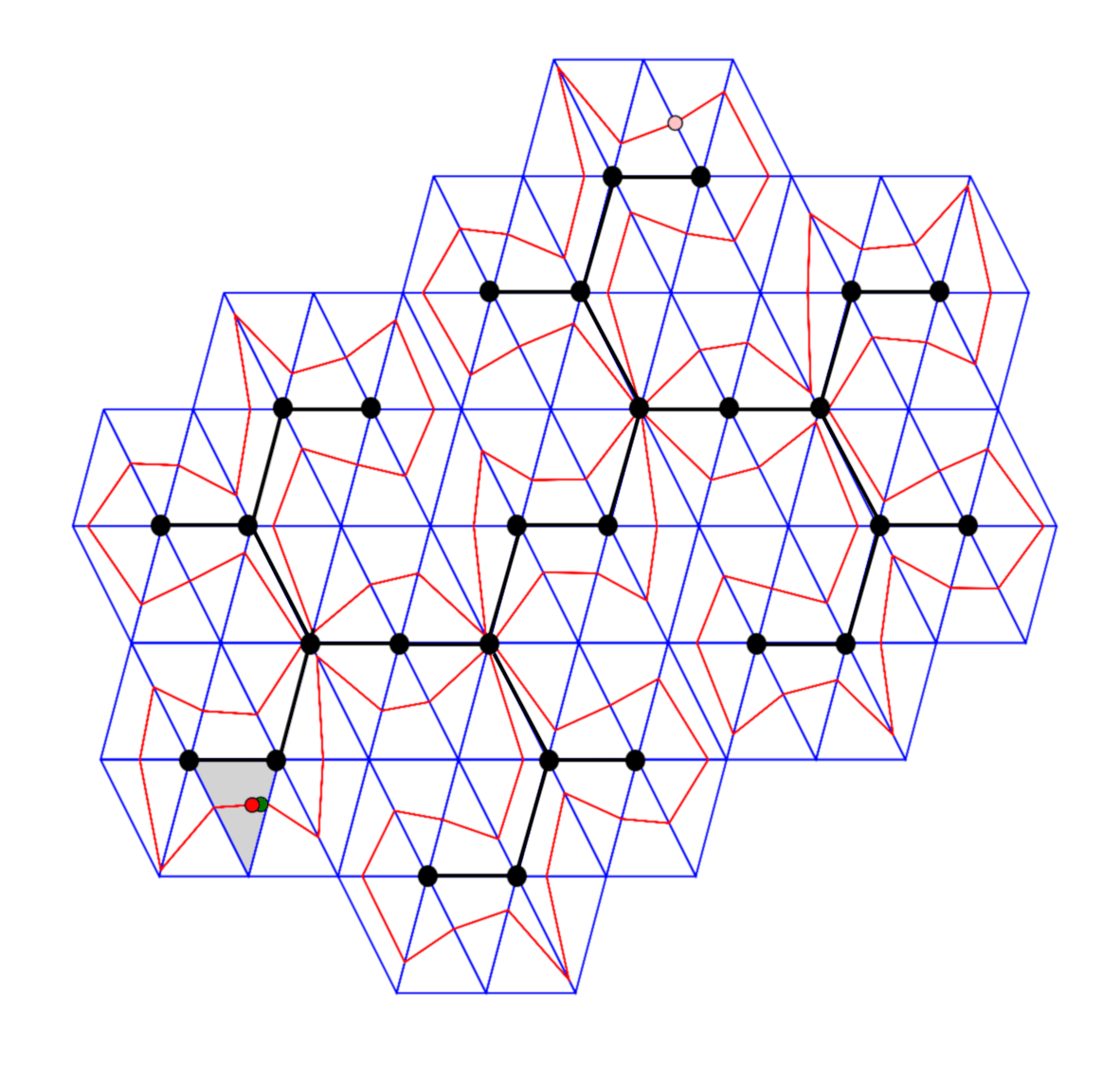}
\includegraphics[scale=0.25]{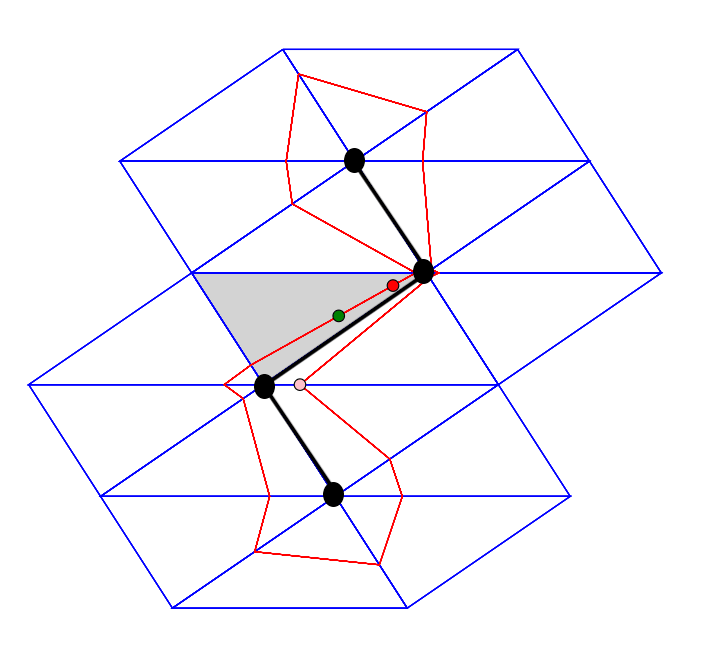}
\caption{From left to right, from top to bottom: several examples of triangle tiling trajectories and the corresponding trees. First, the simplest trajectory is a six-periodic trajectory and the corresponding graph is a simplest tree with only one vertex. Second, for the obtuse triangle tilings, the corresponding trees are always paths. We then provide three more examples of trees for acute triangles: the forms of the trees can be quite different. This Figure is based on the program by Patrick Hooper and Alexander St Laurent.}
\label{fig:tree_conjecture_picture}
\end{figure}

This conjecture was first formulated three years ago in \cite{BDFI18} and proven there for the case of tilings by \emph{obtuse} triangles, a graph $G_{\Delta}^{\delta}$ in question is in this case a chain. 

\smallskip

Our interest in the Tree Conjecture comes from its relationship to the density properties of other interesting and already studied objects, putting tiling billiards in a larger perspective. Indeed, the \emph{Tree conjecture} is a first step in our approach of the arithmetic orbits of the Arnoux-Yoccoz map (and other minimal maps in the Arnoux-Rauzy family). These orbits are fractal curves related to the Peano curve studied in \cite{A88} by P. Arnoux, and another Peano curve studied in \cite{M12} by C. McMullen and in \cite{LPV07} by J. Lowenstein, G. Poggiaspalla and F. Vivaldi. We discuss more on these curves in paragraph \ref{subs:rrdeformAR}. Of course, the Tree Conjecture is interesting in itself since it gives a partial description of the symbolic dynamics of tiling billiards.

\bigskip

The main result of the first part of this work is

\begin{theorem}\label{thm:main}
Conjecture \ref{conj:tree} holds.
\end{theorem}

The Tree Conjecture has a stronger form that we call \emph{Density property}, see Section \ref{subs:Density conjecture}. This Density property is a generalization of the Tree conjecture for any trajectory, not necessarily periodic. We prove that this property holds in Theorem \ref{thm:forest} of this work. 

Of course, analogously to the definitions of the sets $\Lambda_{\Delta}$ and $G_{\Delta}^{\delta}$ for triangle tiling billiards, one can define  $\Lambda_{\square}$ and $G_{\square}^{\delta}$ for cyclic quadrilateral tiling billiards. We suspect the analogue of the Tree Conjecture to hold for cyclic quadrilateral tilings as well but we haven't manage to prove it yet, see the discussion in Section \ref{subs:Q}. 

\smallskip

The idea of the proof of the \emph{Tree Conjecture} is as follows. In order to study the symbolic behavior of one trajectory, it is helpful to study an entire \emph{foliation} of parallel in each tile trajectories that comes with it.  Thanks to this study, the \emph{Tree Conjecture }(which deals with \emph{global} behavior of trajectories) is reduced to the \emph{Bounded Flower Conjecture} which deals with the \emph{local} behavior of separatrices in associated foliations, see paragraph \ref{subs:FLOW} for its formulation.

\smallskip

In the second part of this work, we reinforce the methods used in the proof of the Tree Conjecture with some additional renormalization arguments, in order to prove that exceptional trajectories in triangle tiling billiards pass by all triangles of the tiling. 

\begin{theorem}\label{thm:exceptional_trajectories_intro}
An exceptional trajectory of a triangle tiling billiard passes by all tiles if and only if it doesn't hit any vertex.
\end{theorem}

This Theorem is given in a slightly more general form in the text. See Theorem \ref{thm:exceptional_trajectories}, where we also cover the case of singular trajectories.

It is interesting to compare the resul t of Theorem \ref{thm:exceptional_trajectories_intro} with the results of Lowensten-Poggiaspala-Vivaldi \cite{LPV07} on the density behaviour of algebraic dynamics of the Arnoux-Yoccoz map.

\subsection{Complete description of the dynamics of triangle tiling billirds.}

Any triangle tiling defines a point in a simplex by simply taking a vector of its normalized angles 
\begin{equation}\label{eq:correspondance}
(l_1, l_2, l_3):=\left(\frac{\alpha}{\pi}, \frac{\beta}{\pi}, \frac{\gamma}{\pi}\right) \in \Delta_2.
\end{equation}

The renormalization we define in the Section \ref{sec:trop_cool} of
this work can be seen as the algorithm of induction on the orbits of triangle tiling billiards. To any orbit of a triangle tiling billiard one associates another orbit in an \emph{a priori} \emph{different} triangle tiling billiard. It happens that the renormalization process we introduce  on triangles coincides with the fully subtractive algorithm. 

\bigskip

Define a following algorithm on the triples $(l_1, l_2, l_3) \in \Delta_2$. Suppose that for some $j\ \in \mathcal{N}_{\Delta}:=\{1,2,3\}$, $l_j<l_k, k \neq j$. Then to the initial triple $(l_1, l_2, l_3) \in \Delta_2$ one associates a new triple $(l_1', l_2', l_3') \in \Delta_2$ by linear relations $l_k':=l_k-l_j$ for $k \neq j$ and $l_j'=l_j$ and subsequent rescaling. This algorithm is called a \textbf{fully subtractive algoritm}. The boundary $\partial \Delta_2$ is its set of fixed points, and the fully subtractive algorithm is not well defined when two (or more) of $l_j$ are equal, see the work \cite{AS13} of P. Arnoux and S. Starosta and Section \ref{sec:trop_cool} here for more details. 

Let $\mathcal{E} \subset \Delta_2$ be the set 

Let $\mathcal{E} \subset \Delta_2$ be the set of points $\rho_{\Delta}$ such that a corresponding triple of lengths $(l_1, l_2, l_3)$ is a pre-image of a point $\left(1/3, 1/3, 1/3\right)$ under some iteration of the fully subtractive algorithm. The correspondance is assured by the relations \eqref{eq:rho_delta} and \eqref{eq:correspondance}.

\begin{theorem}\label{thm:complete_classification}
For any triangle tiling billiard with angle parameters $\alpha, \beta, \gamma$,  the following holds:
\begin{enumerate}
\item[1.] if $\rho_{\Delta} \notin \mathcal{R} \cup \mathcal{E}$ then any trajectory on a corresponding tiling is either linearly escaping or periodic, and both behaviors are possible. Moreover, first, the list of words in the alphabet $\mathcal{A}_{\Delta}$ realized by periodic trajectories on such a tiling is finite;  second, there exist two functions $\omega_1, \omega_2: \Delta_2 \setminus  \mathcal{R} \cup \mathcal{E} \rightarrow \mathcal{A}_{\Delta}^{\N}$ such that the symbolic behaviour of any linearly escaping trajectory on the underlying tiling is an infinite concatenation two finite subwords $\omega_1(\rho_{\Delta})$ and $ \omega_2(\rho_{\Delta})$;
\item[2.] if $\rho \in \mathcal{R}$ then  any trajectory on a corresponding tiling escapes to infinity (is periodic) if and only if it passes (doesn't pass) through a circumcenter of a tile. Moreover, a list of symbolic codes of periodic trajectories is infinite (countable), as well as a corresponding list of trees; 
\item[3.] if $\rho \in \mathcal{E}$, then all the trajectories on a corresponding tiling are periodic;
\item[4.] drift-periodic trajectories exist on tilings for which $\rho \in \Q^3 \setminus \mathcal{E}_{\Delta}$ and only on them.
\end{enumerate}
\end{theorem}

The proof of this theorem uses both of the main tools that we introduce in this article - tiling billiard foliations (Section \ref{sec:foliations}) and renormalization for fully flipped $3$-interval exchange transformations (Section \ref{sec:trop_cool}).

This section presented some of our results from the point of view of tiling billiards. In the following section, we precise the connection between tiling billiards and fully flipped interval exchange transformations on the circle, and hence give another point of view on the study of the \emph{a priori }new object, triangle tiling billiards. This point of view is that of a study of parametric families of locally isometric maps, a classical topic in dynamics.

\section{Fully flipped interval exchange transformations on the circle.}\label{sec:fully_flipped_intro}

Fix $(l_1, \ldots, l_n) \in \Delta_n:=\left\{(l_1, \ldots, l_n) \in \R_+^n \left|\right. l_1+\ldots+l_n=1\right\}$. Define a family $\CETn$ of interval exchange transformations \emph{with flips} on the circle as follows. Cut the circle $\Sph^1$ of length $1$ into $n$ disjoint intervals $I_j$ of lengths $l_j, j=1, \ldots,n$. 

Define a map $F_0$ as a global involution of $\Sph^1$ which is a composition of $n$ (commuting) involutions on each one of $n$ intervals of continuity. We say that a map $F$ belongs to the family $\CETn$ if it is a composition $F=R_{\tau} \circ F_0$,  where $R_{\tau}$ is a rotation by an angle $\tau \in \Sph^1$. See Figure \ref{fig:fullyflipped} for an illustration.  The family $\CETn$ is a family of \textbf{fully flipped $n$-interval exchange transformations on the circle with trivial combinatorics.} In the following we often write $F=F_{\tau}^{l_1, \ldots, l_n}$ in order to stress the corresponding parameters. 

Note that the map $F=R_{1/2} \circ F_0$ is a composition of two non-commuting involutions.

\smallskip

For the family $\CETthree$ of maps acting on the circle $\Sph^1$ of unit length, we mark a point $0 \in\Sph^1$ as a beginning of the first interval of continuity. Then the three intervals of continuity are $I_1:=(0, l_1), I_2:=(l_1, l_1+l_2)$ and $I_3:=(l_1+l_1, l_1+l_2+l_3)$. We consider the bijection between the alphabets $\mathcal{N}_{\Delta}=\{1,2,3\}$ and  $\mathcal{A}_{\Delta}=\{a,b,c\}$ defined by the alphabetical order. This defines the symbolic dynamics for any map $F \in \CETthree$ with respect to the alphabet
$\mathcal{A}_{\Delta}=\{a,b,c\}$ in a standard way by associating to any point $p \in \Sph^1$ a sequence of labels in $\mathcal{A}_{\Delta}$ corresponding to the labels $j \in \mathcal{N}_{\Delta}$ of the intervals $I_j$ visited by its orbit $\{F^{\circ k}(p)\}_{k \in \N}$.

\smallskip
In this work we study in detail the dynamics (and the symbolic dynamics) of the family $\CETthree$. This dynamics reunites the dynamics of the Arnoux-Rauzy family, with that of rel deformations of Arnoux-Rauzy surfaces and of the triangle tiling billiards, as we show in the following paragraphs.

\subsection{Family $\CETn$ and tiling billiards.}\label{subs:CETandTriangles}

We defined the symbolic dynamics of the maps in $\CETthree$ with the help of the same alphabet $\mathcal{A}_{\Delta}$ 
as that for the dynamics of triangle tiling billiards. This notation is intentional: indeed, as has been proven in \cite{BDFI18}, the study of the dynamics of a tiling billiard in a triangle tiling defined by a tile with angles $\alpha, \beta, \gamma$ can be reduced to the study of the dynamics of a subfamily of maps 
\begin{equation*}
\left\{F^{l_1, l_2, l_3}_{\tau} \in \CETthree \left|\right. \tau \in [0,1]\right\}
\end{equation*}
with $(l_1, l_2, l_3)$ defined by \eqref{eq:correspondance}. The parameter $\tau$ corresponds to the position of a segment of the trajectory in the circumcircle of a tile it crosses, and it doesn't change along the trajectory by point 1. of Theorem \ref{thm:triangle_tiling_billiards_info}. Such a correpondance follows from the process of folding of a tiling along a trajectory of a tiling billiard that we descrive in Section \ref{sec:locally foldable tilings and associated billiard foliations}. In the "folded coordinates", a triangle moves while a direction of the trajectory doesn't change (modulo orientation) and is encoded by $\tau \in \Sph^1$ coordinates. 

Analogically to the case of triangle tiling billiards, the behavior of cyclic quadrilateral tiling billiards is completely described by the family $\CETfour$. For any cyclic quadrilateral and a trajectory of some parameter $\tau$, a corresponding map $ F\in \CETfour$ is defined by
the lengths $l_j, j\in \mathcal{N}_{\square}:=\{1,2,3,4\}$ corresponding to the angles $(\alpha_1, \alpha_2, \gamma_1, \gamma_2$ in which the diagonal of a tile splits the opposite angles of the quadrilateral, see Figure \ref{fig:angles of a quadrilateral} for the definition of these angles. Any cyclic quadrilateral is defined by the quadruple of angles $(\alpha_1, \alpha_2, \gamma_1, \gamma_2)$ up to homothety.\footnote{Indeed, each of the angles $\alpha_1, \alpha_2$, $\gamma_1, \gamma_2$ bounds an arc corresponding to the chords of length $b,c,d$ and $a$. In other words, a cyclic quadrilateral is defined by the lenghts of its sides up to homothety. Although the quadruple $(\alpha, \beta, \gamma, \delta)$ doesn't define uniquely the form of a cyclic quadrilateral: for example a hyperplane $\alpha=\beta=\gamma=\delta=\frac{\pi}{2}$ defines all the rectangles.}

\begin{figure}
\centering
\includegraphics[scale=0.5]{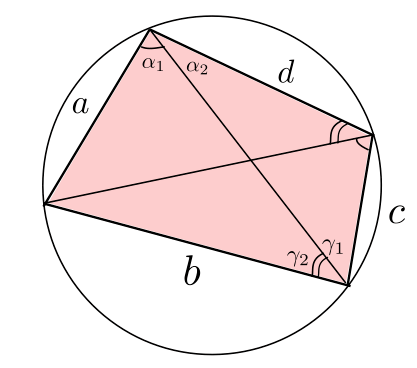}
\caption{\emph{Cyclic quadrilateral tile and the corresponding angle parameters.} The angles $\alpha_1, \alpha_2, \gamma_1, \gamma_2$ define the quadruple $(\alpha, \beta, \gamma, \delta)$ by the relations: $\alpha_1+\alpha_2=\alpha, \beta=\gamma_1+\alpha_2, \gamma=\gamma_1+\gamma_2, \delta=\alpha_1+\gamma_2$.  The  corresponding map is a map $F=F_{\tau}^{\frac{\alpha_1}{\pi}, \frac{\gamma_2}{\pi}, \frac{\gamma_1}{\pi}, \frac{\alpha_2}{\pi}} \in \CETfour$.}
\label{fig:angles of a quadrilateral}
\end{figure}

The \textbf{symbolic code} and the \textbf{accelerated symbolic code }for quadrilateral tiling billiard trajectories (and for maps in $\CETfour$) are defined analogically to the case of triangle tiling billiards.  The alphabets for the symbolic codes of trajectories in quadrilateral tiling billiards are $\mathcal{A}_{\square}:=\{a,b,c,d\}$ and $\mathcal{A}_{\square}^2:=\{ab,ac,ad,ba,bc,bd,ca,cb,cd,da,db,dc\}$ correspondingly.  

\bigskip

The take-away from this paragraph is that the study of symbolic dynamics of a map in $\CETn$ for $n=3$ and $4$ is equivalent to the  study of a related tiling billiard.

The question of symbolic dynamics in the family $\CETn$ is interesting in itself and can be studied for \emph{any} $n$. In this work, we concentrate on the case of the maps in $\CETthree$ simply because it is the only case that we were able to treat. See Section \ref{subs:Q} for the discussion of the family $\CETn$ for $n \geq 4$ and open questions. 

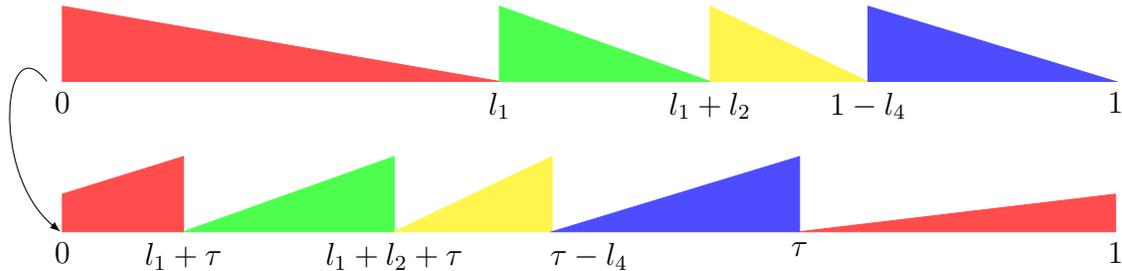
\begin{figure}
\centering
\begin{tikzpicture}[xscale=14]
\path[draw, fill=red,red, opacity=0.7] (0,0)--(0,1)--(0.415,0)--cycle;
\path[draw, fill=green, green, opacity=0.7] (.415,0)--(.415,1)--(0.615,0)--cycle;
\path[draw, fill=yellow, yellow, opacity=0.7] (.615,0)--(.615,1)--(0.765,0)--cycle;
\path[draw, fill=blue, blue, opacity=0.7] (.765,0)--(1,0)--(.765,1)--cycle;
\path[draw] (0.615,0) node[below]{$l_1+l_2$} --cycle;
\path[draw] (0.765,0) node[below]{$1-l_4$} --cycle;
\path[draw] (0,0) node[below] (1){$0$} --cycle;
\path[draw] (1,0) node[below]{$1$} --cycle;
\path[draw] (.415,0) node[below]{$l_1$} --cycle;
 \tikzset{
        arrow/.style={
            color=black,
            draw=black,
            -latex,
                font=\fontsize{12}{12}\selectfont},
        }
\path[draw] (0,-2) node[below] (2){$0$} --cycle;
\path[draw] (1,-2) node[below]{$1$} --cycle;
\path[draw, fill=red,red, opacity=0.7] (0,-2)--(0,0.5-2)--(.115,1-2)--(.115,0-2)--cycle;
\path[draw, fill=green, green, opacity=0.7] (.115,0-2)--(.315,1-2)--(0.315,0-2)--cycle;
\path[draw, fill=yellow, yellow, opacity=0.7] (.315,0-2)--(.465,1-2)--(0.465,0-2)--cycle;
\path[draw, fill=blue, blue, opacity=0.7] (.465,0-2)--(.7,1-2)--(.7,0-2)--cycle;
\path[draw] (0.7,0-2) node[below]{$\tau$} --cycle;
\path[draw, fill=red,red, opacity=0.7] (0.7,-2)--(1,-2)--(1,0.5-2)--cycle;
\path[draw] (.115,0-2) node[below]{$l_1+\tau$} --cycle;
\path[draw] (.315,0-2) node[below]{$l_1+l_2+\tau$} --cycle;
\path[draw] (.5,0-2) node[below]{$\tau-l_4$} --cycle;
\draw[arrow](1) to [out=94,in=95]  (2);
\end{tikzpicture}
\caption[]{\emph{A pictorial representation of a map }$F_{\tau}^{l_1, \ldots, l_4} \in \CETfour$. The shapes above the intervals are drawn in order to facilitate the understanding. This representation visualizes the action of the map $F: \Sph^1 \rightarrow \Sph^1$ on the circle and shows that the \emph{beginning} of each of the intervals $I_j$ maps to the \emph{end} of each interval $F(I_j), j\in \mathcal{A}_{\square}$. The idea of such a pictorial representation comes from [Figure 9, \cite{BDFI18}].}\label{fig:fullyflipped}
\end{figure}

\subsection{Arnoux-Rauzy family.}\label{subs:ARfamily}
By a classical Keane's Theorem proven in \cite{K75}, almost every $n$-interval exchange transformation (IET) with irreducible combinatorics is minimal. A very interesting question and generally not solved question is to study the minimality in the $k$\emph{-parametric families of }$n$\emph{-IET} for $k<n$.  Many recent works shead some light on the partial answers, see for example \cite{ST18, AHS16-1}.

\smallskip

One of the examples of parametric families for which the question of minimality has been explicitely solved is the so-called \textbf{Arnoux-Rauzy family} $\AR$ of $6$-IET on the circle of unit length, with the set of parameters being a $2$-simplex.

Take $(x_1, x_2, x_3) \in \Delta_2$. Then a map $T=T^{x_1,x_2,x_3} \in \AR$ is defined as follows. Cut the circle $\Sph^1$ into six disjoint intervals of lengths $\frac{x_j}{2}, j=1,2,3$ such that intervals of equal length are neighbouring. Then a map $T^{x_1,x_2,x_3} \in \AR$ is a composition of two involutions: first, a simultaneous exchange of intervals of equal length and second, the rotation $R_{\frac{1}{2}}$. The family $\AR$ was first defined and studied by P. Arnoux and G. Rauzy in \cite{AR91} and subsequently in \cite{ABB11, AS13, AD, BCS} and many other works.

\begin{example}
A map $T^{\aaa}:=T^{\aaa,\aaa^2,\aaa^3}$ with $\aaa \in \R$ such that 
\begin{equation}\label{eq:a}
\aaa+\aaa^2+\aaa^3=1,
\end{equation}
 is called the \textbf{Arnoux-Yoccoz map}. It was first introduced and studied in \cite{AY, A81}. This map is the simplest minimal map in the family $\AR$, and has many autosimilarity properties.
\end{example}

The family $\AR$ happens to be related to the Rauzy gasket $\mathcal{R}$. By a result in \cite{AR91}, the Rauzy gasket coincides with the set of parameters $(x_1, x_2, x_3)$ for which the maps $T^{x_1, x_2, x_3}$ are minimal. 

\begin{theorem}\cite{AR91}\label{thm:Arnoux-Rauzy}
A map in the Arnoux-Rauzy family is minimal, if and only if $(x_1, x_2, x_3) \in \mathcal{R}$.
\end{theorem}

The proof by P. Arnoux and G. Rauzy is based on a process of renormalization which is defined as a first return map on the union of two intervals of continuity of the biggest (and equal) length. In this work we give a new proof of this theorem by defining a renormalization process  on a family of the natural "square roots" of the maps in $\AR$ which happens to be a subfamily in $\CETthree$.

\begin{proposition}\label{prop:the_squares_are_Rauzy}\cite{Olga}
The following sets of $6$-IET on the unit circle coincide:
\begin{equation*}
\left\{ 
T^{x_1,x_2,x_3} \in \AR, (x_1,x_2,x_3) \in \Delta_2
\right\}=
\left\{ 
F^2 \mid \; F_{\frac{1}{2}}^{l_1, l_2, l_3} \in \CETthreehalf, (l_1, l_2, l_3) \in \Delta_2, \max(l_j)< \frac{1}{2}
\right\}.
\end{equation*}
Moreover, the correspondance between parameters is given by linear relations:
\begin{equation}\label{eq:l and x}
l_j=\frac{1-x_j}{2}, \; \; x_j=1-2l_j, \; \; j=1,2,3.
\end{equation}
\end{proposition}

This Theorem has been proven in \cite{Olga} by following the ideas in \cite{BDFI18}. In this work, we extend the equality in Proposition \ref{prop:the_squares_are_Rauzy} in  a way that the set on the right is enlarged to contain the maps for any $\tau$, and on the left the family  $\AR$ is enlarged to the family of its real-rel deformations. 

\subsection{Real-rel deformations of Arnoux-Rauzy maps.}\label{subs:rrdeformAR}
For any translation surface $X$, one can consider local deformations of $X$ in its stratum in such a way that the singularities are moving one with respect to another while keeping the translational holonomies of closed curves on $X$ fixed. This defines a \textbf{rel-foliation} in the stratum. The rel-foliations have been studied, among others, in \cite {Sch, McM, HW18} (under different terminologies). In the following we use the terminology from \cite{HW18}, so we refer our reader there for more details.  

In this work we are interested in a family of translation surfaces $X=X^{x_1,x_2,x_3}$ constructed  as suspensions of maps $T=T^{x_1,x_2,x_3}  \in \AR$ with $(x_1, x_2, x_3) \in \Delta_2$. We study the corresponding real-rel foliations constructed by variation of only horizontal holonomies. All of the surfaces $X^{x_1,x_2,x_3}$ belong to the stratum $\mathcal{H}(2,2)$, have genus $3$ and two singularities. Hence for a fixed point $(x_1, x_2, x_3) \in \Delta_2$, the real-rel leaf $\{X^{x_1,x_2,x_3}_r\}$ of the surface $X^{x_1,x_2,x_3}$ is parametrized by one real parameter $r \in \R$. Here $X^{x_1,x_2,x_3}_0=X^{x_1,x_2,x_3}$. Naturally, the surface $X_0$ is a double cover of a non-orientable surface  constructed as a suspension of a map in $\CETthreehalf$ by Proposition \ref{prop:the_squares_are_Rauzy}. Hence, its real-rel deformation corresponds through the first-return map to the subset of maps in $\CETthree$ and hence, to triangle tiling billiards. Moreover, $\tau=\frac{1}{2}-r$ and the parameters $(x_1,x_2,x_3)$ do not change on a real-rel leaf. This connection has already been noticed in \cite{BDFI18} for the Arnoux-Yoccoz map $T^{\boldsymbol{a}}$.

We are especially interested in the real-rel deformations of \emph{minimal} Arnoux-Rauzy maps and their symbolic dynamics. From the discussion above follows that the symbolic dynamics and arithmetic orbits of these maps are in direct correspondence with the dynamics of triangle tiling billiard trajectories. In particular, by describing the symbolic dynamics of maps in $\CETthree$ we manage to understand it for their squares, and hence to prove the fractal properties of arithmetic orbits of the Arnoux-Yoccoz map. In particular, we prove the following conjecture by P. Hooper and B. Weiss from their work \cite{HW18} where they studied real-rel deformations of the surface $X^{\boldsymbol{a}}$.

\begin{conjecture}\label{thm:convergence to the Rauzy fractal}
Any arithmetic orbit of the Arnoux-Yoccoz map $T^{\boldsymbol{a}}$ converges up to rescaling and uniform affine coordinate change to the Rauzy fractal\footnote{The Rauzy fractal is a famous classical fractal set that we define in paragraph \ref{subs:Arnoux-Yoccoz_Rauzy}.}  in the Hausdorff topology.
\end{conjecture}

We prove this conjecture in the following Theorem \ref{thm:big}. The idea of the proof is to first, replace an arithmetic orbit by an exceptional billiard trajectory in the tiling defined by ${\rho}_{\Delta}=(\aaa, \aaa^2, \aaa^3) \in \mathcal{R}$. Then, one approaches such a trajectory by a family of periodic trajectories with growing periods included in the same global foliation of the tiled plane. This construction is based on the periodicity of vertical flows for any surface $X^{\aaa}_r$ in a real-rel leaf of $X^{\aaa}$ for $r \neq 0$.\footnote{This periodicity has been proven in \cite{HW18} but in this work we reprove it with the use of tiling billiard foliations.} The periods of growing periodic trajectories are calculated via renormalization and coincide with the set of doubled Tribonacci numbers.

\bigskip

In addition to its arithmetic orbits, a few other fractal objects may be associated to the map $T^{\boldsymbol{a}}$. Initially, P. Arnoux in \cite{A88} constructed a semi-conjugacy $h$ between the map $T^{\boldsymbol{a}}$ and a translation on the torus with a translation verctor equal to $(\boldsymbol{a},\boldsymbol{a}^2)$. A curve defined as $h(\Sph^1)$ is a Peano curve on the torus which can be approximated by a sequence of piecewise linear curves (since the map $h$ maps the  $T^{\boldsymbol{a}}$-orbit of $\frac{1}{2}$ to the orbit of $0$ under the translation on the torus).

Moreover, to the Arnoux-Yoccoz map $T^{\boldsymbol{a}}$ one can also associate its algebraic dynamics: for any $p \in \Q[\boldsymbol{a},\boldsymbol{a}^2]$ its image $F^{\boldsymbol{a}}(p) \in \Q[\boldsymbol{a},\boldsymbol{a}^2]$. The field $\Q[\boldsymbol{a},\boldsymbol{a}^2]$ can be seen as a three-dimensional vector space with basis $1,\boldsymbol{a},\boldsymbol{a}^2$. For any point $p$ one draws a piece-wise linear curve connecting the subsequent points in its orbit. It happens that such a curve is contained in a small slice of space between two parallel planes. By projecting it on one of these planes, for a typical point $p \in \Q[\boldsymbol{a},\boldsymbol{a}^2]$, one obtains a fractal curve, see [Figure 5 in \cite{M12}] for its representation by C. McMullen. In \cite{LPV07}, J. Lowenstein, F. Poggiaspala and F. Vivaldi study the density properties of such a curve.
For more details, see \cite{LPV07} and \cite{M12}. 

The algebraic Peano curve (associated to the work of McMullen and Lowenstein-Poggiaspala-Vivaldi) converges, up to reparametrization, to the Peano curve constructed by Arnoux which in its turn converges to the Rauzy fractal (up to rescaling), as proven in \cite{ABB11} by P. Arnoux, J. Bernat and X. Bressaud.

Even though the arithmetic orbit of the Arnoux-Yoccoz map is not exactly the same as the algebraic Peano curve studied in \cite{M12} and {LPV07}, they converge one to another after rescaling, as follows from the results in \cite{ABB11} and our results in this article, see Theorems \ref{thm:exceptional_trajectories_intro} and \ref{thm:convergence to the Rauzy fractal}. Even more, we think that 
possibly by using the results of the work \cite{ABB11} connecting algebraic orbits and the Rauzy fractal, and the results of this work connecting the Rauzy fractal with the arithmetic orbits (see paragraph \ref{subs:Arnoux-Yoccoz_Rauzy} in the following), one could possibly prove the stronger density results for algebraic Peano curves associated to the Arnoux-Yoccoz map than those proven in \cite{LPV07}. We hope to provide the formalization of these connections in our future work.

Of course, the Arnoux-Yoccoz map is the first and simplest example of a minimal map in the family $\AR$. Although, all of the points in the Rauzy gasket that give rise to the translation surfaces  $X^{x_1,x_2,x_3}$ admitting a pseudo-Anosov map are interesting. The questions are many for each of these surfaces: what fractal curves arise as arithmetic orbits? (as algebraic orbits?)  what are possible dilatation coefficients of corresponding pseudo-Anosov maps? Our methods can simply be generalized for the periodic points of the Rauzy subtractive algorithm. For other points in $\mathcal{R}$, additional work has to be done. 

\section{Plan of the article.}\label{sec:plan}
The dynamics of the maps in the family $\CETthree$ and the dynamics of triangle tiling billiards are closely related, and can be seen as the same dynamical system. We have split this work into two parts, each of which gives different tools to study this same system. At the end of the second part, we reunite these tools. We now give a more detailed plan.

\bigskip

In the first part of this work, we present a geometric approach to tiling billiards via folding and foliations. In Section \ref{sec:locally foldable tilings and associated billiard foliations} we remind the standard folding argument and generalize it. In Section \ref{sec:foliations} we define and study tiling billiard foliations. In Section \ref{sec:proof of the tree conjecture} we prove the Tree Conjecture for triangle tiling billiards.

In the second part of the work, we study the symbolic dynamics of the family $\CETthree$ and its subfamily, the Arnoux-Rauzy family. In Section \ref{sec:arithmetic_orbits}, we precise the connection between the arithmetic orbits of the real-rel leaves of Arnoux-Rauzy surfaces and triangle tiling billiards that we touched on in paragraph \ref{subs:rrdeformAR}. In Section \ref{sec:trop_cool}, central to the second part of the work, we introduce the renormalization process on $\CETthree$
and use it in order to characterize the symbolic dynamics and prove minimality results. By reuniting the geometric and combonatorial approaches to triangle tiling billiards, in Section \ref{sec:tiling billiards} we finally give a complete classification of their trajectories. In Section \ref{sec:arithmetic_orbits_and_exceptional_trajectories} we study the exceptional trajectories of triangle tiling billiards and prove the density results. In particular, we show the convergence of arithmetic orbits of the Arnoux-Yoccoz map to the Rauzy fractal.

Finally, in the third part of this work, i.e. in Section \ref{subs:Q}, we state the open questions, with a focus on the cyclic quadrilateral tiling billiards. 

In the Appendix, we give several comments on our previous work \cite{Olga} with P. Hubert concerning tiling billiards.

\bigskip

Our methods are elementary, no prerequisits are needed to understand the proofs.
\newpage

\begin{center}
{\textbf{ Part I.-- On a proof of the Tree Conjecture for triangle tiling billiards}}
\end{center}

In this part, we introduce tiling billiard foliations and \emph{flowers} (unions of singular leaves) in these foliations. In a nutshell, the main message of this part is the following. The symbolic dynamics of every periodic trajectory is defined by the symbolic dynamics of a sequence of flowers on which it is contracted. This gives a strategy of the proof of the Tree Conjecture. In the following three sections, we give the necessary definitions and arguments to realize this strategy.

\section{Folding in triangle and cyclic quadrilateral tiling billiards.}
\label{sec:locally foldable tilings and associated billiard foliations}
Tiling billiards on triangle and cyclic quadrilateral tilings have unusual (for generic tiling billiards) rigidity properties. These properties are explained by the folding construction  which has a central place in this work.

\bigskip
 
Folding for triangle tiling billiards was proposed in \cite{BDFI18}. In this Section we present their construction, although our proof is more general and doesn't use in an explicit way the structure of the tiling (triangle or quadrilateral). 

\begin{lemma}[\cite{BDFI18}]\label{lemma:folding}
Consider a periodic triangle (cyclic quadrilateral) tiling and some tile $\theta_0$ in it. Let $\Lambda=(V,E)$ be a corresponding graph ($\Lambda=\Lambda_{\Delta}$ or $\Lambda_{\square}$). Then there exists a unique map $\mathcal{F}=\mathcal{F}(\theta_0): \R^2 \rightarrow \mathcal{F}(\R^2) \subset \R^2$ such that
\begin{enumerate}
\item[1.] for any tile $\theta$ the restriction $\mathcal{F}|_{\theta}$ is an isometry and $\mathcal{F}|_{\theta_0}=\mathrm{id}$; 
\item[2.] for any two  tiles $\theta$ and $\theta$ sharing an edge $e \in E$ their images $\mathcal{F}(\theta)$ and $\mathcal{F}(\theta)^e$ are symmetric one to each other with respect to a line bisector of $\mathcal{F}(e)$,
\item[3.] two different folding maps of the same tiling (with different $\theta_0$) differ by a global isometry of $\R^2$.
\end{enumerate}
Moreover, $\mathcal{F}(V) \subset \mathcal{C}$, where $\mathcal{C}$ is a circumcircle of $\theta_0$. 
\end{lemma}

\begin{proof}
For any tile $\theta$, we construct its image $\mathcal{F}(\theta)$ as follows. Take a \textbf{sequence} of tiles $\theta_0,\theta_1, \ldots, \theta_n=\theta$ \textbf{connecting} $\theta_0$ to $\theta$: the tiles $\theta_k$ and $\theta_{k+1}$ share an edge. Then, fold the union $\theta_1\cup\ldots \cup \theta_n$ by a global isometry on $\theta_0$. This defines $\mathcal{F}(\theta_1)$. Then, we fold $\theta_k \cup \ldots \cup \theta_n$ on $\theta_{k-1}$ for $k=2, \ldots, n$. At the end of the process, one defines $\mathcal{F}(\theta)$ with $\mathcal{F}|_{\theta}$ an isometry.

It is left to prove that $\mathcal{F}(\theta)$ doesn't depend on the connecting sequence $\{\theta_k\}$, or equivalently, $\mathcal{F}(\theta_0)=\theta_0$ for any connecting \textbf{loop} ($\theta_0=\theta_N$). First, when one folds one polygon on another in a tour around a vertex, the difference between the angles of positively and negatively oriented tiles in the vertex defines the displacement of the initial tile $\theta_0$ with respect to its initial position. Since this difference is zero (see Figure \ref{fig:triangle_tiling_and_cyclic_quadrilateral_tiling}), $\mathcal{F}\mid_{\theta_0}=\mathrm{id}$. By breaking any loop into a sum of loops around vertices, one finishes the proof. Clearly, two folding maps differ by an isometry.

Let us now prove that $\mathcal{F}(V) \subset \mathcal{C}$. Indeed, $\mathcal{F}(v) \in \mathcal{C}$ obviously for the vertices of $\theta_0$, and by folding for all the vertices of the tiles sharing an edge with $\theta_0$ (see Figure \ref{fig:picture_for_circle_and_triangle_tb}). Hence, $\mathcal{F}(v) \in \mathcal{C}$ for any $v \in V$ by recurrence. 
\end{proof}

We call the map $\mathcal{F}$ a \textbf{folding map}, or simply, a \textbf{folding}. We call the image of the plane by a folding map a \textbf{bellow}, $\mathcal{B}:=\mathcal{F}(R^2)$. A name \emph{bellow} comes from accordeon bellows.

\begin{figure}
\centering
\includegraphics[scale=0.4]{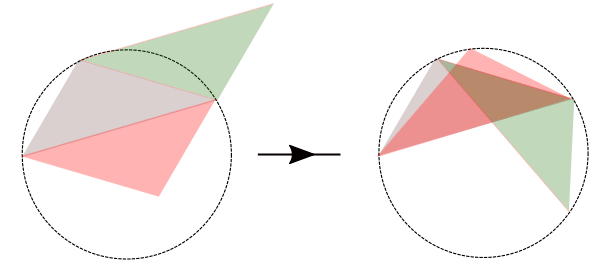}
\includegraphics[scale=0.4]{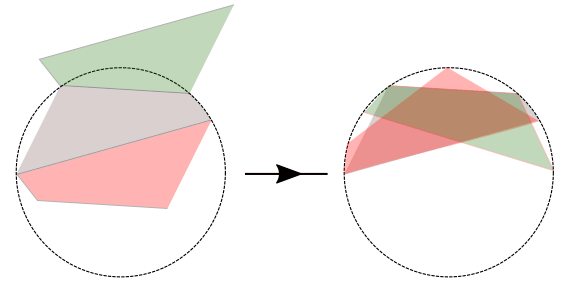}
\caption{Folding on a circle for a patch of a triangle (and cyclic quadrilateral) tiling. A tile $\theta_0$ maps to itself, and the other tiles map inside its circumcircle $\mathcal{C}$ under the folding map $\mathcal{F}(\theta_0)$.}
\label{fig:picture_for_circle_and_triangle_tb}
\end{figure}

\begin{remarkimp}\label{rem:locally foldable}
The two tilings we study in this work belong to a much bigger class of tilings, the so-called locally foldable tilings, for which the statements 1.-3. of Lemma \ref{lemma:folding} directly apply. The \textbf{locally foldable tiling} is a two-colorable tiling with the equilibrium of the
angles preserved at every vertex, i.e. the sum of the angles of tiles of one color around any vertex is equal to $\pi$. This class has been known for centuries in the origami community, and it also appears in the discrete complex analysis for the dimer model. The arguments of Lemma \ref{lemma:folding} are not new, and are used for example in  \cite{Hull, A, dimandcirc, Dima} in different contexts.  In this paper we concentrate ourselves on triangle and quadrilateral tilings. We hope to develop the general theory of tiling billiards in locally foldable tilings in the future.\footnote{Although some properties of tiling billiards in locally foldable tilings are already clear. Indeed, the Tree Conjecture, as well as the the Flower and the  Bounded Flower Conjectures in paragraph \ref{subs:FLOW} can be restated for locally foldable tiling billiards and associated foliations. Moreover, Theorem \ref{thm:cyclic_quadrilateral_tiling_billiards_info} and Proposition \ref{prop:treetoflower} can be proven in this more general context.}
\end{remarkimp}

\subsection{Basic orbit properties in triangle and cyclic quadrilateral tiling billiards.}\label{subsub:basicQ}
In this paragraph we copy the proof from \cite{BDFI18} for triangle tiling billiards, in order to apply it to the case of cyclic quadrilateral tilings.  This result was announced without an explicit proof in \cite{Olga}. The main idea is that a tiling billiard trajectory folds into a subset of a \emph{straight chord} in $\mathcal{C}$. The proof is written in a way to apply to any locally foldable tiling.

\begin{theorem}\label{thm:cyclic_quadrilateral_tiling_billiards_info}
The points 1.-3. of Theorem \ref{thm:triangle_tiling_billiards_info} hold for any cyclic quadrilateral tiling (locally foldable tiling).
\end{theorem}
\begin{proof}
Consider a trajectory $\delta$ of a tiling billiard starting in some tile $\theta_0$, and a folding map $\mathcal{F}=\mathcal{F}(\theta_0)$. Then $\mathcal{F}(\delta)$ is a subset of a segment in the bellow $\mathcal{B}$ given by the intersection of  $\mathcal{B}$ with some line $l$.

Hence for any tile $\theta$ the intersection $l \cap \mathcal{F}(\theta)$ is equal to at most one segment. If $\delta$ is bounded then at some moment $\delta$ comes back to the same tile, and hence $\delta$ is periodic.

A periodic trajectory $\delta$ can't intersect itself in a transverse way inside a tile $\theta$, since it intersects this tile in a \emph{segment} equal to $\mathcal{F}^{-1}\left(\mathcal{F}(\theta) \cap l\right)$. 

Finally, a periodic trajectory is stable under a small enough perturbation since a sequence of tiles crossed by its perturbation $\delta'$ is the same as that for $\delta$. Hence this sequence is a loop, and $\delta'$ is periodic with the same symbolic dynamics as that of $\delta$.
\end{proof}

\begin{remark}
In the context of Hamiltonian dynamics, Arnold-Lioville integrability implies the existence of additional integrals of motion, or the laws of preservation of energy. For tiling billiards we consider here, the direction of a trajectory in "folded coordinates" is a first integral of the system. The folding map reduces the dimension of the phase space, and the dynamics on the plane is reduced to the dynamics on the circle, that of the family $\CETn$ of fully flipped maps on the circle (see paragraph \ref{subs:CETandTriangles}) for $n=3$ and $4$.
\end{remark}

\section{Tiling billiard foliations.}\label{sec:foliations}
Any tiling billiard trajectory may be folded into a segment of a line in the bellow. We now do an inverse procedure. Fix some tile $\theta_0$ and the folding $\mathcal{F}(\theta_0)$. Denote also by $\mathcal{D}$ a disk bounded by $\mathcal{C}$. Slice up the disk $\mathcal{D}$ in a union of non-intersecting segments by either a family of\emph{ parallel} chords, or a family of chords \emph{emanating from one point }on the boundary of $\mathcal{C}$. Finally, pull this slicing back to the tiled plane by $\mathcal{F}^{-1}$.

This defines two families of \emph{foliations} on the plane, and tiling billiard trajectories can be included in these naturally defined foliations. 

\subsection{What happens when a trajectory hits a corner of a tile?}\label{subs:separatrices}
In a classic setting of a billiard in a bounded domain with piecewise smooth boundary, a billiard trajectory that arrives to a non-regular point on the boundary, stops (or is not well defined). In the context of tiling billiards, as in that of geodesic flows on flat surfaces, one can correctly define, although possibly branching, singular trajectories as boundaries of cylinders of parallel trajectories.

\bigskip

A piece-wise linear simple curve $\gamma$ on the tiled plane that passes through at least one vertex of a tiling is called a \textbf{singular tiling billiard trajectory}, if the Snell's refraction law with coefficient $k=-1$ holds in all non-regular points of such a trajectory (even if a non-regular point is a vertex $v \in V$ of a tiling). We call the segment $\theta \cap \gamma$ of a singular trajectory $\gamma$ in the tile $\theta$ a \textbf{separatrix segment} if $\gamma \cap \theta \ cap V \neq \emptyset$, i. e. $\gamma$ passes by a vertex of $\theta$. If a singular trajectory is a closed curve, we call it a \textbf{separatrix loop}. 

Consider a singular trajectory $\gamma$ with at least one singular point $v \in V$. One associates to it a finite number of singular trajectories passing by $v$, via folding. Indeed, $\gamma$ folds into some chord $l$ in the disk $\mathcal{D}$ such that $l \cap \mathcal{C}=\mathcal{F}(v)$. One considers the connected components of the set $\mathcal{F}^{-1}(l \cap \mathcal{D}) \setminus \{v\}$ such that their intersection with $\cup_{\theta: \theta \ni v} \theta$ is non-empty. These connected components (eventually united with a point $\{v\}$) are exactly the separatrix curves passing by $v$ that fold into the same chord as $\gamma$.
 
We call the union of \emph{all} separatrices passing by a fixed vertex $v \in V$ and mapping to the same chord under folding, a \textbf{flower} in $v$. We call each of the separatrix loops in one flower a \textbf{petal} of this flower.  We call $v \in V$ a \textbf{pistil}. A flower is \textbf{bounded} if all of its separatrices are petals. To any line $l$ that cuts out a non-empty chord in $\mathcal{D}$ and passes by $\mathcal{F}(v)$, one may associate a flower.

 \smallskip
As we show in the following, the symbolic dynamics of any trajectory can be described in terms of dynamics of singular trajectories on which it is contracted in the parallel tiling billiard foliation that we define right away.

\subsection{Parallel and ray tiling billiard foliations.}\label{subs:parrad}

We call a foliation of a plane with a tiling a \textbf{tiling billiard foliation} if it is an oriented foliation with all of its leaves being tiling billiard trajectories.  We define two tiling billiard foliations for triangle and  quadrilateral tilings as preimages of two sheaves of lines on the plane containing the bellow. 

\smallskip

Take a tiling, fixe some base tile $\theta_0$ and a corresponding folding map $\mathcal{F}$.

Then for any $\tau \in \Sph^1$ consider a foliation of the plane by parallel lines with a common direction $\exp(i \tau)$. One considers the intersections of the leaves of this foliation with the bellow $\mathcal{B}$. Then, by applying $\mathcal{F}^{-1}$ to these intersections, one obtains a \textbf{parallel foliation} $\mathcal{P}_{\tau}$ (or simply, $\mathcal{P}$) of the plane with a tiling.

Now, take a point $p \in \mathcal{C}$. Consider all the chords in $\mathcal{D}$ passing by $p$, slicing up the bellow $\mathcal{B}$. By unfolding these slices back to the plane with a tiling one obtains the \textbf{ray foliation} $\mathcal{R}_p$ (or simply, $\mathcal{R}$). The set $\mathcal{F}^{-1}(p)$ is non-empty if and only if $p=\mathcal{F}(v)$ for some $v \in V$. Moreover, if the angles of tiles are ratinonally independent, in this case $\mathcal{F}^{-1}(p)=\{v\}$. In this work we restrict the class of ray foliations to those with $p \in \mathcal{F}(V) \subset \Sph^{1}$.

\begin{example}
On Figure \ref{fig:squaretiling_foliations} we give an example of a (very symmetric) square tiling with \emph{periodic} parallel and ray foliations. Although in general these two foliations are not periodic but \emph{quasiperiodic}. Moreover, for the square tiling, the set $V$ of vertices maps to a \emph{finite} subset of a circle $\mathcal{C}$ (consisting of four points) which is also non-generic for triangle and cyclic quadrilateral tiling billiards. Indeed, generically, the set $\mathcal{F}(V)$ is a dense subset of the circle.
\end{example}

\begin{figure}
\centering
\includegraphics[scale=0.15]{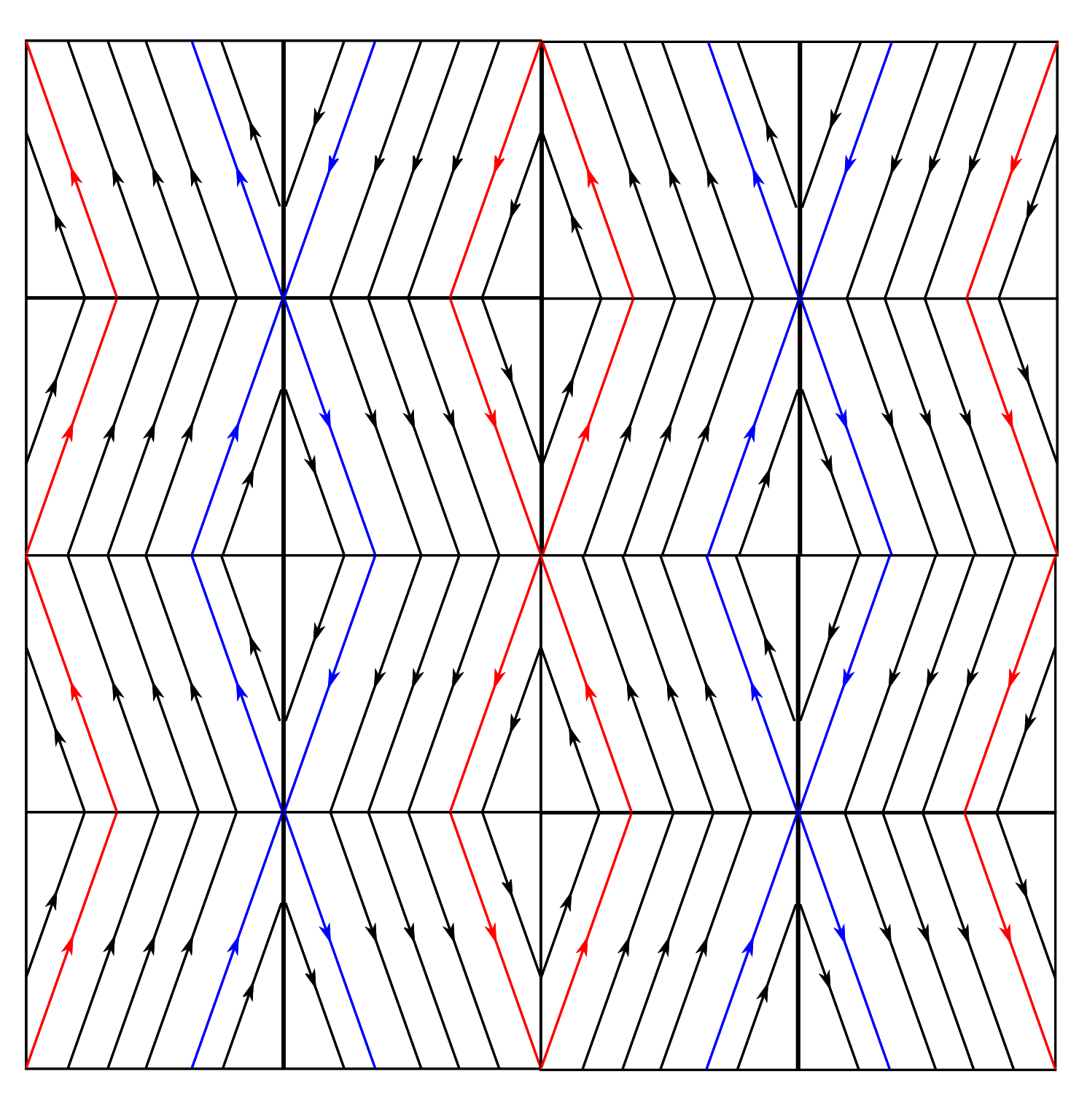}
\includegraphics[scale=0.15]{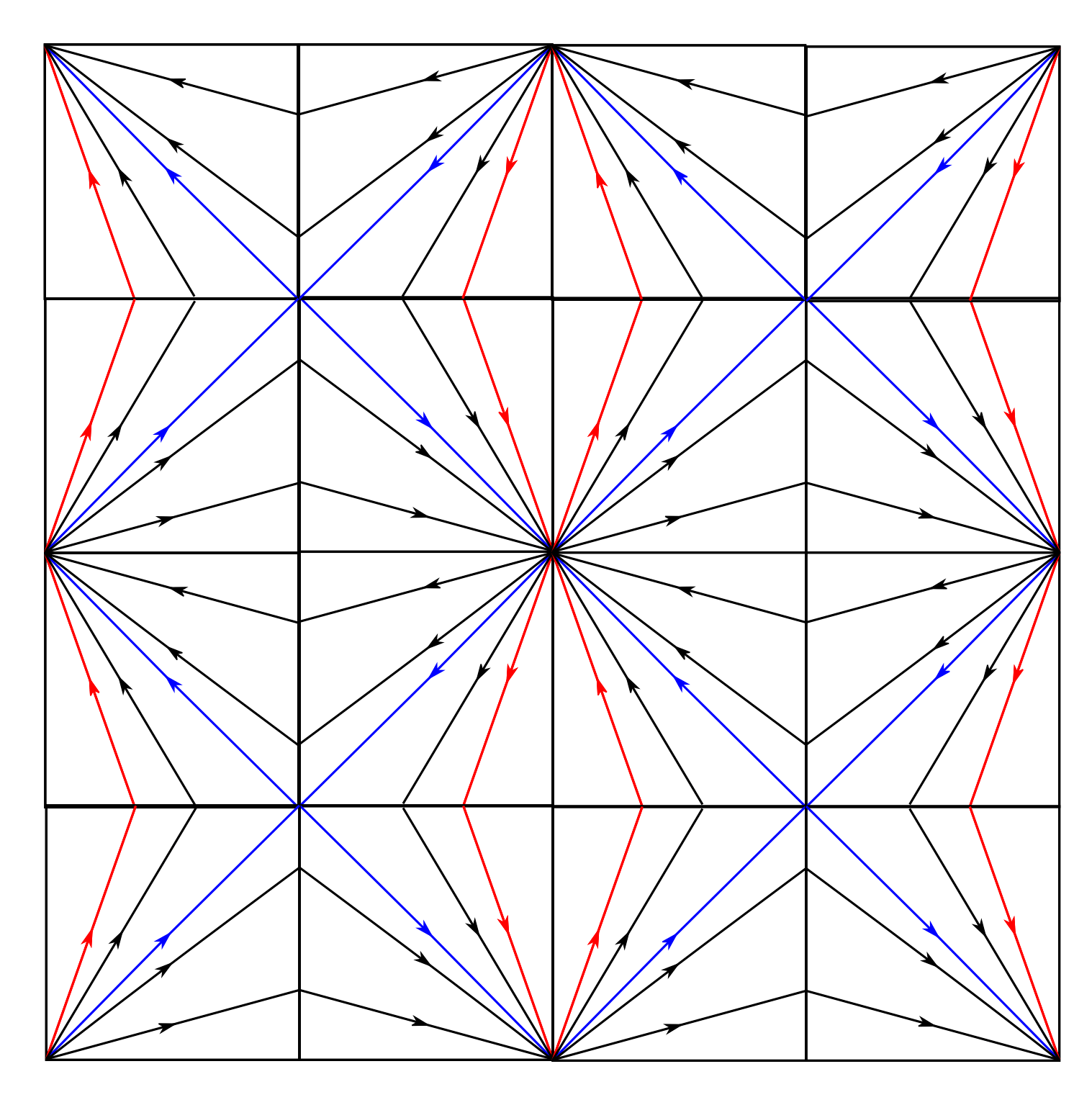}
\caption{\emph{Parallel and radial foliations in a square tiling.} On the left the parallel foliation $\mathcal{P}_{\tau}$ for some $\tau \in  \Sph^1$ and on the right the radial foliation $\mathcal{R}_p, p=\mathcal{F}(v), v\in V$. Every flower in either foliation consists of $0$ or $2$ petals. Red trajectories are contained in both $\mathcal{P}_{\tau}$ and $\mathcal{R}_p$. }
\label{fig:squaretiling_foliations}
\end{figure}

\begin{lemma}\label{lemma:use_me_on parallel and radial}
Fix some tile $\theta_0$ in a triangle (cyclic quadrilateral) tiling. Take any $\tau \in \Sph^1$ and $p \in \mathcal{C} \simeq \Sph^1$ such that $p=\mathcal{F}(v)$ for some $v \in V$. Then, the parallel and ray foliations $\mathcal{P}_{\tau}$ and $\mathcal{R}_p$ verify the following properties:
\begin{itemize}
\item[1.] the  foliations $\mathcal{P}_{\tau}$ and $\mathcal{R}_p$ are well defined and orientable. Moreover, their oriented connected leaves define tiling billiard trajectories;
\item[2.] the set of singularities of each of these foliations coincides with the set $V$;
\item[3.] for any $v \in \mathcal{F}^{-1}(p)$, there exists a finite number of singular leaves in $\mathcal{P}_{\tau}$ passing by $v$, at most one by each tile $\theta$ such that $v \in \theta$. Conversely, two 
separatrices in $\mathcal{P}_{\tau}$ passing by $v$ belong to the same flower;
\item[4.] take any (possibly singular, not necessarily periodic) trajectory $\delta$. Then there exists a unique $\tau$ such that $\delta$ is a leaf of $\mathcal{P}_{\tau}$. We denote this foliation $\mathcal{P}^{\delta}$. If under folding $\delta$ folds into a chord $l$ that intersects $\mathcal{F}(V)$, then $\delta$ can be included in a radial foliation $\mathcal{R}_p$ for each (of at most two) $p \in \mathcal{F}(V) \cap l$. We denote such a foliation $\mathcal{R}^{\delta}$;
\item[5.] for any periodic trajectory $\delta$ its interior $\Omega^{\delta}$ is foliated by the leaves of $\mathcal{P}^{\delta}$ (and of $\mathcal{R}^{\delta}$, if it exists).
\end{itemize}
\end{lemma}
\begin{proof}
This follows from Lemma \ref{lemma:folding} and Theorems \ref{thm:triangle_tiling_billiards_info} (points 1.-3.) and \ref{thm:cyclic_quadrilateral_tiling_billiards_info}. If a tile $\theta_0$ is positively oriented, then the orientation of $\mathcal{R}_p$ and $\mathcal{P}_{\tau}$ coincides with (is opposite to) the orientation of sheaves of lines on the bellow on positively oriented triangles.
\end{proof}

\subsection{Local behavior of separatrices of triangle tiling billiards.}\label{subs:classify}
In the following, we describe the possible combinatorics of local behavior of separatrix segments in flowers for triangle tiling billiards.

\begin{proposition}\label{prop:list of possible local behaviours PTT}
Fix  some $\tau \in \Sph^1$, a vertex $v \in V$ and a tile $\theta_0 \ni v$.  This defines a flower $\gamma$ in $\mathcal{P}_{\tau}$ with a pistil in $v \in V$ in a triangle tiling. Denote the number of its  separatrix segments containing $v$ by $s$.

Then $s \in \{0,2,4,6\}$ and each tile $\theta$, $\theta \ni v$ contains at most one separatrix segment of $\gamma$. Moreover, up to a possible change of orientation $\tau \mapsto -\tau$, in the restriction to the union $\Theta_v:=\cup_{\theta \ni v} \theta$ , the flower $\gamma$ has one of the combinatorial behaviors  represented on Figure \ref{fig:possible_local_behavior of separatrices in PTT}. 
\end{proposition}

\begin{proof}
Finiteness of $s$ follows from the point 3. in Lemma \ref{lemma:use_me_on parallel and radial}, and $s$ is even since the foliation $\mathcal{P}_{\tau}$ is oriented.

The separatrices passing by $v$ are leaves of both $\mathcal{R}_{p}$ and $\mathcal{P}_{\tau}$. Moreover, the ray foliation $\mathcal{R}_{p}$ is orientable and has a very simple form in restriction to the union $\Theta_v$ of six tiles containing $v$. Indeed, all of its leaves pass by $v$  and their directions alternate from one tile to its neighbor, see Figure \ref{fig:alternation}.
\end{proof}

\begin{remarkimp}
The list given in Proposition \ref{prop:list of possible local behaviours PTT} is realizable: one can find examples of triangle tiling billiard foliations $\mathcal{P}_{\tau}$ (by choosing the forms of tiles and the directions $\tau$) and flowers in them with all of the listed local behaviors. In this work we also classify possible global topological behaviors, e.g. we prove that the first case for $s=2$ on Figure \ref{fig:possible_local_behavior of separatrices in PTT} is not realizable by separatrix loops but only by unbounded separatrices, see Proposition \ref{prop:UNBOUNDED}.
\end{remarkimp}

The statement analogous to that of Proposition \ref{prop:list of possible local behaviours PTT} can also be proven for quadrilateral tilings: in this case $s \in \{0,2,4\}$ and for each value of $s$ the only one combinatorial distribution of tiles intersecting the flower $\gamma$ is possible (for $s=2$, these tiles are neighbouring). Although, the combinatorics of sides crossed by the separatrix segments is richer than in the case of triangle billiards. We do not discuss this issue in more detail, the case of cyclic quadrilaterals is still quiet mysterious for us.

\begin{figure}
\includegraphics[scale=0.1]{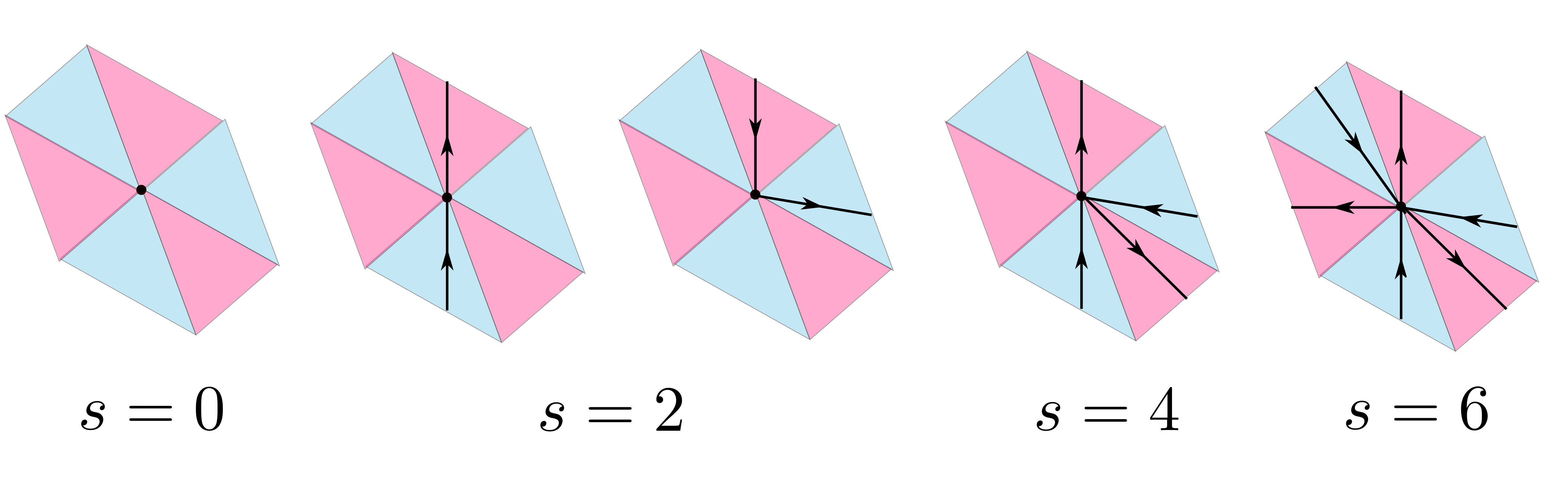}
\caption{\emph{All possible local combinatorial behaviors of restrictions} $\gamma \cap \Theta_v$\emph{ of a flower} $\gamma \in \mathcal{P}_{\tau}$ \emph{on the union} $\Theta_v$ \emph{of six tiles containing }$v$. For $s=0$ the only possible behavior is trivial. For $s=2$ two behaviors are possible. For $s=4$ and $6$ only one combinatorial behavior is possible. This Figure contains the information on the number of separatrix segments and their \emph{relative} positions.
}
\label{fig:possible_local_behavior of separatrices in PTT}
\end{figure}

\begin{figure}
\centering
\includegraphics[scale=0.1]{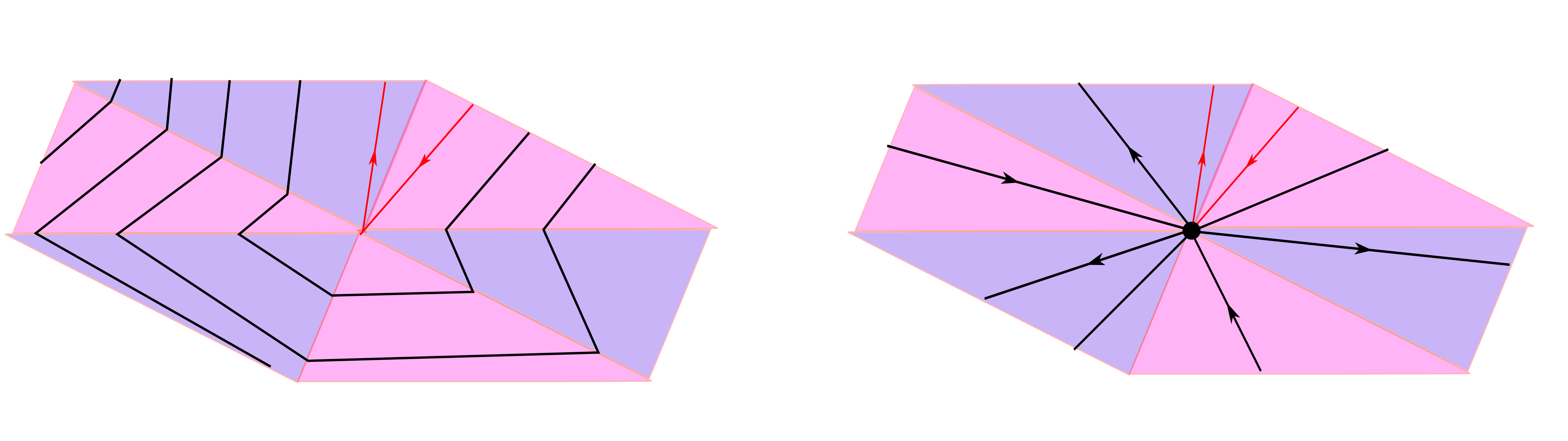}
\caption{An example of possible restriction of the periodic foliation $\mathcal{P}_{\tau}$ and a ray folation $\mathcal{R}_{p}$ to the set $\Theta_v$ with $v \in \mathcal{F}^{-1}(p)$. In red are given their common separatrix segments.}
\label{fig:alternation}
\end{figure}

\section{Tree conjecture for triangle tiling billiards.}\label{sec:proof of the tree conjecture}
In this Section, we prove the Tree Conjecture for triangle tiling billiards. 

\subsection{Reducing the Tree Conjecture to the Bounded Flower Conjecture.}\label{subs:FLOW}
The \emph{Tree Conjecture} is a statement about the global symbolic behavior of periodic trajectories. We reduce it to the \emph{Bounded Flower Conjecture} which is a local statement about the topology of separatrix loops in {one vertex}.

\smallskip

We introduce some notations. We say that two tiles are \textbf{neighbouring in} $e$ if they share an edge $e$.
Additionaly, and only for triangle tilings, we say that two tiles are \textbf{opposite in a vertex} $v$ if they both pass by $v$ and are centrally symmetric to each other with respect to $v$. For any tile $\theta_0$ such that $e \subset \theta_0, v \in \theta_0$ with $e \in E, v\in V$ we denote by $\theta_0^e$ its neightbouring tile in $e$, and by $\theta_0^v$ its opposite tile in $v$, see Figure \ref{fig:flower_notation}.

\smallskip

We say that the \textbf{Flower Conjecture holds for a separatrix loop} $\gamma$ such that for any $v \in \gamma \cap V$, there exists $e \in E$ such that $e \ni v$, $\gamma$ passes by $\theta$ and $\theta^e$ and $e \in \Omega^{\gamma}$. In other words, a separatrix loop in $v$ has to pass by two neighbouring tiles and to contain their common edge in its interior, see Figure \ref{fig:flower_notation}.

We say that the \textbf{Flower conjecture holds for a tiling} if it holds for all the separatrix loops of a tiling billiard on this tiling. We say that the \textbf{Bounded Flower Conjecture holds for a tiling} if the previous property is verified by all the petals of all the \emph{bounded} flowers in parallel foliations.

\begin{itemize}
\item[•] Obviously, the Flower Conjecture implies the Bounded Flower Conjecture.
\item[•] The Flower Conjecture implies that two separatrix loops $\gamma_1$ and $\gamma_2$ in a vertex belonging to the same parallel foliation $\mathcal{P}_{\tau}$ have the \emph{same} index with respect to infinity. In other words, the corresponding open domains  $\mathring{\Omega}^{\gamma_1}$ and $\mathring{\Omega}^{\gamma_2}$ are disjoint.
\item[•] The Flower Conjecture for periodic triangle tiling excludes a separatrix loop passing by two opposite triangles, as well as a separatrix loop passing by neighbouring triangles but not contouring an edge between them. These two topological configurations (that the Flower Conjecture excludes) are represented on Figure \ref{fig:OTS2}.
\end{itemize}

\begin{theorem}\label{thm:bounded_flower_conjecture}
The Bounded Flower Conjecture holds periodic triangle tilings.
\end{theorem}

\begin{theorem}\label{thm:flower_conjecture_holds}
The Flower Conjecture holds for periodic triangle tilings.
\end{theorem}

Theorem \ref{thm:bounded_flower_conjecture} together with Proposition \ref{prop:list of possible local behaviours PTT} give four possible topological forms of bounded flowers with the number of petals in the range from $0$ to $3$, see Figure \ref{fig:petals}. A singular point with no petals is also considered a flower, even though in real life such flowers are a little sad.

\bigskip

We postpone the proof of the Theorem \ref{thm:flower_conjecture_holds} to Section \ref{subs:Q}. We prove the Theorem \ref{thm:bounded_flower_conjecture} in this section. Finally, for the proof of the Tree Conjecture, it suffices to reduce it to Theorem \ref{thm:bounded_flower_conjecture}.

\begin{figure}
\centering
\includegraphics[scale=0.34]{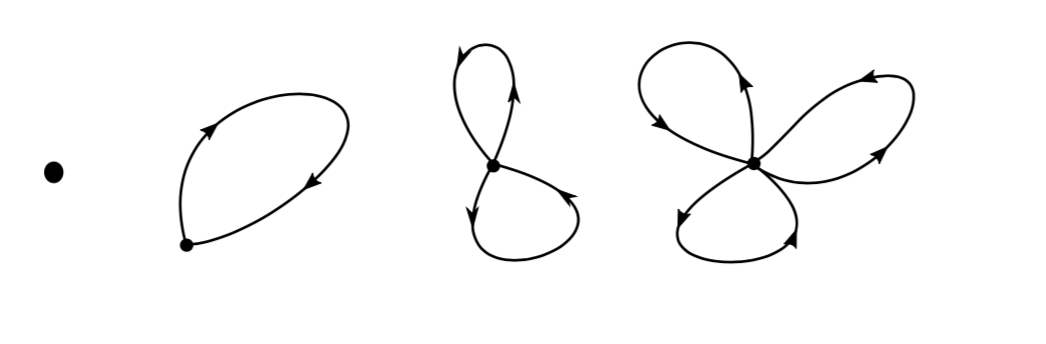}
\caption{Possible topological behaviors of bounded flowers for triangle tiling billiards, up to orientation, supposing the Bounded Flower Conjecture.}
\label{fig:petals}
\end{figure}

\begin{proposition}\label{prop:treetoflower}
For a triangle (cyclic quadrilateral) tiling, the Bounded Flower Conjecture is equivalent to the Tree Conjecture.
\end{proposition}

\begin{proof}
Suppose that the Bounded Flower Conjecture fails for the petals $\gamma_j, j \in J$ of some flower $\gamma$ with a pistil $v \in V$. Take all of the petals $\gamma_i$ in this family that are not contained in $\Omega^{\gamma_j}$ for some $j \in J, j \neq i$. Suppose that their indices belong to a subset $J_0 \subset J$. Then  there exists a periodic trajectory $\delta$ passing by the same tiles as $\cup_{j \in J_0} {\gamma_j}$, with $\cup_{j \in J_0} \Omega^{\gamma_j} \subset \Omega^{\delta} $. This trajectory then contours a tile.\footnote{Our reader can easily find a trajectory $\delta$ for all obstructions for the Flower Conjecture for periodic triangle tilings on Figure \ref{fig:one big picture with what can go wrong}.}

Now we prove that the Bounded Flower Conjecture implies the Tree Conjecture. Take some periodic trajectory $\delta$. Then the domain $\Omega^{\delta}$ contoured by $\delta$ is foliated by a family of trajectories in $\mathcal{P}^{\delta}$, among which only a finite number of singular ones. Now we contract $\delta$ inside $\Omega^{\delta}$ in a direction of the inner normal to $\partial \Omega^{\delta}$, in order to obtain a flower $\gamma$ with a singularity in some vertex $v \in \Omega^{\delta} \cap V$. 

If the trajectory $\delta$ contracts to a vertex, hence the corresponding graph $G_{\Delta}^{\delta}$ is a point and the proof is finished. 

Suppose now that $\delta$ contracts to a non-trivial flower $\gamma$. \emph{We can assume that such flower has its only singularity in} $v \in V$. Indeed, if it hasn't, then under folding $\gamma$ maps to a chord $l$ which connects $\mathcal{F}(v)$ with $\mathcal{F}(v')$ for some $v,v'\in V, v \neq v'$. But then one may perturb the initial direction of $\delta$ to obtain a perturbed trajectory $\delta'$, in such a way that a perturbed chord $l'$ (defined by $\mathcal{F}(\delta') \subset l'$) passes by $v$ but doesn't pass by $v'$ anymore, and the symbolic dynamics of the trajectory $\delta'$ is the same as that of $\delta$. This can be achieved since the set $\mathcal{F}(V)$ is a countable subset of $\Sph^1$.

Hence we obtain a flower $\gamma$ with a pistil in some vertex $v \in V$ with $m$ petals, where $m \in \{0,1,2,3\}$ for triangle tilings and $m \in \{0,1,2\}$ for quadrilateral tilings. Now approach each of the petals $\gamma_j$ (from the inside) by periodic trajectories $\delta_j \subset \Omega^{\gamma_j}$ as leaves of $\mathcal{P}^{\gamma}$. Then we have a decomposition: $G_{\Delta}^{\delta}=\cup_j G_{\Delta}^{\delta_j} \cup e_j$, where $e_j$ are the edges passing through $v$ inside each of the petals $\gamma_j$.\footnote{Moreover, the symbolic dynamics of the initial periodic trajectory $\delta$ is defined by the dynamics of the periodic trajectories $\delta_j$.} We define in this way a recurrence process (by the length of $\delta$) that will eventually stop at a trajectory of period $6$. This proves that the Tree Conjecture follows from the Bounded Flower Conjecture.
\end{proof}

\begin{figure}
\centering
\includegraphics[scale=0.5]{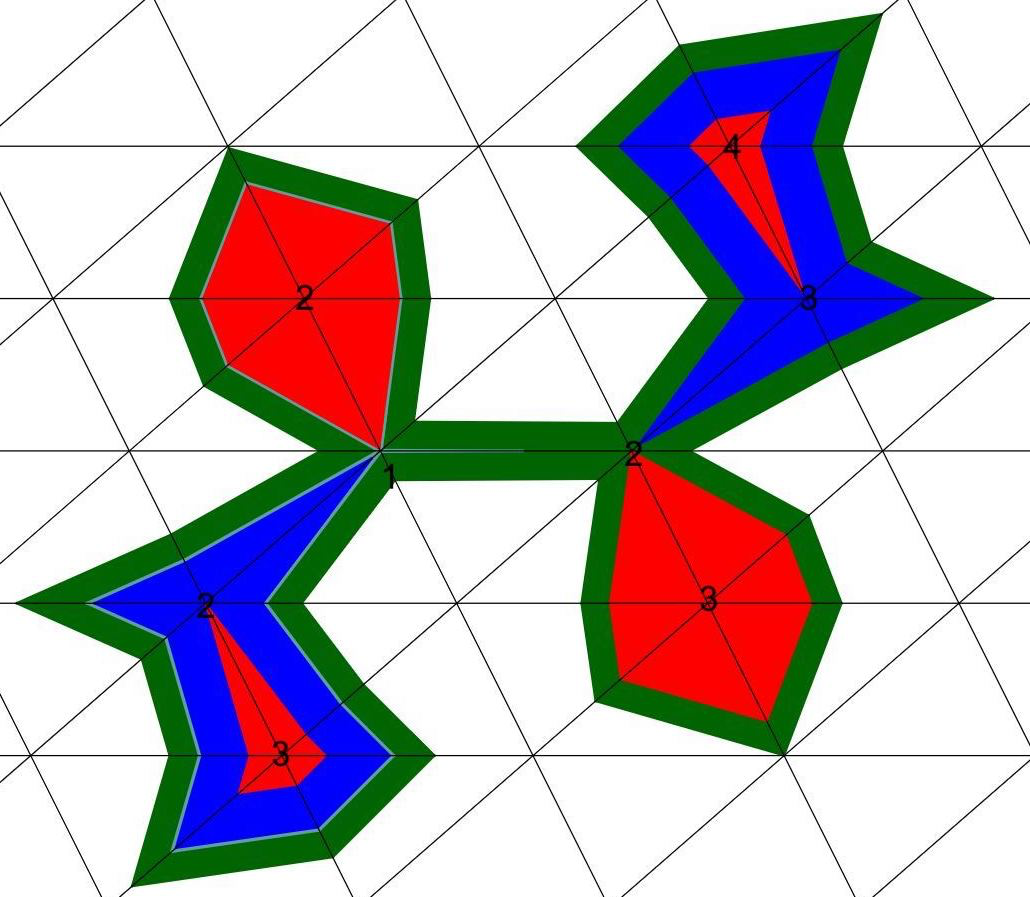}
\caption{\emph{Contraction of a periodic trajectory onto a sequence of flowers in the parallel foliation.} Figure by Ofir David.}
\label{fig:full}
\end{figure}

\begin{example}
The proof of the Tree Conjecture from the Bounded Flower Conjecture is constructive. For any periodic trajectory $\delta$, the tree $G_{\Delta}^{\delta}$ can be constructed  as a growing union of finite graphs, $G_{\Delta}^{\delta}= \cup_{k=1}^K G_k$. On each step one adds to the graph $G_k$ pistils of new flowers at each step with the edges inside their petals connected to these pistils. Any vertex $v \in \Omega^{\delta} \cap V$ is a pistil of a flower on some step of this process, by the point 2. of Lemma \ref{lemma:use_me_on parallel and radial}.

For a trajectory $\delta$ on Figure \ref{fig:full}, on the first step, one obtains a flower with one pistil and $3$ petals, and $G_1$ is a graph with one vertex marked by $1$ on Figure \ref{fig:full}. One then defines $G_2$ as a $1$-level tree with a parent marked by $1$ and three children marked by $2$ connected to it. Then, $G_3$ is the union of $G_2$ with edges going to additional vertices marked by $3$, and $G_4$ is $G_3$ with one additional vertex marked by $4$ and a corresponding edge. The final graph $G_{\Delta}^{\delta}$ is a tree.
\end{example}

\subsection{Obstructions to the Flower Conjecture.}\label{subs:proof of the tree conjecture}

Starting from here and till the end of this Section, the only tiling we consider is a periodic triangle tiling. 

The only cases of global behavior of bounded flowers contradicting the Bounded Flower Conjecture and respecting Proposition \ref{prop:list of possible local behaviours PTT} can be simply enumerated, see Figure \ref{fig:one big picture with what can go wrong} and the following list. In this list, in each of the cases we stress the set of petals $\mathcal{O}$ for which the Bounded Flower Conjecture doesn't hold.

Till the end of this Section, all of the flowers are considered bounded. We denote flowers by $\gamma$, and their petals by the same letter with indices.

\bigskip

\begin{center}
\textbf{Topological obstructions to the Bounded Flower conjecture.}
\end{center}

\begin{itemize}
\item[\textbf{2.1}] A flower has one petal $\gamma_1$ that passes by a pair of opposite tiles, $\mathcal{O}=\{\gamma_1\}$. 
\item[\textbf{2.2}] The petal $\gamma_1$ passes by a pair of neighbouring tiles in $e$ but $e \notin \Omega^{\gamma_1}$, $\mathcal{O}=\{\gamma_1\}$.
\item[\textbf{4.1}] A flower has two petals $\gamma_1, \gamma_2$ of different indices as curves. For \textbf{4.1a} and \textbf{4.1b}, a petal $\gamma_1$ passes by opposite tiles and a petal $\gamma_2$ passes by two neighbouring tiles. The two cases occur when \textbf{4.1a} $\Omega^{\gamma_2} \subset \Omega^{\gamma_1}$ (and $\mathcal{O}=\{\gamma_1\}$) or \textbf{4.1b} 
$\Omega^{\gamma_1} \subset \Omega^{\gamma_2}$ (and $\mathcal{O}=\{\gamma_2\}$). In the case \textbf{4.1c} both petals $\gamma_1$ and  $\gamma_2$ pass by neighbouring tiles but $\Omega^{\gamma_1} \subset \Omega^{\gamma_2}$ and $\mathcal{O}=\{\gamma_2\}$.
\item[\textbf{4.2}] The petals $\gamma_1$ and $\gamma_2$ are of the same index but $\gamma_1$ passes by opposite triangles, $\mathcal{O}=\{\gamma_1\}$.
\item[\textbf{6.1}] A flower has three petals $\gamma_j, j=1,2,3$. One of the loops $\gamma_3$ passes by opposite triangles and $\Omega^{\gamma_2} \subset \Omega^{\gamma_3}$, $\mathcal{O}=\{\gamma_3\}$.
\item[\textbf{6.2}] The petals $\gamma_j, j=1,2,3$ are such that $\Omega^{\gamma_3} \subset  \Omega^{\gamma_2} \subset  \Omega^{\gamma_1}$, and $\mathcal{O}=\{\gamma_1, \gamma_2\}$.
\item[\textbf{6.3}] A flower has three petals $\gamma_j$, and all of them pass by neighbouring tiles. Although $\Omega^{\gamma_1} \cup \Omega^{\gamma_2} \subset \Omega^{\gamma_3}$ and $\mathcal{O}=\{\gamma_3\}.$
\end{itemize}

\begin{remark}
This list is given modulo a possible change of orientations of all the petals. Without loss of generality, we fix the orientations as shown on Figure \ref{fig:one big picture with what can go wrong}.
\end{remark}

Our goal is now to prove that all of the cases listed above are not realized by triangle tiling billiard trajectories. We first present our main tools.

\begin{figure}
\includegraphics[scale=0.31]{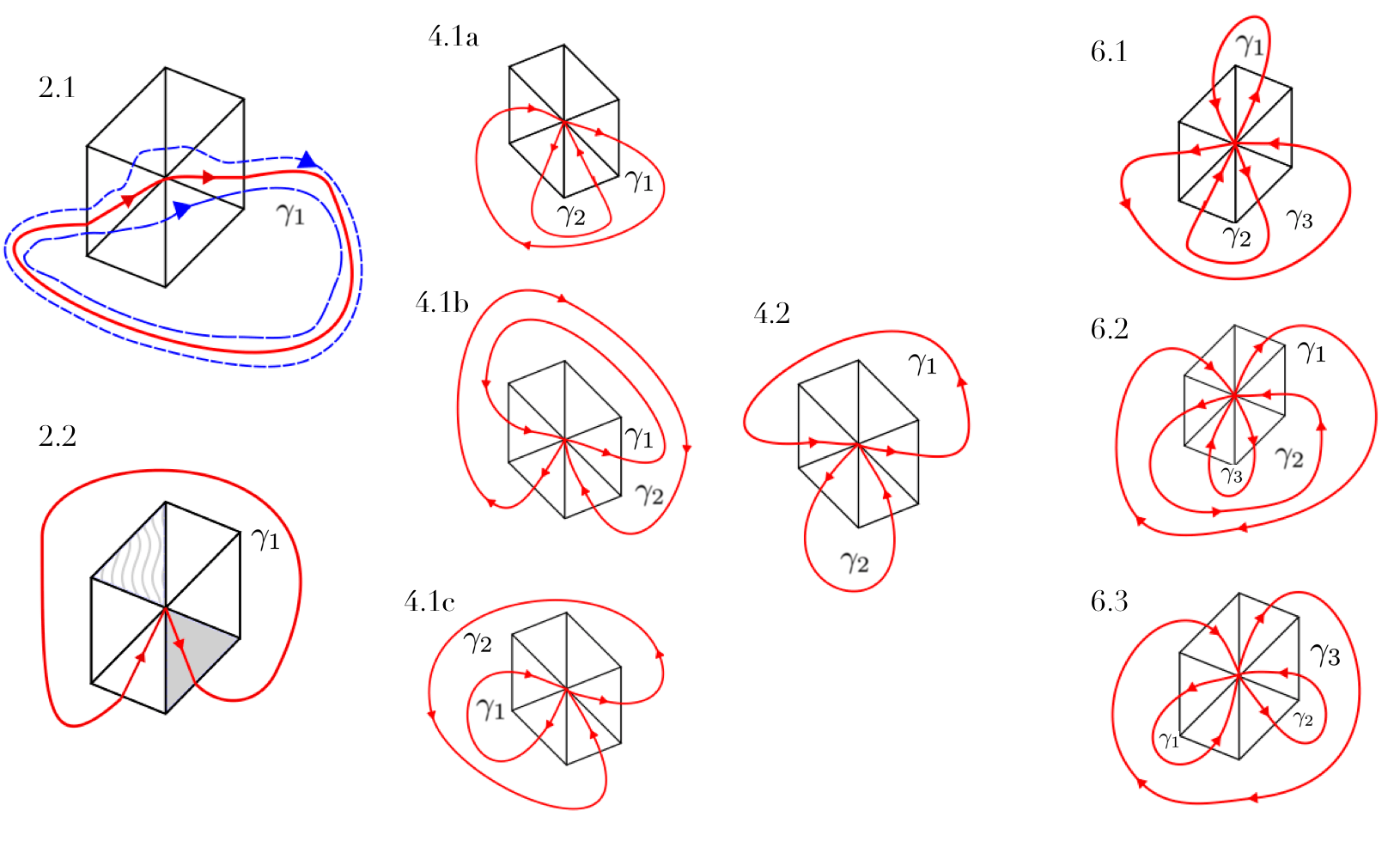}
\caption{\emph{A list of topological obstructions for the Bounded Flower Conjecture for triangle tiling billiards. } 
The first number in the name of the obstruction is the number $s$ of separatrix segments (twice the number of petals) in a bounded flower.
This Figure carries topological information, i.e. the way the trajectories are placed with respect to each other globally as well as the local combinatorics in the union $\Theta_v$ of six trianglular tiles containing  $v\in V$.}
\label{fig:one big picture with what can go wrong}
\end{figure}

\bigskip

As said before, the parallel and ray foliations can be defined for a very large class of tilings, see Remark \ref{rem:locally foldable}. But more specifically, we use two tools which are proper to a periodic triangle tiling.
First, the periodic symbolic words are the \emph{squares} of some symbolic words, see 5. in Theorem \ref{thm:triangle_tiling_billiards_info}. This \emph{square property} is a very strong property. We prove it in Theorem \ref{thm:one-more-time} but for now we use it as acquiered to obtain the proof of the Bounded Flower Conjecture.\footnote{If our reader wants to be sure that there is no logical loop in the argument (and they are right!), we send them to study Section \ref{sec:trop_cool} of the second part of the work. The Section \ref{sec:trop_cool} is completely independent from the first part of this work, and gives a proof of the square property as a corollary of the renormalization process introduced in it.} 

Second, we use the \emph{symmetry} of the ray foliation $\mathcal{R}$ for triangle tiling billiards centered at a singularity. Both of these tools are very strongly related to the special features of the periodic triangle tiling. For example, both of these two properties break for cyclic quadrilateral tilings.

\smallskip

We first show in detail how to exclude the cases \textbf{2.1} and \textbf{2.2}, and then treat all the other cases.

\subsection{Exclusion of topological obstructions for one petal flowers.}\label{subs:MAIN}
Define a sign alphabet $\mathcal{S}:=\{+,-\}$ and a sign map $\sigma: \mathcal{A}_{\Delta}^2 \rightarrow \mathcal{S}$ explicitely by $\sigma(ab)=\sigma(bc)=\sigma(ca)=+$ and $\sigma(ba)=\sigma(cb)=\sigma(ac)=-$. This sign map extends to the map $\sigma: \left(\mathcal{A}_{\Delta}^2\right)^{\N} \rightarrow \mathcal{S}^{\N}$ that we denote by the same letter. This map simplifies any accelerated symbolic code of a curve into its \textbf{sign code}.

Very importantly, we consider the (accelerated) symbolic codes of periodic trajectories as \emph{cyclic words}, i.e. for us the two periodic words $w_0 \ldots w_n$ and $w_k w_{k+1} \ldots w_n w_0 \ldots w_{k-1}$ are equal for any $j,k \in \{0,1, \ldots, n\}, k \neq 0$ and any $w_j \in \mathcal{A}_{\Delta}^2$. Any (accelerated) symbolic code is a \emph{square} of some word in the alphabet $\mathcal{A}_{\Delta}^2$, and hence in the alphabet $\mathcal{S}$. 

\begin{example}
The accelerated (cyclic) symbolic code of a $6$-periodic orbit in a triangle tiling billiard can be written as $(ab \; bc \; ca)^2$ but also as $(bc \; ca \; ab \;)^2$. Its corresponding sign code for both cases is $(+++)^2$.
\end{example}

\textbf{A word on the notation.} In the following, we denote by $\gamma_j$ the petals and by $\delta_j$ the periodic trajectories approaching these petals or their unions. Second, we identify the trajectories with their symbolic orbits. We denote by the same letter an oriented closed curve on the plane as well as a corresponding cyclic periodic word in the alphabet $\mathcal{A}_{\Delta}^2$ or in the alphabet $\mathcal{S}$, via the sign map.

\smallskip

In order to exclude the case \textbf{2.1}, one uses the square property.

\begin{proposition}\label{prop:first}
A configuration \textbf{2.1} is never realized by a bounded flower.
\end{proposition}

\begin{proof}
Suppose that a configuration \textbf{2.1} is realized by some petal $\gamma_1$ in the vertex $v$ of a triangle tiling billiard, the only petal of its flower $\gamma, \gamma=\gamma_1$.

We now perturb $\gamma_1$ in the foliation $\mathcal{P}^{\gamma}$ in order to obtain two periodic trajectories $\deltain$ and $\deltaout$ in a small neighbourhood of $\gamma_1$ with $\deltain \subset {\Omega}^{\gamma_1}$ and $\deltaout \nsubseteq {\Omega}^{\gamma_1}$, see Figure \ref{fig:OTS2}.

We suppose that outside the set $\Theta_v$ the trajectories $\deltain, \deltaout$  and $\gamma_1$ pass by the same tiles. Then there exists a word $S \in \mathcal{S}^{\N}$ of even length such that the accelerated cyclic symbolic words of $\deltain$ and $\deltaout$ in the sign alphabet are: $\deltain=+--+S$ and $ \deltaout=-++-S$. We split $S=s \bar{s}$ into a concatenation of two words of equal length, $s,\bar{s} \neq \emptyset$. Then $\deltain=-+s \bar{s}+-$ and $\deltaout=+-s \bar{s}-+$.

But since the words $\deltain$ and $\deltaout$ are squares of some words in the alphabet $\mathcal{S}$, length considerations give that  simultaneously $-+s=\bar{s}+-$ and $+-s=\bar{s}-+.$ But these two equations imply that the word $s$ finishes  by $+$ and $-$ at the same time, which is a contradiction.
\end{proof}

In order to exclude \textbf{2.2}, one uses the symmetry of the ray foliation $\mathcal{R}_{p}$ with $p=\mathcal{F}(v)$.

\begin{proposition}\label{prop:pacman}
A configuration \textbf{2.2} is never realized by a bounded flower.
\end{proposition}

We observe that for the case \textbf{2.2} a following property holds. There exists a petal $\gamma_1$ and a tile $\theta_0 \ni v$ such that $\gamma_1 \cap \theta_0 \neq \emptyset$ and $\theta_0^v \subset \Omega^{\gamma_1}.$ In this case, we say that the tile $\theta_0$ is a \textbf{hungry tile} and that it eats up $\theta_0^v$. We call a flower $\gamma$ (not necessarily bounded) a \textbf{ hungry flower} if there exists a petal in this flower passing by a hungry triangle, see Figure \ref{fig:flower_notation}. This property is shared by configurations \textbf{2.2, 4.1c, 6.2} and \textbf{6.3}. In order to prove Proposition \ref{prop:pacman}, we prove a more general statement that excludes all the cases that we have just mentionned. 

\begin{figure}
\centering
\includegraphics[scale=0.05]{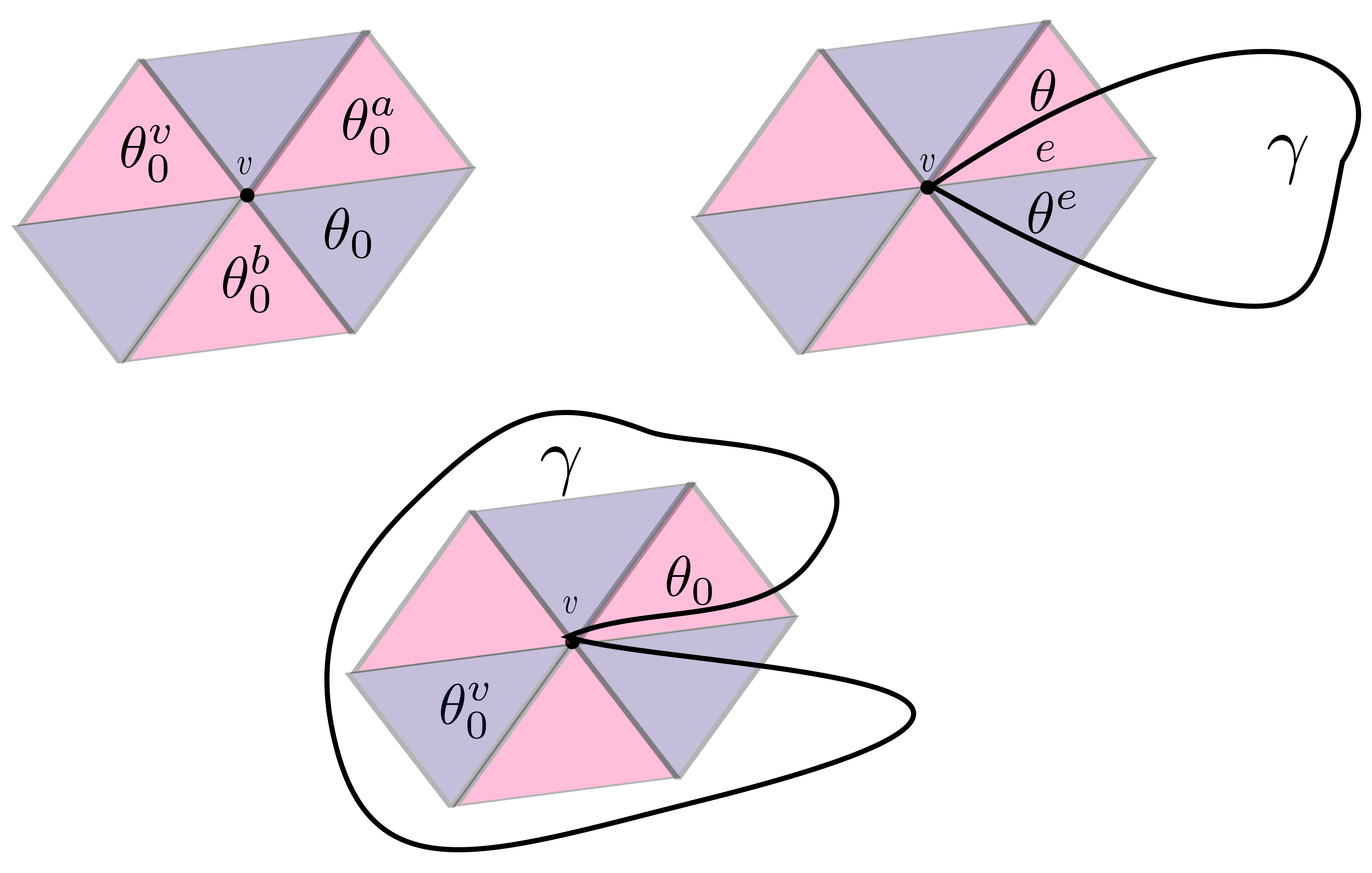}
\caption{\emph{Different notations related to triangle tilings and flowers in them.} From left to right, from top to bottom: first, for a tile $\theta_0$ we mark here two out of three of its neighbouring tiles $theta_0^a$ (sharing an edge $a$) and $\theta_0^b$ (sharing an edge $b$) as well as its opposite tile $\theta_0^v$ in the vertex $v\in V$; second, an illustration for the Flower Conjecture on the triangle tiling, a loop $\gamma$ satisfying the Flower Conjecture passes by $\theta$ and $\theta^e$ and the set $\Omega^{\gamma}$ contains the edge $e$; third, a petal of a hungry flower $\gamma$ is represented on the last picture, passing by a tile $\theta_0$ and such that the opposite tile $\theta_0^v$ is contained inside $\Omega^{\gamma}$.}
\label{fig:flower_notation}
\end{figure}

\begin{proposition}\label{prop:nohungry}
1. The ray foliation $\mathcal{R}_p$ with $p=\mathcal{F}(v), v \in V$ is centrally symmetric with respect to $v$, modulo a change of orientation of leaves in opposite tiles.
2. A configuration of separatices forming a hungry flower is never realized by triangle tiling billiard foliations.
\end{proposition}
\begin{proof}
For any separatrix segment of the trajectory $\gamma_0$ starting in a vertex $v$ and in the tile $\theta_0 \ni v$, consider a separatrix segment starting in $v$ and crossing the tile $\theta_0^v$ such that it lies on the same line as the initial segment. Simply by symmetry, the corresponding trajectory $\gamma_0^v$ is \emph{globally} centrally symmetric to $\gamma_0$, although its orientation is different from that of $\gamma_0$. This proves 1.

Consider now a hungry flower $\gamma$ in the vertex $v$ and include it in its ray foliation $\mathcal{R}^{\gamma}$. This foliation contains a symmetric flower $\gamma^v$ defined as in the proof of point 1 by symmetry. But the hungry flower configuration implies that these two flowers $\gamma$ and $\gamma^v$ intersect outside $v$. This is not possible since $\gamma$ and $\gamma^v$ are leaves of the same foliation, see Figure \ref{fig:symmetry of radial folation}.
\end{proof}

\begin{remark}
The two tiles $\theta_0$ and $\theta_0^v$ fold into two triangles in the bellow, symmetric with respect to the diameter $d$ of the circle $\mathcal{C}$ such that $d \ni \mathcal{F}(p)$. The corresponding symmetric trajectories $\gamma_0$ and $\gamma_0^v$ constructed in the proof of the above Proposition \ref{prop:nohungry} fold into the chords symmetric with respect to the same diameter $d$, see Figure \ref{fig:symmetry of radial folation}. In the ray foliation $\mathcal{R}_p$ the trajectories crossing $\theta_0$ ($\theta_0^v$) go out of (into) $v$.
\end{remark}

\begin{corollary}
Configurations \textbf{2.2, 4.1c, 6.2} and \textbf{6.3} are never realized by bounded flowers.
\end{corollary}

The possible obstructions that are left to exclude are \textbf{4.1a, 4.1b, 4.2}, and \textbf{6.1}. They are treated analogously to \textbf{2.1} in the next paragraph, by using the square property of accelerated symbolic codes in the sign alphabet.

\begin{figure}
\centering
\includegraphics[scale=0.08]{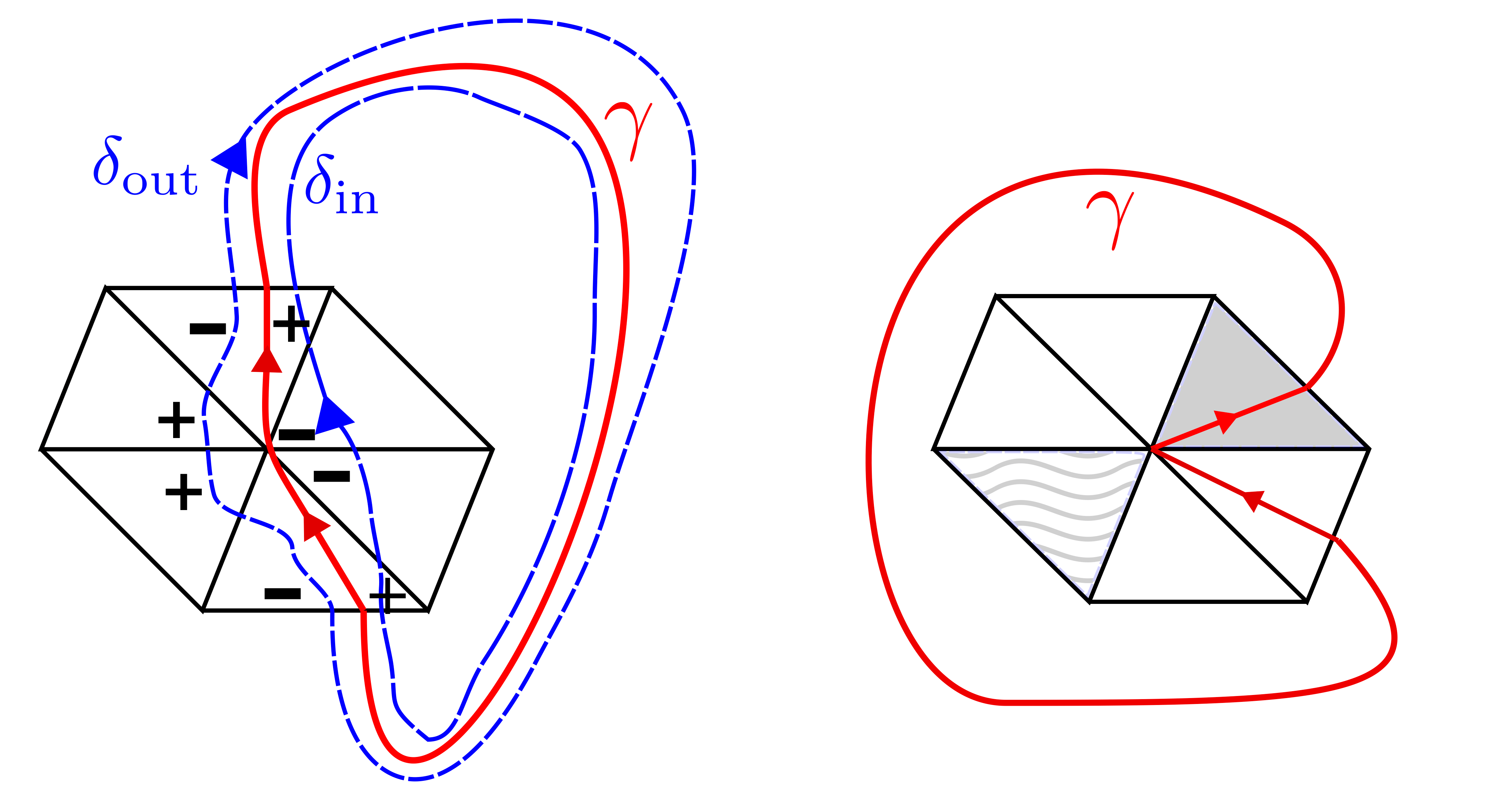}
\caption{This Figure serves several illustration purposes:
1. For any petal $\gamma$, the Flower Conjecture obstructions for it are represented by two possible behaviours represented on this Figure;
2. if $\gamma$ is an only petal in its bounded flower, then this Figure represents the obstructions \textbf{2.1} and \textbf{2.2};
3. this is an illustration for the proofs of Propositions \ref{prop:first} and \ref{prop:pacman}. For the case \textbf{2.1}: a petal $\gamma_1$ and two periodic trajectories $\deltain, \deltaout$ approaching it in the parallel foliation $\mathcal{P}^{\gamma_1}$. The sign codes of $\deltain$ and $\deltaout$ while passing by $\Theta_v$ are correspondingly $+--+$ and $-++-$. For the case \textbf{2.2}, a hungry tile $\theta_0$ (with a petal $\gamma_1$ passing through it) eats up the tile $\theta_0^v$.}
\label{fig:OTS2}
\end{figure}
\smallskip

\begin{figure}
\centering
\includegraphics[scale=0.4]{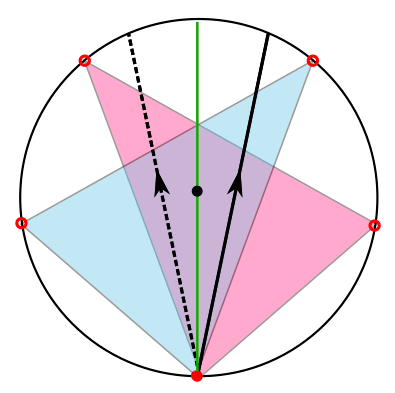}
\includegraphics[scale=0.04]{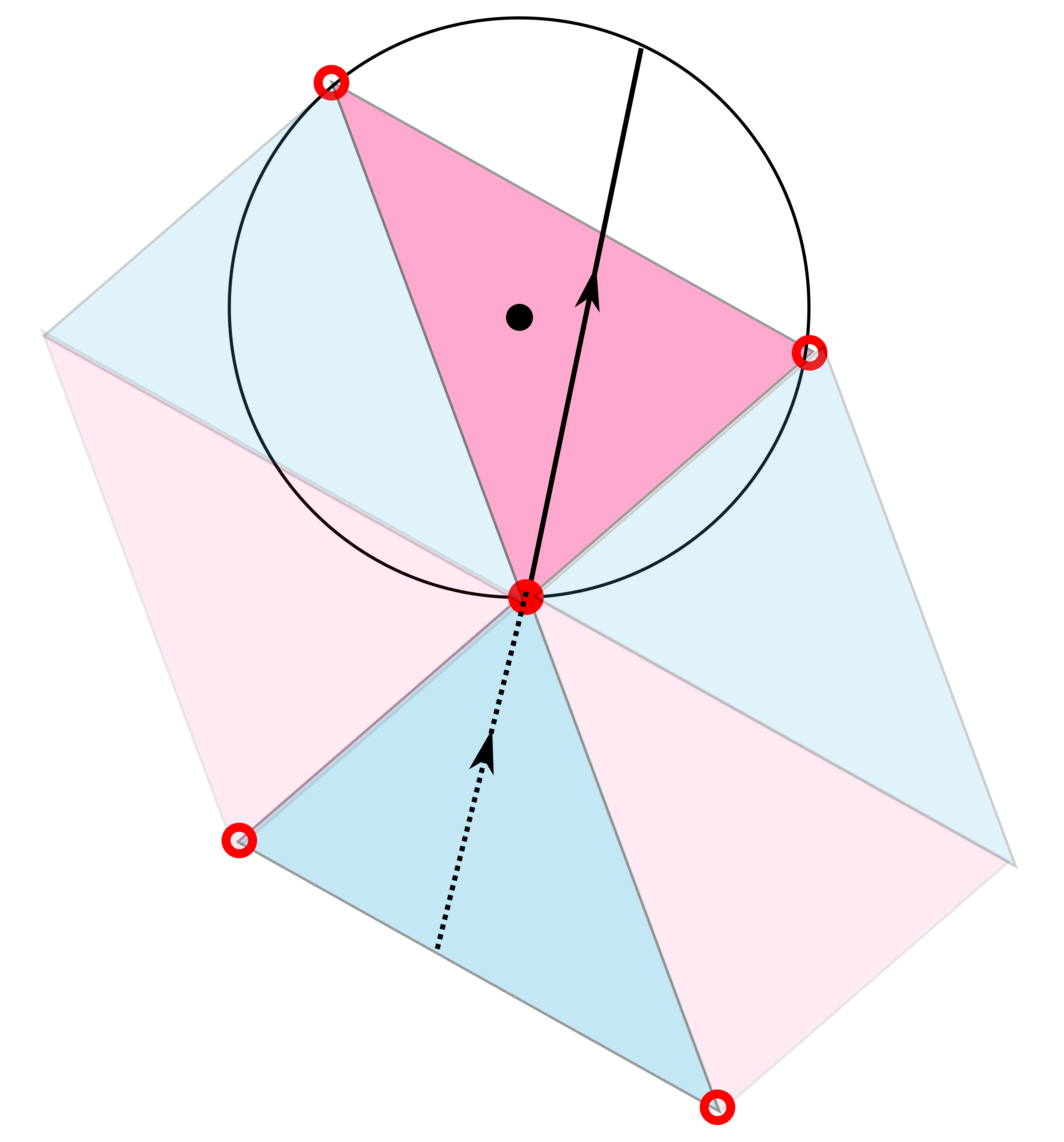}
\includegraphics[scale=0.04]{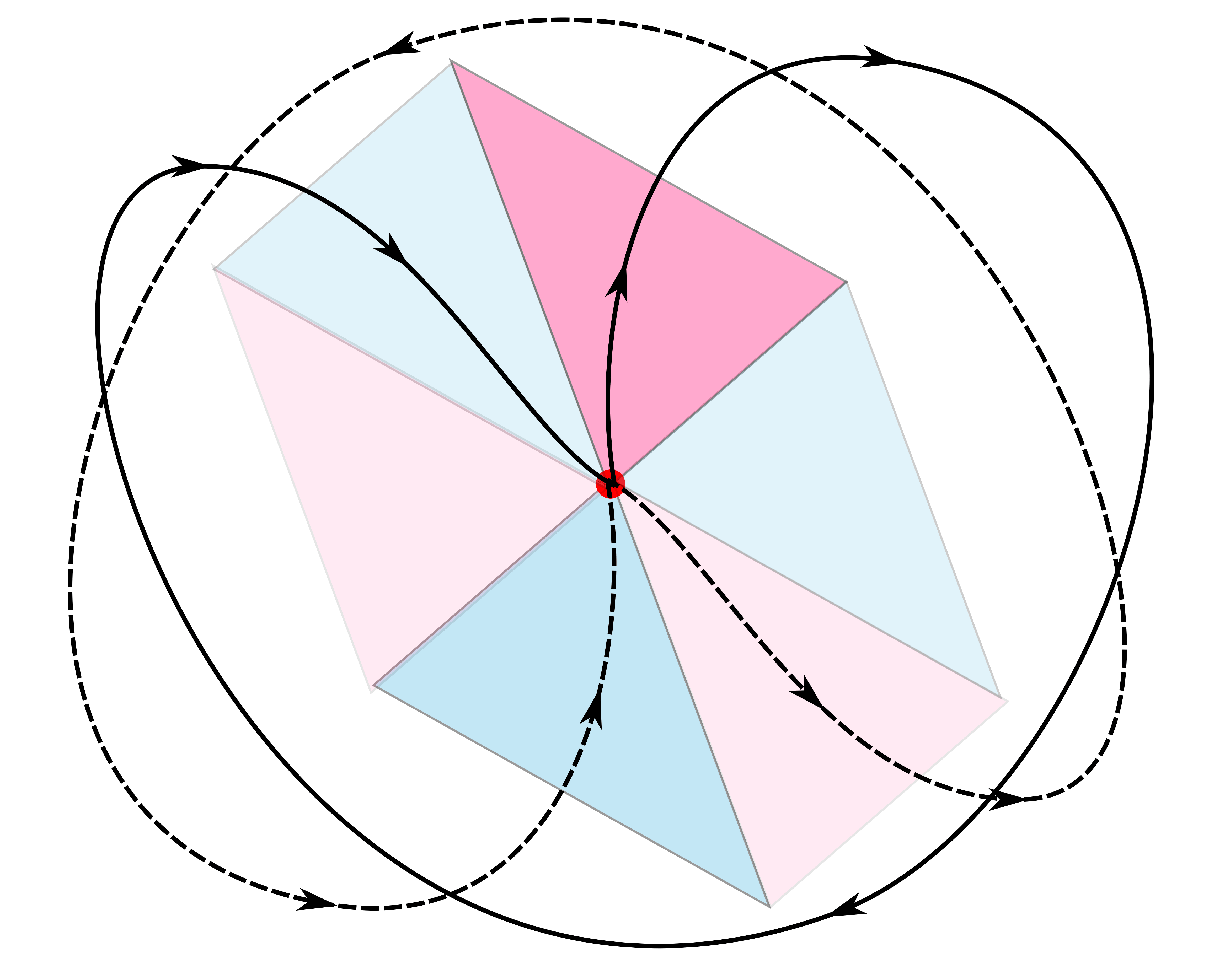}
\caption{\emph{Symmetry of the ray foliation} $\mathcal{R}_p$\emph{ with} $p=\mathcal{F}(v), v \in V$. From left to right:
1. Folded triangles  $\mathcal{F}(\theta_0)$ and $\mathcal{F}(\theta_0^v)$ are symmetric to each other with respect to the diameter $d \ni p$. 2. Associated unfolded segments. 3. A hungry flower $\gamma$ and a flower $\gamma^v$: their petals have to intersect but they can't! }
\label{fig:symmetry of radial folation}
\end{figure}

\subsection{Exclusion of remaining cases and finalisation of the proof.}\label{subs:REMAIN}

\begin{proposition}
Configurations \textbf{4.1a} and \textbf{4.1b} are never realized by bounded flowers.
\end{proposition}

\begin{proof}
Consider the case \textbf{4.1a}. We denote $\gammain:=\gamma_2$ and $\gammaout:=\gamma_1$. We approach $\gammain$ by a  trajectory $\delta_1$ from the inside ($\delta_1 \subset \Omega^{\gammain}$), and $\gammaout$ by a trajectory $\delta_2$ from outside ($ \delta_2 \subset \R^2 \setminus \Omega^{\gammaout}$). One can choose a trajectory $\delta$ inside the set $\Omega^{\gammaout} \setminus  \Omega^{\gammain}$ close enough to its boundary (in such a way that it passes by the same tiles as $\gammain \cup \gammaout$). All of the trajectories $\delta_1, \delta_2, \delta$ are chosen to be periodic, non-singular and belong to the same foliation $\mathcal{P}^{\gamma}$. 

Then, by the square property and from the combinatorics of such a configuration, we conclude that there exist the words $w,u \in \mathcal{S}^{\N}$ such that
\begin{align*}
\delta_1=(w--)^2,\\
\delta_2=(u-++-)^2,\\
\delta=++w--w++u-++-u.
\end{align*}

But since $\delta$ is also a symbolic square, from length considerations, we obtain the word equality $
-w++u-+=+-u++w-
$
which is impossible since $-\neq+$. The argument for the case \textbf{4.1b} is the same, with $\gammain:=\gamma_1$ and $\gammaout:=\gamma_2$.
\end{proof}

\begin{proposition}
Configuration \textbf{4.2} is never realized by a bounded flower.
\end{proposition}
\begin{proof}
Define three non-singular periodic trajectories $\delta_1, \delta_2$ and $\delta$ in the parallel foliation $\mathcal{P}^{\gamma}$. First, $\delta_j \in {\Omega}^{\gamma_j}$ and $\delta_j$ passes by the same tiles as $\gamma_j$ for $j=1,2$. Then, we take a trajectory $\delta$ that passes by the same tiles as the flower $\gamma$ and such that $\gamma \subset {\Omega}^{\delta}$. Then, there exist the words 
 $s, \bar{s}, w \in  \mathcal{S}^{\N}$ such that the words $s$ and $\bar{s}$ have equal length and 

\begin{align*}
\delta_1=-++-s\bar{s},\\
\delta_2=(--w)^2,\\
\delta=++w--w++s\bar{s}.
\end{align*}

Length considerations imply the following two equations:
$\bar{s}-+=+-s$ and $-w++s=\bar{s}++w-$. These two are incompatible, since the word $s$ has to finish by $-$ and $+$ simultaneously.
\end{proof}

\begin{proposition}\label{prop:last}
Configuration \textbf{6.1} is never realized by a bounded flower.
\end{proposition}

\begin{proof}
We choose periodic non-singular trajectories $\delta_j, j=1,2,3,4$ as follows:
\begin{itemize}
\item[•] the trajectories $\delta_j$  pass by the same tiles as $\gamma_j$  and $\delta_j \subset {\Omega}^{\gamma_j}$ for $j=1,2$,
\item[•]  a trajectory $\delta_3$ is close to the boundary $\partial 
\left( {\Omega}^{\gamma_1} \cup {\Omega}^{\gamma_3} \right)$ and ${\Omega}^{\delta_3}$ contains this boundary,
\item[•] a trajectory $\delta_4$ is close to the boundary of the set ${\Omega}^{\gamma_3} \setminus {\Omega}^{\gamma_2}
$ and is contained inside this set. 
\end{itemize}
Then there exist the words $w,v, U \in  \mathcal{S}^{\N}$ such that

\begin{align*}
\delta_1=(w--)^2,\\
\delta_2=(v++)^2,\\
\delta_3=++w--w++U,\\
\delta_4=--v++v--U.
\end{align*}
Since both $\delta_3$ and $\delta_4$ are symbolic squares, one can split the word $U$ in two words $u, \bar{u} \in \mathcal{S}^{\N}$ of equal length, $U=u \bar{u}$. The length considerations for $\delta_3$ and $\delta_4$ imply:
\begin{align*}
-w++u=\bar{u}++w-,\\
\bar{u}--v-=+v-- u.
\end{align*} 
Since the word $\bar{u}$ can't start from $+$ and $-$ at the same time, we have a contradiction.
\end{proof}

The Bounded Flower Conjecture for triangle tilings (Theorem \ref{thm:bounded_flower_conjecture}) now follows.

\begin{proof}
Take any vertex  $v \in V$ in a triangle tiling and a bounded flower in it.
One can suppose that $v$ is the only singularity that this flower meets.\footnote{The proof of this fact is word by word coming from the argument in the proof of Proposition \ref{prop:treetoflower}. A statement in the proof we are interested in is marked with italics.} Then such a flower has to satisfy the Bounded Flower Conjecture since "it has no choice": all the obstructions have been excluded in Propositions \ref{prop:first}--\ref{prop:last}.
\end{proof}

By Proposition \ref{prop:treetoflower},  this finishes the proof of the Tree Conjecture for triangle tiling billiards. Our strategy gives a new proof of the Tree Conjecture for obtuse triangle tiling billiards, previously proven in \cite{BDFI18}.

\begin{corollary}[Theorem 5.7., \cite{BDFI18}]
Any periodic trajectory in an \emph{obtuse} triangle tiling billiard encloses a tree which is a \emph{path.}
\end{corollary}
\begin{proof}
Consider flower in a vertex $v \in V$ (bounded or not) in an obtuse triangle tiling. Let $\gamma$ be an obtuse angle, and denote the six tiles in $\Theta_v$ as $\theta_{\bullet}$ and $\theta_{\bullet}^v$ correspondingly for the opposite to $\theta_{\bullet}$ tile. Here $\bullet \in \{\alpha, \beta, \gamma\}$ is an angle a tile $\theta_{\bullet}$ (and $\theta_{\bullet}^v$) has in the vertex $v$. 

Any flower in an obtuse tiling has \emph{at most two} petals. Indeed, fold $\Theta_v$ into a bellow. Then one simply verifies that $\mathcal{F}\left(\theta_{
\alpha}^v\right) \cap \mathcal{F}\left(\theta_{
\alpha}\right)=\{p\}$ and $\mathcal{F}\left(\theta_{
\beta}^v\right) \cap \mathcal{F}\left(\theta_{
\beta}\right)=\{p\}$, see Figure \ref{figure:proof_obtuse}. Hence a flower in $v$ can't simultaneously pass by the interior of the tiles $\theta_{\alpha}$ and $\theta_{\alpha}^v$ (the same for $\theta_{\beta}$ and $\theta_{\beta}^v$). This gives that each flower has at most $4$ separatrix segments in $v$ (two passing by $\theta_{\gamma}$ and $\theta_{\gamma}^v$ and two passing by one representative of each of the couples with angles $\alpha$ and $\beta$ in $v$). Hence, the graphs inside periodic trajectories in obtuse triangle tiling billiards are paths.
\end{proof}

\begin{figure}
\centering
\includegraphics[scale=0.1]{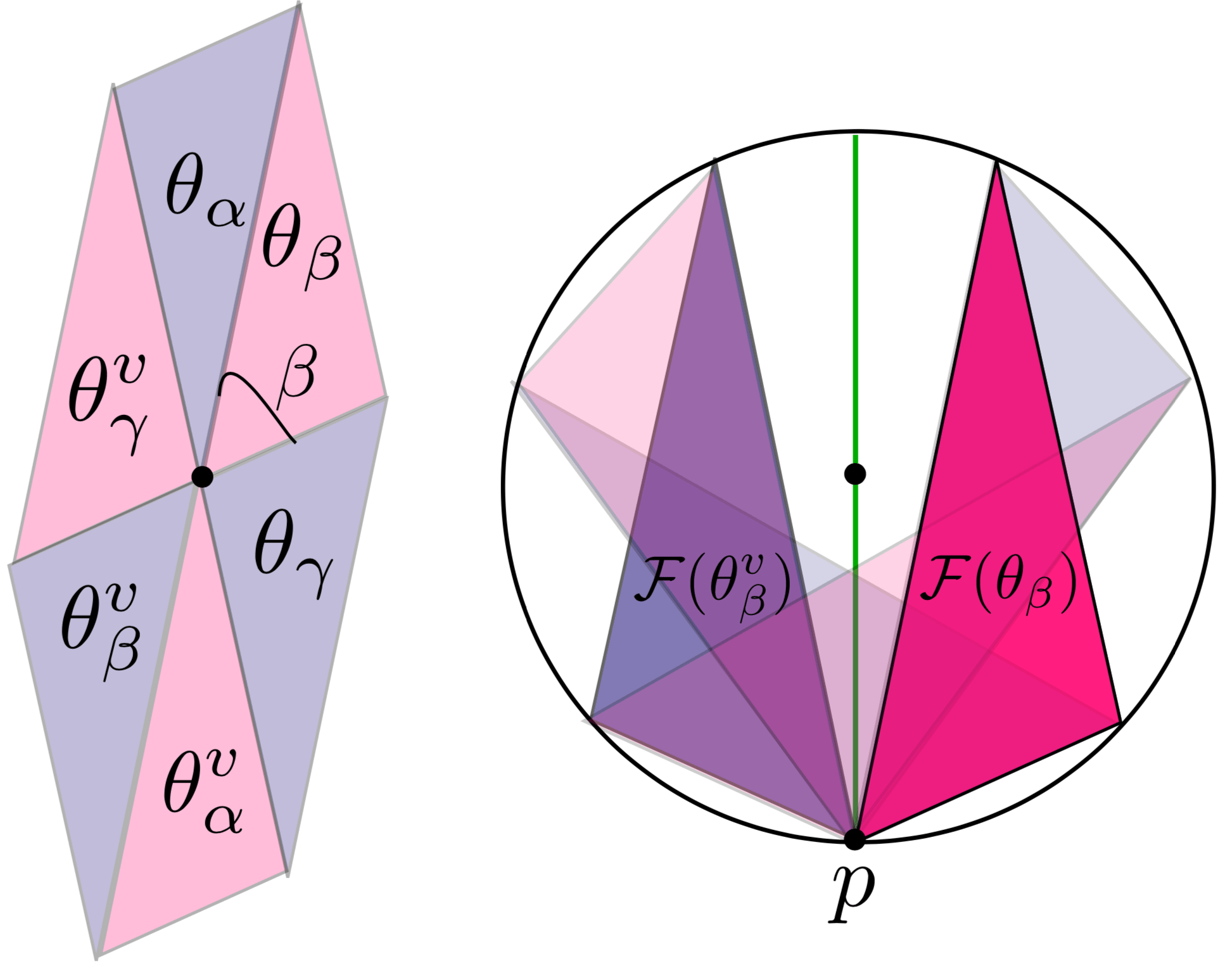}
\caption{On the left the neighbourhood of a vertex $v \in V$ in a triangle tiling is represented as a union of three pairs of opposite tiles. The tile $\theta_{\beta}$ has its angle equal to $\beta$ in the vertex $v$, as well as its opposite tile $\theta_{\beta}^v$. On the right one can see the image $\mathcal{F}(\Theta_v)$ under the folding map. The images of tiles $\theta_{\bullet}$ and  $\theta_{\bullet}^v$ with acute angle in $v$ intersect only in the point $p=\mathcal{F}(v)$.}
\label{figure:proof_obtuse}
\end{figure}

\newpage
\begin{center}
\textbf{Part II.-- Renormalization for fully flipped $3$-interval exchange transformations}
\end{center}

In this part, we introduce the renormalization process on the family $\CETthree$ of fully flipped $3$-IET on the circle (see Section \ref{sec:fully_flipped_intro} for definitions). This renormalization process is a combinatorial counterpart of the process of contraction of periodic trajectories of tiling billiards onto flowers that has been described in paragraph \ref{subs:FLOW} and used in the proof of Tree Conjecture (Conjecture \ref{conj:tree}) in the Part I.

\section{Arithmetic orbits of real-rel leaves and billiard trajectories.}\label{sec:arithmetic_orbits}

Any map $F \in \CETthree$ is defined by a triple $(l_1, l_2, l_3) \in \Delta_2$ and a parameter $\tau \in \Sph^1$, see Section \ref{sec:fully_flipped_intro}. The family $\CETthree$ has a $3$-dimensional space of parameters $\Delta_2 \times \Sph^1$ with a symmetry around the plane $\tau=\frac{1}{2}$. Indeed, a map $F_{\tau}^+:=F^{l_1, l_2, l_3}_{\tau}$ is conjugated to a map $F_{1-\tau}^-:=F^{l_3, l_2, l_1}_{1-\tau}$ via a change of orientation, $F_{\tau}^+=i \circ F_{1-\tau}^- \circ i$. Here $i$ is a global involution on $\Sph^1, i: p \mapsto 1-p$. In particular, this means that the maps in $\CETthreehalf$ have extra symmetries and commute with a global involution. This was already noticed in [paragraph 4.1,  \cite{Olga}]. For the following we suppose $\tau \in [0, 1/2]$.

In this Section and till the end of the article we associate to a quadruple of parameters $(l_1, l_2, l_3, \tau) \in \Delta_2 \times [0, \frac{1}{2}]$ a quadruple $(x_1, x_2, x_3, r) \in \Delta_2 \times [0, \frac{1}{2}]$ connected to it by linear relations \eqref{eq:l and x} and the relation
\begin{equation}\label{eq:r_and_x}
r:=\frac{1}{2}-\tau, r \in [0, \frac{1}{2}].
\end{equation}
The connection between these two sets of parameters is one-to-one, and in this work we navigate from one to another.

\bigskip

In the following Lemma, we formalize the connection between triangle tiling billiards and real-rel deformations of Arnoux-Rauzy maps discussed in the paragraph \ref{subs:rrdeformAR}.

\begin{lemma}\label{lem:ARconnection}
Take any triple $(x_1, x_2, x_3) \in \Delta_2 \setminus \partial \Delta_2$. Define $T:=T^{x_1,x_2,x_3} \in \AR$ a corresponding Arnoux-Rauzy map. Then for any $r \in \left[0, \frac{1}{2}\right]$ the following holds:
\begin{itemize}
\item[1.] let $T_r$ be a first-return map of a vertical flow on the translation surface $X^{x_1,x_2,x_3}_r$ in a real-rel leaf of $X_T$ on a horizontal transversal. Then $T_r=F^2$ for $F=F^{l_1,l_2,l_3}_{\tau} \in \CETthree$, where its parameters $l_j$ are defined by \eqref{eq:l and x} and \eqref{eq:r_and_x};
\item[2.] for any point $p \in \Sph^1$ the \textbf{displacement} $T_r(p)-p$ belongs to a finite set $\{0, \pm l_j \mid j\in \mathcal{N}_{\Delta}\}$;
\item[3.] moreover, if $r \leq \min\{x_j\}_{j=1}^3$, then for any $p \in \Sph^1$, $T_r(p)-p \neq 0$ and the map $T_{r}: \Sph^1 \rightarrow \Sph^1 $ is a $6$-IET with the intervals of continuity $I_j^{\pm}$ of lenghts $|I_j^{\pm}|=\frac{x_j}{2} \pm r, j \in \mathcal{N}_{\Delta}$.
\end{itemize}
\end{lemma}

\begin{proof}
For $r=0$, the statement of this Lemma is equivalent of that of Proposition \ref{prop:the_squares_are_Rauzy} and has already been proven in \cite{Olga}. Moreover, the point 1. follows from the fact that the surface $X^{x_1,x_2,x_3}_0$ is a double-cover of a non-orientable surface with a first-return map equal to $F_{1/2}^{l_1,l_2,l_3}$, by Proposition \ref{prop:the_squares_are_Rauzy}. The horizontal moves of singularities for the fully flipped interval exchange transformation $F$ are giving birth to horizontal moves of singularities on the surface $X$. The change of parameter $r$ is exactly that of the relative positions of singularities on the surface $X_r$ in the real-rel foliation. 

Suppose now that $\tau:=\frac{1}{2}-r$ and $r \neq 0$. We suppose that $r \in \left[ 0, \min_{j} \left\{ \frac{x_j}{2}\right\}\right)$ or, equivalently, $\tau > \max(l_j)$. Then, by a direct calculation, one shows that the map $F^2$ has $6$ intervals of continuity defined as follows:
\begin{align*}
I_2^+:=\left(l_2+\tau, 1 \right), I_2^-:=\left(0,\tau-l_2\right),\\
I_3^+:=\left(\tau-l_2, l_1\right), I_3^-:=\left(l_1, l_1+\tau-l_3\right),\\
I_1^+:=\left(l_1+\tau-l_3, l_1+l_2\right), I_1^-:=\left(l_1+l_2, l_2+\tau\right).
\end{align*}
The lengths of these intervals verify $|I_j^{\pm}|=\frac{x_j}{2} \pm r$.\footnote{For $r=0$ the intervals $I_j^+$ and $I_j^-$ have equal length, and are exactly the interavls of continuity of the maps in the Arnoux-Rauzy family $\AR$.}  Here the intervals of continuity of $F$ can be represented as unions:
\begin{equation}\label{eq:repres_abc}
I_1 = I_2^- \cup I_3^+, \; \; \; I_2=I_3^- \cup I_1^+, \; \; \; I_3=I_1^- \cup I_2^+.
\end{equation}
The map $F$ is an orientation reversing isometry on each of the intervals $I_j^{\pm}, j \in \mathcal{N}_{\Delta}$ and for any couple $(j,k)$, with $j \neq k$:
\begin{equation*}
|I_j^+|+|I_k^-|=\frac{x_j+x_k}{2}=|I_j^-|+|I_k^+|.
\end{equation*}
This implies that the previous decomposition \eqref{eq:repres_abc} can be rewritten as
\begin{equation*}
I_1 = F(I_3^-) \cup F(I_2^+), \; \; \; I_2=F(I_1^-) \cup F(I_3^+), \; \; \; I_3=F(I_2^-) \cup F(I_1^+),
\end{equation*}
Moreover, the images $F(I_j^{\pm})$ cover the interval $[0,1]$ in the following order: $[0,1]=F(I_3^-) \cup F(I_2^+) \cup F(I_1^+) \cup F(I_3^-) \cup F(I_2^-) \cup F(I_1^+)$. Then, one more application of $F$ maps the intervals $F(I_j^{\pm})$ onto the circle in the following order $\Sph^1=T(I_3^-) \cup T(I_3^+) \cup T(I_1^-) \cup T(I_1^+) \cup T(I_2^-) \cup T(I_2^+)$.

The intervals $I_j^{\pm}$ can be distinguished one from another by their symbolic dynamics, e.g. $I_1^+=\{p \in \Sph^1 : p \in I_2, F(p) \in I_3\}$. Analogically, the first steps of accelerated symbolic codes of $I_1^-, I_2^+, I_2^-, I_3^+, I_3^-$ are $cb, ca, ac, ab$ and $ba$ correspondingly. The displacement for every $p \in  I_j^{\pm}, j \in \mathcal{N}_{\Delta}$ can be calculated explicitely by the use of these codes. The displacement is equal to zero if and only if $F$ has a $2$-periodic interval (this happens if and only if $\tau \leq \max(l_j)$).
\end{proof}

\begin{remarkimp}\label{rem:connectionbilliards}
From the point of view of triangle tiling billiards, the inclusion  $T_r(p)-p \in \{\pm l_j\}$ is represented by the fact that a trajectory changes its direction after two refractions exactly by this amount, see Figure \ref{fig:triangle_angle} and Theorem 3.6 in \cite{BDFI18} for details. If the displacement for a map $F \in \CETthree$ is equal to $0$, there is no corresponding billiard trajectory. The six-element set $\{\pm l_j\}$ has also been considered in relation to the artihmetic orbits of Arnoux-Rauzy maps (in different terminology) by P. Hopper and B. Weiss in \cite{HW18}, see their Proposition 4.6 and following discussion. Most importantly, the set of displacement values of $T_r$ doesn't depend on $r$.
\end{remarkimp}

\begin{figure}
\centering
\includegraphics[scale=0.05]{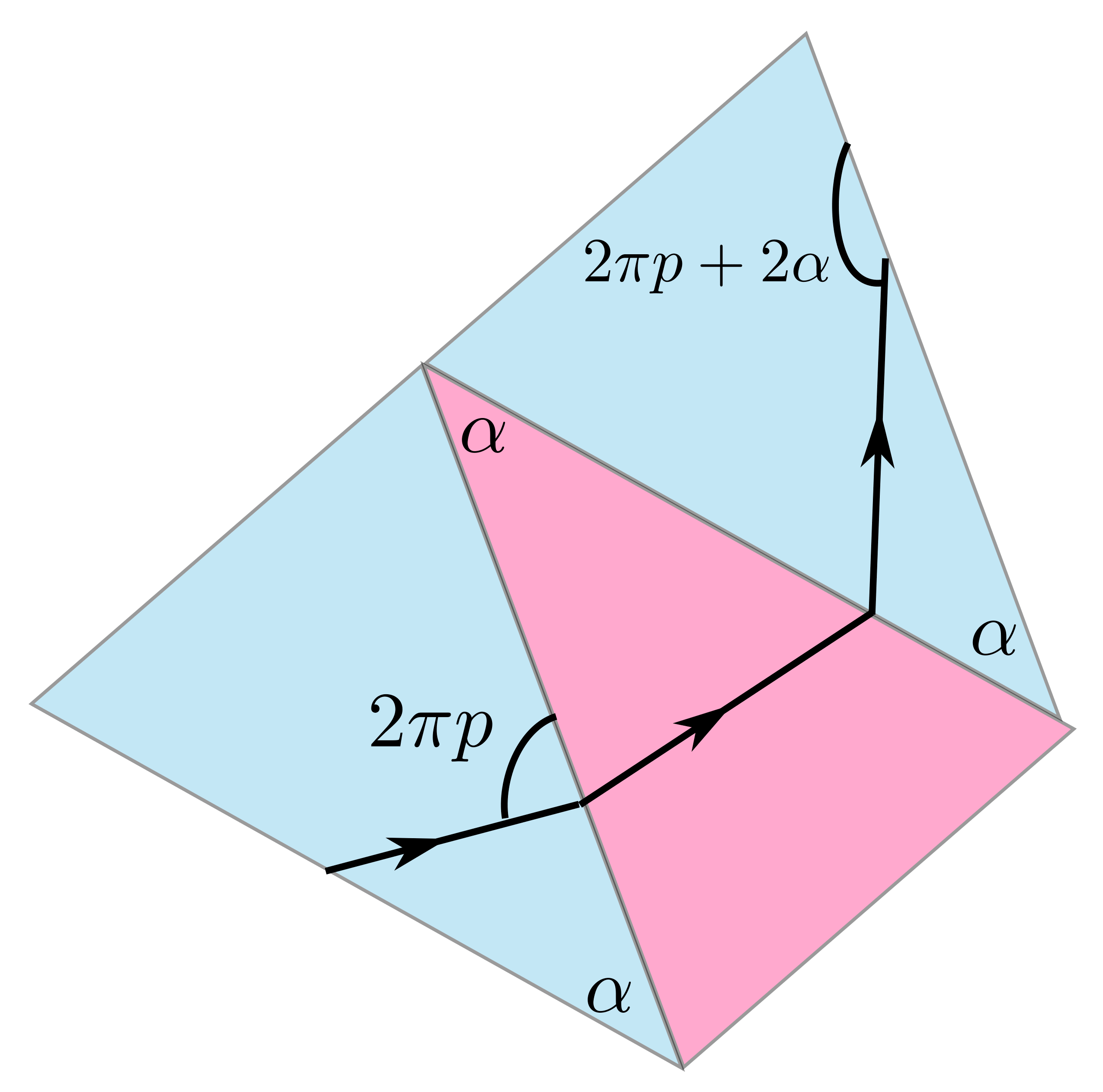}
\caption{For a trajectory that has an angle equal to $2 \pi p$ with some fixed line, $p \in \Sph^1$, after two reflections with respect to the side $c$ and then side $b$, the trajectory has an angle equal to  $2\pi p+2 \alpha=2\pi (p + l_1)$ with respect to that same line. Analogously, the displacements in all the other directions are measured by $\pm l_j, j \in \mathcal{N}_{\Delta}$.}
\label{fig:triangle_angle}
\end{figure}

We now define the arithmetic orbits of the family of dynamical squares of the maps in the family $\CETthree$. Here we follow almost word by word the definition in [Section 4, \cite{HW18}] modulo one important difference. 

For any map $F \in \CETthree$, define an interval exchange transformation $T: \Sph^1 \rightarrow \Sph^1$ by $T:=F^2$. Let $H$ be a group of rotations of $\Sph^1=\R / \mathbb{Z}$ generated by six numbers $\pm l_j, j \in \mathcal{N}_{\Delta}$. Denote $\Gamma$ the Cayley graph of $H$ with respect to these six generators. Consider a periodic triangle tiling with the angles of tiles defined by the relation \eqref{eq:correspondance}. We isomorphically embed $\Gamma$ to the plane (as a graph) to be the set of edges connecting the barycenters of all \emph{positively oriented} triangles in this tiling. Of course, one could see this graph in a slightly simpler way (as it is done  in \cite{HW18}). Although the way we propose to do it here is a great use for us since $\Gamma$ is a conformal copy of the graph $\Lambda_{\Delta}$.

A choice of a point $p \in \Sph^1$ defines an embedded curve in the graph $\Gamma$, i.e. a sequence of elements $h_n \in H$ such that 
$T^n(p)-p = h_n \mod \mathbb{Z}$. We call $\{h_n\}$ the \textbf{arithmetic orbit} of $p$.

\bigskip

For any triangle tiling billiard trajectory $\delta$ in a tiling corresponding to a map $T \in \AR$ via Proposition \ref{prop:the_squares_are_Rauzy}, we define a piece-wise linear curve $\gamma(p)$ that follows the orbit $(T^a)^{\circ k}(p)=F^{\circ 2k}(p), p \in \Sph^1$ in a following way. It starts in a \emph{barycenter} of a starting tile $\theta_0$ that the trajectory $\delta$ crosses, and connects it to the barycenter of a tile in which $\delta$ arrives after two reflections. Without loss of generality, we suppose that $\theta_0$ is positively oriented. 

Then the (oriented) segments that form $\gamma(p)$ belong to a six-element set $\left\{\pm \overrightarrow{AB}, \pm \overrightarrow{BC}, \pm \overrightarrow{CA}\right\}$. Here $A,B,C$ are the vertices of a tile $\theta_0$ marked on the plane. From all of the above follows that the \emph{curve} $\gamma(p)$ \emph{coincides with the arithmetic orbit of} $p$ \emph{under the map} $T$.

To sum up, the study of arithmetic orbits of the Arnoux-Rauzy maps and their real-rel deformations is equivalent to the study of triangle tiling billiard trajectories. In some sense, the latter are finer objects since their dynamics is a "square root" of the dynamics of the real-rel deformations of the family $\AR$. 

 \begin{remark}
For all of the maps $T=F^2$ with $F \in \CETthree$ their SAF invariant is zero.\footnote{For the definition of the SAF-invariant, see \cite{A81}. The Arnoux-Yoccoz map $T^a$ is a map for which $SAF(T^a)=0$  coexists with minimality, see \cite{AY}.}More generally, a square of any fully flipped interval exchange transformation has a zero SAF invariant. This statement has already been proven in [Proposition 18 in \cite{Olga}]. We give now a simpler proof which is a remark by Victor Kleptsyn. A fully flipped map $F:\Sph^1 \rightarrow \Sph^1$ can be represented as a composition $F={i} \circ H$ with $H \in \mathrm{IET}(\Sph^1)$ and $i$ the global involution on $\Sph^1$. Obviously, $SAF (i \circ H \circ i) = -SAF (H)$. Since $SAF: \mathrm{IET} \rightarrow \R \wedge_{\Q} \R$ is a group homomorphism, we have:
$SAF(F^2)=SAF(i \circ H \circ i \circ H)= SAF (i \circ H \circ i) + SAF(H)=- SAF (H) + SAF (H)=0$. 
\end{remark}

\section{Renormalization.}\label{sec:trop_cool}
A goal of this Section is to describe a renormalization process on the family $\CETthree$.

\subsection{Complete periodicity and integrability.}\label{subs:complete_periodicity_and_rotations}
First,we deal with several simple cases.

\smallskip

For any map $F \in \CETthree$ we say that an interval $I \subset \Sph^1$ is $k$\textbf{-periodic} if $F^k \mid_I=\mathrm{id}$ for some $k \in \N^*$ (and such $k$ is minimal). We call the set $\mathcal{P}_F$ of all $k \in \N^*$ such that there exists a $k$-periodic interval, the \textbf{set of interval periods} of the map $F$.

\begin{lemma}\label{lem:integrability}
Fix $(l_1, l_2, l_3) \in \Delta_2$ and $\tau \in [0, \frac{1}{2}]$. Then, the following holds for $F=F_{\tau}^{l_1, l_2, l_3} \in \CETthree$:
\begin{itemize}
\item[1.] if $\tau \leq \max(l_j)$ then $F$ is completely periodic. Moreover, if $\tau \in (0, \min(l_j)]$ then $\mathcal{P}_F=\{2,6\}$. If $\tau \in (\min(l_j), \midd(l_j)]$ then $\mathcal{P}_F=\{2,4n+2, 4n+6\}$, where $n=\lfloor \frac{\tau}{\min (l_j)}\rfloor \in \N^*$. In particular, if $l_j>\frac{1}{2}$ for some $j$, and $\tau \leq 1-l_j$ then $F$ is completely periodic;
\item[2.]  if $l_j>\frac{1}{2}$ for some $j$, and $\tau > 1-l_j$, then for any point $p \in \Sph^1$ either $F^2(p)=p$ or $F^2(p)=R_{\kappa}$ where $R_{\kappa}$ is a rotation by $\kappa=\frac{l_3}{l_2+l_3}$,  defined on an entire interval $I$ (with its endpoints identigied). This interval is defined as a connected component of points $q$ such that $F^2(q) \neq q$, containing $p$;
\item[3.] the set $\mathcal{P}_F$ is finite in any of these cases, and $\mathcal{P}_F \subset \{4n+2 \mid n\in \N^*\}$ for point 1. and in point 2. it is as well if $\varkappa \notin \Q$.
\end{itemize}
\end{lemma}

\begin{proof}
We suppose that $l_1 \geq l_2 \geq l_3$. Let for any $j \in \mathcal{N}_{\Delta}$ 
\begin{equation}\label{eq:Jj}
K_j:=I_j \cap F(I_j).
\end{equation}

First, if $\tau \leq l_3$, $F$ has three $2$-periodic intervals $K_j$. The set $\Sph^1 \setminus \cup_{j=1}^{3} K_{j}$ splits into three intervals, all belonging to the same $6$-periodic interval orbit.

Second, if $\tau \in (l_3, l_2]$ then $F$ has two $2$-periodic intervals $K_1$ and $K_2$. Denote $I_1^-:=(0,l_3), I_1^+:=(l_3, \tau), I_2^-:=(l_1, l_1+\tau-l_3), I_2^+:=( l_1+\tau-l_3, l_1+\tau)$. Then $[0,1]=I_1^- \sqcup I_1^+ \sqcup K_1 \sqcup I_2^- \sqcup I_2^+ \sqcup K_2 \sqcup I_3$ and
we have a following chain of images:

\begin{align*}
(0,l_3)=I_1^- \xmapsto{F} I_2^+ \mapsto I_3 \xmapsto{F} (\tau-l_3, \tau) \subset (0,\tau);\\
(l_3, \tau)=I_1^+ \xmapsto{F} I_2^+ \xmapsto{F} (0, \tau-l_3) \subset(0,\tau).
\end{align*}
Then in restriction to $(0, \tau)$ the first return map $F'$ of $F$ is a $2$-interval exchange transformation with combinatorics $\begin{pmatrix}
\overline{I_1^- }&I_1^+\\
{I_1^+}&\overline{I_1^-}
\end{pmatrix}$.\footnote{The notation here is analogous to the standard notation for the combinatorics of the dynamics of $\mathrm{IET}$. The difference is that some of the intervals (e.g. here the interval $I_1^- $) may be flipped. In this case we write a bar over such intervals. For more on these notations and in general, dynamical behavior of IETs with flips, see \cite{N89, Olga}.} Such a first return map is completely periodic (as first proven in \cite{K75}) with
$\mathcal{P}_{F'}=\{2n, 2n+2\}$, where $\frac{|I_1^-|}{|I_1^-|+|I_1^+|}=\frac{l_3}{\tau} \in [\frac{1}{n+1}, \frac{1}{n})$. This gives that $\mathcal{P}_F=\{4n+2, 4n+6\}$.

Finally, suppose $\tau \in (l_2, l_1]$, then $K_1$ is the only $2$-periodic interval for $F$. Consider now a following subdivision of the initial intervals of continuity: $I_1=I_1^- \cup I_1^0 \cup I_1^+ \cup K_1, I_2=I_2^- \cup I_2^+, I_3=I_3^- \cup I_3^+$, with
\begin{align*}
I_1^-:=(0, \tau-l_2), I_1^0:=(\tau-l_2, l_3), I_1^+:=(l_3, \tau)\\
I_2^-:=(l_1, \tau+l_1-l_3), I_2^+:=(\tau+l_1-l_3, l_1+l_2)\\
I_3^-:=( l_1+l_2, l_1+\tau), I_3^-:=(l_1+\tau, 1).
\end{align*}
Then we have a following chain of images:
\begin{align*}
I_1^- \xmapsto{F} I_3^+ \xmapsto{F} (l_2, \tau) \subset (0, \tau)\\
I_1^0 \xmapsto{F} I_2^+ \xmapsto{F} I_3^+ \xmapsto{F} (\tau-l_3, l_2) \subset (0, \tau)\\
I_1^+ \xmapsto{F} I_2^- \xmapsto{F} (0, \tau-l_3) \subset (0, \tau).
\end{align*}
This gives that the first-return map on $(0, \tau)$ has the combinatorics $
\begin{pmatrix}
I_1^-&\overline{I_1^0}&I_1^+\\
I_1^+&\overline{I_1^0}&I_1^-
\end{pmatrix},
$ with the lengths of its intervals of continuity $|I_1^-|=\tau-l_2, |I_1^0|=l_2+l_3-\tau, |I_1^+|=\tau-l_3$.

This first return map is completely periodic since the Nogueira-Rauzy induction for this map stops, and its Rauzy diagram is finite.\footnote{ This induction is the Rauzy-Nogueira induction for IETs with flips and was first introduced in \cite{N89}.This induction is defined in an analogous way to the standard Rauzu induction, by inducing on each step on the difference between the initial and losing intervals. For \emph{almost any} IET with flips the Rauzy-Nogueira induction stops, as proven by A. Nogueira. For more details, see for example \cite{Olga}.} Indeed, one has a following Rauzy diagram:

\begin{equation}\label{eq:comp_per_map}
\begin{pmatrix}
I_1^-&\overline{I_1^0}&I_1^+\\
I_1^+&\overline{I_1^0}&I_1^-
\end{pmatrix} \overset{|I_1^+|>|I_1^-|}{\underset{|I_1^-|>|I_1^+|}\rightleftarrows} 
\begin{pmatrix}
I_1^-&\overline{I_1^0}&I_1^+\\
I_1^+&I_1^-&\overline{I_1^0}
\end{pmatrix}
\overset{|I_1^0|>|I_1^+|}{\underset{}\rightarrow}
\begin{pmatrix}
I_1^-&\overline{I_1^+}&\overline{I_1^0}\\
\overline{I_1^+}&I_1^-&\overline{I_1^0}
\end{pmatrix}
\overset{|I_1^+|>|I_1^-|}{\underset{}\rightarrow}
\begin{pmatrix}
\overline{I_1^-}&\overline{I_1^+}&\overline{I_1^0}\\
\overline{I_1^-}&\overline{I_1^+}&\overline{I_1^0}
\end{pmatrix}.
\end{equation}

We do not give a full Rauzy diagram but only one of its parts,
since the diagram is symmetric with respect to the exchange of $I_1^-$ and $I_1^+$. After a finite number of steps of the Rauzy-Nogueira induction, one obtains a completely periodic map (indeed, a permutation on the right in \eqref{eq:comp_per_map} is completely periodic).\footnote{One can also explicitely calculate the set $\mathcal{P}_F$ in this case but we do not need it in the following.} This proves the point 1.

For the point 2., if $l_1>\frac{1}{2}$ and $\tau>1-l_1$ then we have $0 < l_1+\tau-1<\tau<l_1$. This means that the map $F$ has two $2$-periodic intervals $I_1^-:=(0, l_1+\tau-1)$ and $I_1^+:=(\tau, l_1)$, $I_1^- \cup I_1^+=K_1$. Then, the first return map on the interval $I_2 \cup I_3=(l_1, 1)$ is equal to $F^2$ and coincides with a rotation $R_{\kappa}$ with with $\kappa=\frac{l_3}{l_2+l_3}$.

Finally, for all the maps studied above the elements of $\mathcal{P}_F$ have the form $\{4n+2\ \mid n \in \N\}$ (except for the point 2. and $\kappa \in \Q$ that may induce periods of the form $4n, n \in \N$). The set $\mathcal{P}_F$ is always finite.
\end{proof}

\begin{remarkimp}\label{rem:INTEGRABILITY}
In terms of triangle tiling billiards, the maps with parameters described in Lemma \ref{lem:integrability} are \textbf{integrable}, i.e. the corresponding trajectories are either \emph{periodic} (correspond to periodic intervals) or \emph{linearly escaping} (correspond to the point 2. of the Lemma \ref{lem:integrability}). The point 2. corresponds to the case of obtuse triangle tilings: on any of such tilings one finds linearly escaping trajectories. The point 1. corresponds to the case when trajectories start far enough from the circumcenter and are periodic. For more on the notion of integrability for tiling billiards, see  [Section 5, \cite{Olga}].
\end{remarkimp}

\subsection{Renormalization process.}\label{subs:renormalization}
Now we are ready to define the renormalization process on the family $\CETthree$: we will do it for all the cases that were not covered by the previous pargraph.

\bigskip

\begin{theorem}\label{thm:renormalization_process}
Take a map $F=F_{\tau}^{l_1, l_2, l_3} \in \CETthree$ with $\tau \in [0, \frac{1}{2}]$. Let $\max\{l_j\}_{j=1}^3 \leq \frac{1}{2}$ and $\tau> \max\{l_j\}_{j=1}^3$. Define $x_j$ and $r$ via the relations \eqref{eq:l and x} and \eqref{eq:r_and_x}. Then the following holds.
\begin{itemize}
\item[1.] A map $T=F^2: \Sph^1 \rightarrow \Sph^1$ is a $6$-IET with intervals of continuity $I_j^{\pm}$ of lengths $|I_j^{\pm}|=\frac{x_j}{2} \pm r, j \in \mathcal{N}_{\Delta}$. Moreover, $I_j^+$ and $I_j^-$ are neigbouring in the preimage, and their images $T(I_j^+)$ and $T(I_j^-)$ are neighbouring in the image.
\item[2.] Suppose that $l_j= \min\{l_j\}_{j=1}^3$ for some $j \in \mathcal{N}_{\Delta}$.  
Consider the interval $S_j:=I_j^+ \cup I_j^-=:(s_j^-, s_j^+)$ and reglue its endpoints to obtain a circle $S_j / {s_j^-\sim s_j^+}$. Then a first return map on this circle is well-defined. Let ${R}_j F: \Sph^1 \rightarrow \Sph^1$ be its rescaling back to the unit circle. Then ${R}_j F \in \CETthree$ and its parameters  $(l_1', l_2', l_3', \tau') \in \Delta_2 \times [0, 1/2]$ are defined as follows: $(l_1', l_2', l_3')$ is th image of $(l_1, l_2, l_3)$ under the fully subtractive algorithm, and 
\begin{align*}
\tau'=\frac{1}{2}-r', \; \; r'=\frac{r}{|S_3|} \geq r.
\end{align*}
\item[3.] A map ${R}_j F$ has a $2$-periodic interval if and only if $l_j \geq \frac{1}{4}-\frac{r}{2}$. 
\end{itemize}
We call the interval $I_j$ \textbf{the interval in play}.
\end{theorem}

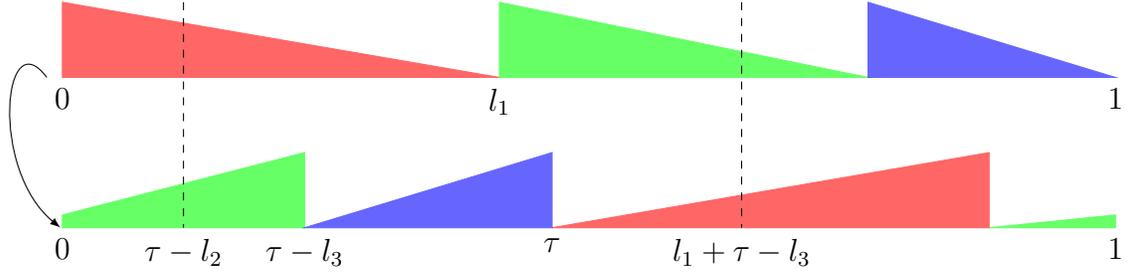
\begin{figure}
\centering
\begin{tikzpicture}[xscale=14]
\path[draw, fill=red,red, opacity=0.6] (0,0)--(0,1)--(0.415,0)--cycle;
\path[draw, fill=green, green, opacity=0.6] (.415,0)--(.415,1)--(0.765,0)--cycle;
\path[draw, fill=blue, blue, opacity=0.6] (.765,0)--(1,0)--(.765,1)--cycle;
\path[draw] (0,0) node[below] (1){$0$} --cycle;
\path[draw] (1,0) node[below]{$1$} --cycle;
\path[draw] (.415,0) node[below]{$l_1$} --cycle;

 \tikzset{
        arrow/.style={
            color=black,
            draw=black,
            -latex,
                font=\fontsize{12}{12}\selectfont},
        }

\path[draw, fill=red, red, opacity=0.6] (.465,0-2)--(.880,1-2)--(0.880,0-2)--cycle;
\path[draw, fill=green,green, opacity=0.6] (0.880,0-2)--(1,0-2)--(1,0.170-2)--cycle;
\path[draw, fill=green,green, opacity=0.6] (0,0-2)--(0,0.170-2)--(0.230,1-2)--(0.230,0-2)--cycle;
\path[draw, fill=blue,blue, opacity=0.6] (0.230,0-2)--(.465,0-2)--(.465,1-2)--cycle;
\path[draw] (0,-2) node[below] (2){$0$} --cycle;
\path[draw] (1,-2) node[below]{$1$} --cycle;
\path[draw] (.465,0-2) node[below]{$\tau$} --cycle;
\path[draw] (.230,0-2) node[below]{$\tau-l_3$} --cycle;
\draw[arrow](1) to [out=94,in=95]  (2);
\draw [dashed](0.115,1)--(0.115,-2);
\draw [dashed](0.645,1)--(0.645,-2);
\path[draw] (0.115,-2) node[below] (2){$\tau-l_2$} --cycle;
\path[draw] (0.645,-2) node[below] (2){$l_1+\tau-l_3$} --cycle;
\end{tikzpicture}
\caption[]{\emph{Interval} $S_3$ \emph{of the induction.}
Here $F=F_{\tau}^{l_1, l_2, l_3} \in \CETthree$ with the parameters satisfying the relations $l_3 < l_2 \leq l_1$ and $\tau \in \left(l_1, \frac{1}{2}\right]$. One step of renormalization gives a map ${R}_3 F$ which is a rescaled first return map on the interval $S_3$. The middlepoint of $S_3$ is equal to $\tau+l_1-\frac{1}{2}$ and coincides with a singularity $l_1$ and only if $\tau=\frac{1}{2}$.}\label{fig:illust_renormalization}
\end{figure}

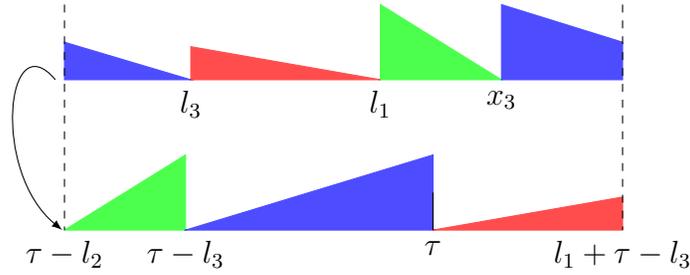
\begin{figure}
\centering
\begin{tikzpicture}[xscale=14]
\path[draw] (0.115,0) node[below] (1){} --cycle;
\path[draw] (0.235,0) node[below] {$l_3$} --cycle;
\path[draw] (0.53,0) node[below] {$x_3$} --cycle;
\path[draw] (0.115,-2) node[below] (2){} --cycle;
\path[draw, fill=blue,blue, opacity=0.7] (0.115,0)--(0.235,0)--(0.115,0.5)--cycle;
\path[draw, fill=red,red, opacity=0.7] (0.235,0)--(0.415,0)--(0.235,0.44)--cycle;
\path[draw, fill=green, green, opacity=0.7] (0.415,0)--(0.415,1)--(0.53,0)--cycle;
\path[draw, fill=blue, blue, opacity=0.7] (0.53,0)--(0.53,1)--(0.645,0.5)--(0.645,0)--cycle;
\path[draw] (0.415,0) node[below]{$l_1$} --cycle;

 \tikzset{
        arrow/.style={
            color=black,
            draw=black,
            -latex,
                font=\fontsize{12}{12}\selectfont},
        }

\path[draw] (0.23,0-2) node[below] {$\tau-l_3$} --cycle;
\path[draw] (0.465,0-2) node[below] {$\tau$} --cycle;
\path[draw] (0.115,-2) node[below] (2){$\tau-l_2$} --cycle;
\path[draw] (0.645,-2) node[below] {$l_1+\tau-l_3$} --cycle;
\path[draw, fill=green, green, opacity=0.7] (0.115,0-2)--(0.23,1-2)--(0.23,0-2)--cycle;
\path[draw, fill=blue,blue, opacity=0.7] (0.23,0-2)--(0.465,1-2)--(0.465,0-2)--cycle;
\path[draw, fill=red,red, opacity=0.7] (0.465,0-2)--(0.645,0-2)--(0.645,0.44-2)--cycle;
\draw[arrow](1) to [out=94,in=95]  (2);
\draw [dashed](0.115,1)--(0.115,-2);
\draw (0.465,0-2)--(0.465,0.5-2);
\draw [dashed](0.645,1)--(0.645,-2);
\end{tikzpicture}
\caption{\emph{First return map on} $S_3$ \emph{is  a fully flipped interval exchange transformation}.The intervals $J_a, J_b, J_c^1, J_c^2$ are intervals of continuity of such a map, and the dynamics is defined by equations \eqref{eq:cut_up} and \eqref{eq:cut-down}. By regluing the extremities of $S_3$, a singularity between $J_c^1$ and $J_c^2$ dissapears and a rescaled map ${R}_3 F \in \CETthree$.}
\label{fig:induced_map}
\end{figure}

\begin{proof}
The point 1. follows from the proof of Lemma \ref{lem:ARconnection}. As already mentionned before, the inequality $\tau>l_j$ is equivalent to the absence of $2$-periodic intervals for $F$.
 
In the following we suppose that $l_3=\min \{l_j\}_{j=1}^3$ or, equivalently,  $x_3= \max \{x_j\}_{j=1}^3$. Then $S_3=\left(\tau-l_2, l_1+\tau-l_3 \right)$ and we study the first return map on $S_3$, see Figure \ref{fig:illust_renormalization}.

Cut each of the intervals $I_3^+$ and $I_3^-$ into two subintervals by points $l_3$ and $x_3$ correspondingly. Then $I_3^+=J_3^2 \cup J_1$ and $I_3^-=J_2 \cup J_3^1$, where the intervals $J_1, J_2, J_3^1$ and $J_3^2$ are defined by
\begin{align*}
J_3^1:=\left(l_1+l_2-l_3, l_1-l_3+\tau\right), J_3^2:=\left(\tau-l_2, l_3\right),\\
J_1:=\left(l_3, l_1 \right),\\
J_2:=\left(l_1, l_1+l_2-l_3\right).
\end{align*}
We see that $|J_1|=\frac{x_3-x_1}{2}=l_1-l_3, |J_2|=\frac{x_3-x_2}{2}=l_2-l_3$ and $|J_3^1|+|J_3^2|=\left(\frac{x_2}{2}-r\right)+\left(\frac{x_1}{2}+r\right)=l_3$. Moreover, the interval $S_3$ is cut into four disjoint intervals in the following order:
\begin{equation}\label{eq:cut_up}
S_3=J_3^2 \sqcup J_1 \sqcup J_2 \sqcup J_3^1.
\end{equation}

One can easily see that $F(J_1) \cup F(J_2) \subset S_3$, and that $F(J_1)$ is put to the right end of $S_3$, and $F(J_2)$ is put to the left end of $S_3$ by the dynamics. 

For the intervals $J_3^1$ and $J_3^2$, one has the following chains of iterations:

\begin{align*}
J_3^1 \xmapsto{F} I_2^-  \xmapsto{F} \left(l_1+l_2, l_1+\tau\right) \xmapsto{F} \left(l_2, \tau\right) \subset S_3, \\
J_3^2 \xmapsto{F} I_1^+ \xmapsto{F} \left(l_1+\tau, 1\right) \xmapsto{F} \left(\tau-l_3, l_2\right) \subset S_3.
\end{align*}

This proves that the first return map on $S_3$ coincides with $F^3$ in restriction to $J_3^1 \cup J_3^2$, see Figure \ref{fig:induced_map}.

Finally, we conclude that the images of the four intervals $J_1, J_2, J_3^1, J_3^2$ under the first return map cover $S_3$ without intersection. Indeed, we have

\begin{equation}\label{eq:cut-down}
S_3=F(J_2) \sqcup F^3 (J_3^2) \sqcup F^3 (J_3^1) \sqcup F (J_1).
\end{equation}

Hence after regluing the ends of $S_3$ together, the first return map becomes a map in ${R}_3 F \in \CETthree$ with three intervals of continuity: the (rescaled) intervals $J_1, J_2$ and $J_3=J_3^1 \cup J_3^2$. The direct calculation shows that $\tau=\frac{\tau-l_3}{|S_3|}$. By writing out $\tau=\frac{l_1+l_2+l_3}{2}-r$ we conclude $\tau'=\frac{1}{2}-\frac{r}{|S_3|}$. Thus the point 2. is proven.

For the point 3. we see that $F (J_2) \cap J_2 =\emptyset$ and $ F (J_1) \cap J_1 =\emptyset$ since $\tau > l_j$. Finally, $F^3 (J_3^1) \cap J_3^1 \neq \emptyset$ is equivalent to the inequality 
$l_1+l_2-l_3 \leq \tau  \Leftrightarrow l_3 \geq \frac{1}{4}+\frac{r}{2}$.  Analogously, $F^3 (J_3^2) \cap J_3^2 \neq \emptyset$ is equivalent to the analogous inequality $\tau-l_3 \leq l_3 \Leftrightarrow  l_3 \geq \frac{1}{4}-\frac{r}{2}.$ By uniting these two inequalities, we finish the proof. 
\end{proof}

\smallskip
We now define the\textbf{ renormalization process on the family} $\CETthree$ as follows. Take any map $F \in \CETthree$ and let $k=0$, $F_0=F$.
If the conditions of Theorem \ref{thm:renormalization_process} do not hold (equivalently, conditions of Lemma \ref{lem:integrability} do hold) for $F$ , we say that the \textbf{renormalization process stops for the map} $F$. If these conditions do hold, that defines the index $t_1 \in \mathcal{N}_{\Delta}$ of the interval in play and one defines $R_{t_1} F \in \CETthree$. 

Then, one continues by recurrence. On the $k$-th step of the renormalization process (it if is defined), one obtains an interval exchange map $F_k \in \CETthree$ defined by 
\begin{equation}\label{eq:RENORMALIZATION_DEF}
F_k={R}_{t_k} \circ \ldots \circ {R}_{t_1} F.
\end{equation}

Here $\{t_k\} \in \mathcal{N}_{\Delta}^{\N}$ is a sequence of indices corresponding to the intervals in play. 

Define $\boldsymbol{\lambda}:=(l_1, l_2, l_3, \tau) \in \Delta_2 \times [0, \frac{1}{2}]$ as a vector of parameters for any map $F \in \CETthree$. Then we denote by  $\{\boldsymbol{\lambda}^{(k)}\}_{k \in \N}$ a sequence of such vectors corresponding to the maps $F_k$. Here
$\boldsymbol{\lambda}^{(k)}=(l_1^{(k)}, l_2^{(k)}, l_3^{(k)}, \tau^{(k)}) \in \Delta_2 \times [0, 1/2]$.

The corresponding vectors $(x_1, x_2, x_3, r)$ are also defined in an analogous manner via  \eqref{eq:l and x} and \eqref{eq:r_and_x}.

We denote by $S^{(k)} \subset \Sph^1$ a set of definition of $F_{k}$, considered as a subset of the initial circle $S^{(0)}$, for any $k \in \N^*$. Obviously, the lengths $S^{(k)}$ diminish along the renormalization process since $S^{(k)} \subset S^{(k-1)}$.

\begin{remarkimp}\label{rem:onestep}
From  the proof of Theorem \ref{thm:renormalization_process} follows that one step (for example, $F \mapsto {R}_3 F$) of the renormalization process corresponds to one step of the fully subtractive algorithm for the triple $(l_1, l_2, l_3) \in \Delta_2$:
\begin{equation}\label{eq:l_jchange}
\left[l_1^{(1)}: l_2^{(1)}: l_3^{(1)}\right]=[l_1-l_3: l_2-l_3: l_3].
\end{equation}

The renormalization process does not depend on the parameter $\tau$ (although the moment it stops, does depend on $\tau$, see Theorem \ref{thm:renormalization_process}). In restriction to the coordinates $x_j$,  the map \eqref{eq:l_jchange} is the Rauzy subtractive algorithm: 
\begin{equation*}
[x_1: x_2: x_3] \mapsto [x_1': x_2': x_3]=[x_1:x_2: x_3-x_1-x_2].
\end{equation*}

The fully subtractive algoritm is defined for all triples of $l_j$, and one of the lengths can be bigger that $1/2$. Hence the Rauzy subtractive algorithm can be expanded to any triple $(x_1, x_2, x_3)$ with $x_j \in [-1,1]$, not necessarily positive, and it always continues with the index $j$ in play for $x_j=\max\{x_j\}_{j=1}^3$. 

Define the simplex $\Delta_2^{\pm}$ as a convex hull of the points $(1,1,-1), (1,-1,1)$ and $(-1,1,1)$. Then the Rauzy gasket is a part of $\Delta_2^{\pm}$ on which the fully subtractive algorithm is chaotic, and it is the complement of the three basins of attraction. This idea has been formulated in \cite{AS13} by P. Arnoux and S. Starosta, see in particular their Figure 10. In the following we interpret the renormalization process in terms of triangle tiling billiards. Indeed, it can be seen as acting on the space of orbits of periodic triangle tiling billiards (via \eqref{eq:correspondance}), moving from one tiling to another, with the set of right triangles being invariant. 
\end{remarkimp}

\subsection{Minimality in the family $\CETthree$.}\label{subs:proof_of_minimality}
The goal of this paragraph is to give a new proof of

\begin{theorem}[\cite{Olga}]\label{thm:minimality}
A map $F^{l_1, l_2, l_3}_{\tau} \in \CETthree$ is minimal if and only if $\tau=\frac{1}{2}$ and $(x_1,x_2,x_3) \in \mathcal{R}$.
\end{theorem}

The proof of this Theorem that we give with P. Hubert in \cite{Olga} was based first, on Theorem \ref{thm:Arnoux-Rauzy} by Arnoux-Rauzy and second, on a "miracle". By explicitely studying the Rauzy graphs of $4$-IET with flips,  we have proved the existence of some invariant of these graphs that implied  the hyperbolicity of the Rauzy-Nogueira induction in the neighbourhood of the repelling hyperplane $\{\tau=\frac{1}{2}\}$. Although, we think that the standard Rauzy-Nogueira induction is not the most appropriate tool to study the families of fully flipped maps. Indeed, already after one step of this induction the induced map is not anymore a fully flipped map. The renormalization process we propose in paragraph \ref{subs:renormalization} is better adapted to such families, and corresponding Rauzy graphs are much smaller. For the family $\CETthree$, such graph is one vertex. 

\smallskip

Here is a standard

\begin{lemma}\label{lem:standard}
Consider a map $F \in \CETthree$ and the renormalization process for this map. Then a map $F$ is minimal if and only if the renormalization process is infinite, and $\lim_{k \rightarrow \infty} |S^{(k)}|=0$.
 \end{lemma}

Now we are ready to prove Theorem \ref{thm:minimality}.
\begin{proof}
Take a map $F \in \CETthree$ with a vector of parameters defined by $\boldsymbol{\lambda}$. If the renormalization process reaches the $k$-th step, then by Theorem \ref{thm:renormalization_process} and Remark \ref{rem:onestep}, for the map $F_k \in \CETthree, k \in \N^*$ defined by \eqref{eq:RENORMALIZATION_DEF} we have 

\begin{equation*}
\boldsymbol{\lambda}^{(k)}=A_{t_k} \boldsymbol{\lambda}^{(k-1)},
\end{equation*}
where $t_k \in \mathcal{N}_{\Delta}$ are the indices of intervals in play and the matrices $A_j, j \in \mathcal{N}_{\Delta}$ are defined explicitely by
\begin{equation*}
A_1:=\begin{pmatrix}
1&0&0&0\\
-1&1&0&0\\
-1&0&1&0\\
-1&0&0&1
\end{pmatrix}, \; \; \; \; 
A_2:=\begin{pmatrix}
1&-1&0&0\\
0&1&0&0\\
0&-1&1&0\\
0&-1&0&1
\end{pmatrix}, \; \; \; \; 
A_3:=\begin{pmatrix}
1&0&-1&0\\
0&1&-1&0\\
0&0&1&0\\
0&0&-1&1
\end{pmatrix}.
\end{equation*}

Define now $B_j:=\left(A_j^{-1}\right)^T, j \in \mathcal{N}_{\Delta}$. Then 
\begin{equation*}
B_1=\begin{pmatrix}
1&1&1&1\\
0&1&0&0\\
0&0&1&0\\
0&0&0&1
\end{pmatrix}, \; \; \; \; 
B_2=\begin{pmatrix}
1&0&0&0\\
1&1&1&1\\
0&0&1&0\\
0&0&0&1
\end{pmatrix}, \; \; \; \; 
B_3=\begin{pmatrix}
1&0&0&0\\
0&1&-1&0\\
1&1&1&1\\
0&0&0&1
\end{pmatrix}.
\end{equation*}

A map $F \in \CETthreehalf$ if and only if $(\boldsymbol{\lambda}, \boldsymbol{v}^{\perp})=0$ for $\boldsymbol{v}^{\perp}:=(1,1,1,-2)$. Moreover, the vector $\boldsymbol{v}^{\perp}$ is  \emph{invariant} for all three matrices $B_j, j \in \mathcal{N}_{\Delta}$, i.e. $B_j \boldsymbol{v}^{\perp}=\boldsymbol{v}^{\perp}$. This implies

\begin{multline*}
\left(\boldsymbol{\lambda}^{(0)}, v^{\perp}
\right)=\left(
A_{t_1}^{-1} \cdot \ldots \cdot A_{t_k}^{-1} \boldsymbol{\lambda}^{(k)}, v^{\perp}
\right)= 
\left(
\boldsymbol{\lambda}^{(k)}, B_{t_k} \cdot \ldots B_{t_1} v^{\perp} 
\right)= \left(
\boldsymbol{\lambda}^{(k)}, v^{\perp} 
\right)
= \left|S^{(k)}\right|-2 \tau^{(k)}  \left|S^{(k)}\right|.
\end{multline*}

This calculation gives that 
\begin{equation*}
 \tau^{(k)}=\frac{1}{2}-\frac{\left(
\boldsymbol{\lambda}^{(k)}, v^{\perp} 
\right)}{ \left|S^{(k)}\right|}.
\end{equation*}

We see from here that $r^{(k)}=\frac{r^{(0)}}{|S^{(k)}|}$. Suppose now that $F$ is minimal. Hence necessarily by Lemma \ref{lem:integrability}, $F$ satisfies (infinitely) the conditions of Theorem \ref{thm:renormalization_process}. Then, by Lemma \ref{lem:standard}, one obtains that if $\left(
\boldsymbol{\lambda}^{(k)}, v^{\perp} 
\right) \neq 0$, then $r^{(k)}$ tends to $-\infty$ while $k \rightarrow \infty$ which is impossible since $r^{(k)} \in [0, \frac{1}{2}]$. Hence necessarily  $\left(
\boldsymbol{\lambda}^{(k)}, v^{\perp} 
\right)=0$ and $\tau^{(0)}= \tau^{(k)}=\frac{1}{2}$. 
Then, for $F \in \CETthreehalf$ to be minimal, by Theorem \ref{thm:renormalization_process}, for every $k \in \N^*$ the following inequality should hold:
\begin{equation}\label{eq:inequalityformin}
l_{t_k}^{(k)} <\frac{1}{4}-\frac{r^{(k)}}{2}.
\end{equation}
Since $r^{(k)}=0$, this implies $l_{t_k}^{(k)} <\frac{1}{4}$ for all $k \in \N^*$. In terms of parameters $x_{t_k}$ these are equivalent to $x_{t_k}^{(k)}>1-x_{t_k}^{(k)}$ which, by definition gives $(x_1, x_2, x_3) \in \mathcal{R}$. 

To prove the inverse statement, if  $F \in \CETthreehalf$ with the parameters $(x_1, x_2, x_3) \in \mathcal{R}$, one can directly reference the result by Arnoux and Rauzy, see Theorem \ref{thm:Arnoux-Rauzy}. Or, alternatively, we see that the renormalization process is defined infinitely and $\left|S^{(k)}\right| \rightarrow 0$. This proves the minimality of $F$ by Lemma \ref{lem:standard}.

\end{proof}


\section{Classification of dynamics of triangle tiling billiards.}\label{sec:tiling billiards}

We use the renormalization process $R$ on the family $\CETthree$ and tiling billiard foliations in order to completely describe the dynamics of triangle tiling billiards.

\subsection{Vocabularly: tiling billiards and the family $\CETthree$.}\label{subs:vocabularly}

We now make the connection between triangle tiling billiards and maps in the family $\CETthree$ discussed above explicit. We provide a vocabularly between these, which was for the most part established in \cite{BDFI18}. We add to it the two last lines.

\bigskip

\begin{tabular}{|c|c|}
  \hline
\makecell{Triangle tiling billiards} &\makecell{$\CETthree$ (\emph{and translation})}\\
  \hline
\makecell{angles of a tile \\ $\alpha, \beta, \gamma$}&\makecell{parameters $(l_1, l_2, l_3)\in \Delta_2$ (\emph{via rescaling }\eqref{eq:correspondance})}\\
  \hline
 \makecell{oriented distance $d$ from a segment \\ of a trajectory to the circumcenter of a tile}&\makecell{$\tau \in \Sph^1 $ (\emph{via} $d=\cos \pi \tau$, \emph{see} [Proposition 1, \cite{Olga}])}\\  \hline
\makecell{relative position of a tile with respect \\ to the folded trajectory} &\makecell{$p \in \Sph^1$ (\emph{via folding }$\mathcal{F}$)}\\  
\hline
\makecell{starting tile $\theta_0$\\ of fixed orientation}&\makecell{$p_0 \in \Sph^1$ (\emph{via folding}, $p_0 \in \mathcal{F}(\theta_0) \cap \mathcal{C}$)} \\
\hline
\makecell{the set $V$ of vertices\\ and a corresponding set $\mathcal{F}(V)$}&\makecell{$\mathcal{C}(p_0):=\{n \alpha + m \beta + p_0, n,m \in \Z\}$ \\ (\emph{by identification} $\mathcal{C} \simeq\Sph^1$)}\\
\hline
\makecell{ray foliation $\mathcal{R}_{p_0}$ with $p_0=\mathcal{F}(v_0)$}&\makecell{action of a subfamily with fixed $(l_1, l_2, l_3) \in \Delta_2$\\ and varying $\tau$, on the subset $\mathcal{C}(p_0) \subset \Sph^1$}\\
\hline
\makecell{parallel foliation $\mathcal{R}_{\tau}, \tau \in \Sph^1$}&\makecell{action of a subfamily with fixed $(l_1, l_2, l_3) \in \Delta_2$\\ and varying $\tau(\varepsilon)$, on the set $\mathcal{C}(p(\varepsilon)),$ \\ here $\tau(\varepsilon)=\tau_0+2 \varepsilon$ and $p(\varepsilon)=p_0+\varepsilon$} \\
\hline
\end{tabular}

\bigskip

All of these connections follow from \cite{BDFI18} and the discussions above. The only calculation is that of the parameters $\tau(\varepsilon)$ and $p(\varepsilon)$ in the last line of the table. It follows from the definition of the coordinate $\tau$ in \cite{BDFI18} in a straightforward way.

\subsection{Symbolic dynamics of triangle tiling billiards.}\label{subs:symbolic_dynamics_proofs}

First, as a corollary of Theorem \ref{thm:renormalization_process}, we  give a simple proof of points 4. and 5. of Theorem \ref{thm:triangle_tiling_billiards_info}. We remind the reader that initially 4. was announced as a $4n+2$ Conjecture in \cite{BDFI18} and a first attempt of a proof was given in \cite{Olga}.

\begin{theorem}[$4n+2$ Conjecture]\label{thm:one-more-time}
Points 4. and 5. of Theorem \ref{thm:triangle_tiling_billiards_info} hold.
\end{theorem}
\begin{proof}
It is sufficient to prove the statements 4. and 5. of Theorem \ref{thm:triangle_tiling_billiards_info} for the symbolic dynamics of any map $F \in \CETthree$.  Indeed, if some triangle tiling billiard trajectory is periodic, there exists a \emph{periodic interval} for a corresponding map in $\CETthree$.\footnote{The inverse is not true since periodic intervals of $F \in \CETthree$ may also define drift-periodic orbits of tiling billiards.}In this case, the map $F$ is not minimal, and hence, the renormalization process stops for $F$. Then, by Theorem \ref{thm:renormalization_process} and Lemma \ref{lem:integrability}, the periodic interval is necessarily flipped on itself or comes as a periodic orbit of a rational rotation $R_{\kappa}$ (see point 2. in Lemma \ref{lem:integrability}). But in the latter case, a map $F$ can be perturbed by a slight change of parameters $(l_1, l_2, l_3) \in \Delta_2$ in order for $\kappa=\frac{l_3}{l_2+l_3} \notin \Q$. Then, the corresponding periodic interval disappears which is not the case for periodic orbits of triangle tiling billiards, see point 3. in Theorem \ref{thm:triangle_tiling_billiards_info}. Indeed, this second case defines drift-periodic orbits.
\end{proof}

\smallskip

One may give a much more precise description of symbolic dynamics of triangle tiling billiards than that of Theorem \ref{thm:one-more-time} with the help of the following

\begin{proposition}\label{prop:SYMBOLIC}
Consider one step of the renormalization process on $\CETthree$. Then for any orbit of the induced map $R_j F, j \in \mathcal{N}_{\Delta}$, the symbolic code of a corresponding orbit of $F$ is obtained via the substitution $\sigma_j$, where 
\begin{align}
\sigma_1: \; \;  \left\{
\begin{array}{cc}
a \mapsto bca, \textit{if a precedent symbol was not}\; \;  b, \\
a \mapsto cba, \textit{if a precedent symbol was not}\; \;  c, \\
b \mapsto b,\\
c \mapsto c. 
\end{array}
\right.;
 \nonumber
\\
\sigma_2: \; \;  \left\{
\begin{array}{cc}
a \mapsto a, \\
b \mapsto acb, \textit{if a precedent symbol was not}\; \;  a, \\
b \mapsto cab, \textit{if a precedent symbol was not}\; \;  c, \\
c \mapsto c
\end{array}
\right.;
\label{eq:sigmas}\\
\sigma_3: \; \;  \left\{
\begin{array}{cc}
a \mapsto a,\\
b \mapsto b,\\
c \mapsto bac, \textit{if a precedent symbol was not}\; \;  b,\\
c \mapsto abc, \textit{if a precedent symbol was not}\; \;  a. 
\end{array}
\right.  \nonumber
\end{align}
Consequently, if $F_k$ is defined by \eqref{eq:RENORMALIZATION_DEF} then the symbolic code of any orbit of $F$ is deduced from a symbolic code of a corresponding orbit of $F_k$ by applying to it  a substitution $\sigma_{t_1} \circ \ldots \sigma_{t_k}$. 
\end{proposition}
\begin{proof}
The proof follows from the proof of Theorem \ref{thm:renormalization_process}, and we use the notations coming from there. Suppose that the induced map is $R_3 F \in \CETthree$ ($j=3$). Then any orbit of the map $F$ passes by a Poincaré section $S_3$ and has a corresponding orbit in $R_3 F$. Moreover, for any point $p \in J_1 \cup J_2$, its $F$- and $R_3 F$-orbits coincide, hence $\sigma_3(a)=a, \sigma_3(b)=b$. Finally, $J_3^1 \subset I_2, F(J_3^1) \subset I_1, F^2(J_3^1) \subset I_3$ and $J_3^2  \subset I_1, F(J_3^2) \subset I_2, F^2(J_3^2) \subset I_3$. Since both $J_3^1$ and $J_3^2$ both have the symbolic code $c$, $\sigma_3$ is defined conditionally. This finishes the proof.
\end{proof}

\subsection{Complete description of the dynamics of triangle tiling billiards.}

Now we are ready to prove Theorem \ref{thm:complete_classification} which is a much stronger version of Theorem \ref{thm:exceptional-zero measure} proven in \cite{Olga}. 

\begin{proof}
First, via the relations \eqref{eq:correspondance} and \eqref{eq:l and x}, we have $\rho_{\Delta}=(x_1, x_2, x_3)$. We now study the dynamics of a subfamily of maps in $\CETthree$ with varying $\tau$ and fixed $(x_1, x_2, x_3)$, which corresponds to the dynamics of a tiling billiard on a fixed tiling. Take a map $F$ in this family.

\smallskip

\textbf{Step 1.} First of all, if the renormalization process stops for $F$, then $F$ is integrable (see Remark \ref{rem:INTEGRABILITY}), i.e. all the corresponding tiling billiard trajectories are either periodic or linearly escaping. Indeed, we have that $\tau^{(k)} \leq \max\{l_j^{(k)}\}_{j=1}^3$ or $l_j^{(k)}>\frac{1}{2}$  for some $j \in \mathcal{N}_{\Delta}$. In both cases, the dynamics of the map $F_k$ is integrable, and hence is that of $F$. 

If $\rho_{\Delta} \notin \mathcal{R}$, the renormalization process will necessarily stop, see Remark \ref{rem:onestep} and the proof of Theorem \ref{thm:minimality}.

\smallskip

\textbf{Step 2.} Take $\rho_{\Delta} \notin \mathcal{R}$. The linearly escaping behaviour exists on a corresponding tiling if and only 
if for some $k \in \N^*$, the map $F_k$ verifies the conditions of point 2. in Lemma \ref{lem:integrability}. An additional calculation shows that it is indeed true for all $\rho_{\Delta} \in \mathcal{R} \setminus \mathcal{E}$. The argument goes as follows.

Suppose that there exists some $k \in \N^*$ such that $l_i^{(k+1)} \neq 0$ for all $i \in \mathcal{N}_{\Delta}$ and 
\begin{equation}\label{eq:obtuse_becoming}
l_j^{(k+1)}>\frac{1}{2} |S^{(k+1)}|, \; \; \textit{and} \; \; \forall m<k \; \; \max\{l_i^{(m)}\}_{i=1}^3 \in [0, \frac{1}{2}).
\end{equation} 

In the above relation, necessarily $j=t_k$. Indeed, since
$\max \{l_j^{(k)}\}_{j=1}^3<\frac{1}{2}$ for $j \neq t_k$, we have
\begin{equation*}
l_j^{(k)}-l_{t_k}^{(k)}<\frac{1}{2}\left(
1-2l_{t_k}^{(k)}
\right) 
\end{equation*} which is equivalent to $ l_j^{(k+1)}<\frac{1}{2}$. 
Although, it is possible that \eqref{eq:obtuse_becoming} holds for $j=t_k$. This condition can be rewritten as
\begin{equation}
l_{t_k}^{(k+1)}>\frac{1}{2}|S^{(k+1)}| \Longleftrightarrow 
l_{t_k}^{(k)}>\frac{1}{2}(|S^{(k)}|-2l_{t_k}^{(k)}) \Longleftrightarrow 
l_{t_k}^{(k)}>\frac{1}{4} |S^{(k)}|.
\end{equation}

But the last inequaity holds for all $\rho_{\Delta} \notin \mathcal{R}_{\Delta}$ for some $k \in \N^*$. This implies that if  $l_i^{(k+1)} \neq 0$ for all $i \in \mathcal{N}_{\Delta}$ then the linearly escaping behavior does occur on the triangle tiling defined by $\rho_{\Delta}$. Indeed, it suffices to take $\tau^{(k+1)}=\tau^{(0)}=\frac{1}{2}$, by Lemma \ref{lem:integrability}.

The case which is left is to study is what happens if for some  $i \neq t_k,  l_i^{(k)}=l_{t_k}^{(k)}$ (and hence $ l_i^{(k+1)}=0$). First, $l^{(k)}_1=l^{(k)}_2=l^{(k)}_3=\frac{1}{3}$ is equivalent to $\rho \in \mathcal{E}_{\Delta}$. Since the dynamics on the equilateral triangle tiling is $6$-periodic, then for any $\rho_{\Delta} \in \mathcal{E}_{\Delta}$, by Theorem \ref{thm:renormalization_process}, all of the tiling billiard trajectories on the tiling defined by $\rho_{\Delta}$, are periodic.

Otherwise, if there exists only one $j \neq k$ such that of $l^{(k)}_j=l^{(k)}_{t_k}$ coincide, without loss of generality we can suppose $t_k=3$ and $j=2$. Then $l^{(k)}_3=l^{(k)}_2 \in [\frac{1}{4}, \frac{1}{3})$ and $l^{(k)}_1 \in (\frac{1}{3}, \frac{1}{2}]$. Take $\tau^{(0)}=\frac{1}{2}$, then $\tau^{(k)}=\frac{1}{2}$. Then a map $F_k$ is explicitely verified to have two types of orbits: fully flipped intervals of periods $6$ (corresponding to periodic orbits) and a periodic interval of period $4$ which corresponds to a periodic linear drift, see Figure \ref{fig:drifty}.\footnote{Our argument also shows that $4$ is the shortest period of the drift behaviour in a triangle tiling billiard.} This implies that $F$ has necessarily drift periodic orbits.

\smallskip

\textbf{Step 3.} If $\rho_{\Delta} \in \mathcal{R}$ and $\tau \neq \frac{1}{2}$, all corresponding triangle tiling billiard orbits are periodic by Lemma \ref{lem:integrability}. Indeed, the renormalizaiton  stops at some step $k \in \N^*$ and $\max\{l^{(k)}_j\}<\frac{1}{2}$ are all smaller than $\frac{1}{2}$. For $\tau=\frac{1}{2}$, $F$ is minimal by Theorem \ref{thm:minimality}, and the corresponding trajectories escape.\footnote{In Theorem \ref{thm:exceptional_trajectories} we show that their escape is non-linear.} The inverse is true as well: escaping trajectories exist only for $\tau=\frac{1}{2}$.

\smallskip
\textbf{Step 4.} Finally, as already shown in \cite{BDFI18}, drift-periodic behaviour only occurs if $(l_1, l_2, l_3) \in \Q^3$. This also follows obviously from renormalization. Moreover, the arguments above show that for any tiling such that $(l_1, l_2, l_3) \in \Q^3 \setminus \mathcal{E}$ the drift-periodic trajectories indeed exist, and for $(l_1, l_2, l_3) \in \mathcal{E}$ they do not.

\smallskip

\textbf{Step 5.} The statements about symbolic dynamics follow from Lemma \ref{lem:integrability} and Theorem \ref{thm:renormalization_process}. Indeed, for a tiling with $\rho_{\Delta} \in \mathcal{R}$ the set of possible obtained trees $\{G_{\Delta}^{\delta}\}$ with $\delta$ - periodic trajectories, is infinite. Indeed, there exists a sequence of periodic trajectories with monotonously growing periods by renormalization. The set $\{G_{\Delta}^{\delta}\}$ is although countable.\footnote{This set can be explicitely calculated via the substitutions $\sigma_j, j \in \mathcal{N}_{\Delta}$ defined in Proposition \ref{prop:SYMBOLIC}.} Then, we show that the number of possible periodic dynamical behaviours (and hence, contoured trees, by Theorem \ref{thm:main}) is finite on any tiling except that with $\rho_{\Delta} \in \mathcal{R}$. If $\rho_{\Delta} \in \mathcal{E}$ this is, indeed, true, since the renormalization process defines the list of possible periodic trajectories uniquely, from one $6$-periodic trajectory. 

Then, if  $\rho_{\Delta} \notin \mathcal{E}$, the renormalization process stops at some obtuse triangle tiling on the step $k$. On this tiling, for $\tau=\frac{1}{2}$, one obtains zero possible periodic behaviours since the corresponding map goes into the point 2. of Lemma \ref{lem:integrability} (only linear escape).\footnote{Then, by renormalization for $\tau=\frac{1}{2}$, there always is linear escape in obtuse tilings.} For smaller $\tau$, the point $1.$ of Lemma \ref{lem:integrability} applies. The only thing that one now needs to prove that in this case, the periods of all possible symbolic codes of trajectories are bounded. This is true since for 
$\tau^{(k)}=1-\max\{l_j^{(k)}\}$ all the trajectories are periodic and the number of their combinatorial behaviors is bounded. The other combinatorial behaviors are obtained by contraction of flowers inside these trajectories, hence one obtains a finite number of trees. 

Finally, the statement about symbolic dynamics of linear escaping trajectories follows directly from point 2. in Lemma \ref{lem:integrability} and Proposition \ref{prop:SYMBOLIC}.
\end{proof}

\textbf{Question.} Given a triangle tiling, what is a list of possible trees that billiard trajectories on this tiling contour?  Theorem \ref{thm:renormalization_process} and Proposition \ref{prop:SYMBOLIC} above give an algorithm to compute the symbolic behavior of trjaectories and hence, the corresponding trees. But we wonder if a tree can be calculated in a more direct way. 

\begin{remark}
The set $\mathcal{E}$ is the set of preimages of a point  $[1:1:1] \in \Delta_2$ under the fully subtractive algorithm. Here is a list of preimages up to level $3$. 

\smallskip
   \begin{tikzpicture}
      [
	 level 2/.style={sibling distance=10em},
 	 level 3/.style={sibling distance=5em},      
       level distance=6em,
      every node/.style={shape=rectangle,draw}, grow=right]
      \node{$[1:1:1]$}
      child{node{$[1:2:2]$}
			child{node{$[2:3:4]$}
								child{node{$[4:6:7]$}}      	
								child{node{$[3:5:7]$}}      	
									child{node{$[2:5:6]$}}      		
			}      
      		child{node{$[1:3:3]$}
						child{node{$[3:4:6]$}}      		
						child{node{$[1:4:4]$}}      		
      		}
      };
   \end{tikzpicture}
   
\smallskip

All the trajectories on corresponding tilings are periodic. For example, a point $[1:2:2]$ corresponds to a tiling by triangles with angles $36^{\circ}, 72^{\circ}, 72^{\circ}$ and all billiard trajectories in it are periodic with periods $6$ or $10$. The question whether the equilateral triangle tiling is the only tiling permitting \emph{only} periodic trajectories was initially asked by Serge Troubetzkoy. Theorem \ref{thm:complete_classification} gives a negative answer to it.
\end{remark}

\begin{figure}
\centering
\includegraphics[scale=0.27]{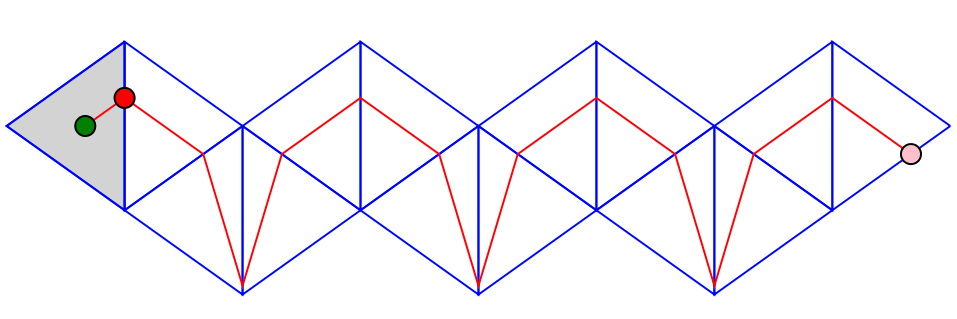}
\caption{A drift-periodic trajectory of period $4$, with for $\gamma=\beta =  50^{\circ}, \alpha = 80^{\circ} $.}\label{fig:drifty}
\end{figure}

\section{Arithmetic orbits of Arnoux-Rauzy surfaces and exceptional trajectories. }\label{sec:arithmetic_orbits_and_exceptional_trajectories}
While interested in \cite{HW18} in the dynamics of real-rel leafes of Arnoux-Yoccoz surfaces, P. Hooper and B. Weiss conjectured the convergence of the arithmetic orbits of the Arnoux-Yoccoz map to the Rauzy fractal, under reparametrization. Subsequently, P. Baird-Smith, D. Davis, E. Fromm and S. Iyer, following the connection between the arithmetic orbits and trajectories of tiling billiards they have discovered, restated the Hooper-Weiss Conjecture in terms of triangle tiling billiards. In this Section we prove this Conjecture. 

\subsection{Exceptional trajectories pass by all tiles.}\label{subs:exceptional_trajectories_go_everywhere}
We are especially interested in the exceptional trajectories of triangle tiling billiards since they are connected to arithmetic orbits of the Arnoux-Rauzy maps, see Section \ref{sec:arithmetic_orbits}. We remind our reader that by definition, the \textbf{exceptional trajectories } are those that are defined in the triangle tilings with $\rho_{\Delta} \in \mathcal{R}$ and pass through the circumcenters of crossed tiles. 

\begin{theorem}\label{thm:exceptional_trajectories}
For any $\rho_{\Delta} \in \mathcal{R}$ and any tiling billiard trajectory $\delta$ on a corresponding tiling passing by a circumcenter of the tile $\theta_0$, the following holds:
\begin{itemize}
\item[1.] if $\delta$ doesn't pass by any vertex of a tiling, then it passes by the interiors of \textbf{all} tiles.
\item[2.] if $\delta$ passes by some vertex $v \in V$ (is a singular ray) there exist $5$ additional singular rays in a corresponding flower such that the union of these six rays passes by \textbf{all} tiles, and this union doesn't pass by any other vertex.
\end{itemize}
\end{theorem}

\begin{proof}
First, for any $\rho_{\Delta} \in \mathcal{R}$, the corresponding triangles are acute. Consider a folding map $\mathcal{F}=\mathcal{F}(\theta_0)$. Let $l$ be a chord in a bellow such that $\mathcal{F}(\delta) \subset l$.

Suppose that $\delta$ doesn't pass by any singularity in a tiling. This implies $l \cap \mathcal{F}(V)=\emptyset$. Suppose first that $\delta$ doesn't pass by all of the triangles. Hence there exists some tile $\theta$ in a tiling and its edge $e$ such that $\delta \cap \theta \neq \emptyset$ and $\delta \cap \theta^e=\emptyset$. Consider a trajectory $\delta'$ passing by a circumcenter of $\theta^e$ in the same parallel foliation $\mathcal{P}^{\delta}$. Then $\delta' \neq \delta$ and $\delta \cap e = \delta' \cap e = \emptyset$.

Consider now two singular segments of the foliation $\mathcal{P}^{\delta}$ in the tiles $\theta$ and $\theta^e$. One can easily see from the folding that the only way these segments may behave is to pass by the \emph{same} vertex $v \in e$. Then, the corresponding singular trajectories are periodic by Theorem \ref{thm:complete_classification} and have to coincide since $\delta$ and $\delta'$ escape. We denote a corresponding periodic petal by $\delta_{\tau_1}$, see Figure \ref{fig:explanation_escaping}. Now consider a family $\{\delta_{\tau}\}_{\tau \in [\tau_1, 1/2]}$ of trajectories starting by the segments in $\theta$. Here $\delta_{\frac{1}{2}}=\delta$. Analogously to the above argument, the trajectory $\delta_{\tau}$ is periodic and passes by $\theta^e$ for any $\tau \neq \frac{1}{2}$(since $\delta$ and $\delta'$ are escaping and belong to the same foliation). Moreover, we see that $\Omega_{\delta_{\tau_-}} \subset \Omega_{\delta_{\tau_+}}$ for any $\tau_-, \tau_+ \in  [\tau_1, 1/2]$ such that $\tau_-<\tau_+$. 

Hence, by passing to the limit, the trajectories $\delta$ and $\delta'$ can be both approached as a set of nested trajectories $\{\delta_{\tau}\}$ with growing $\tau, \tau \rightarrow \frac{1}{2}$. Hence $\delta \cap \delta'\neq \emptyset$. If $\delta$ is non-singular, then $\delta=\delta'$ and $\delta=\lim_{\tau \rightarrow \frac{1}{2}} \delta_{\tau}$ and $\delta$ passes by all the triangles.

Otherwise, if $\delta \cap \delta' \neq \emptyset$ then necessarily $\delta \cap \delta' = \{v\}, v\in V$ and $\delta$ and $\delta'$ are singular rays in some unbounded separatrix flower. Then the parallel foliaiton $\mathcal{P}^{\delta}$ has $6$ singular rays going out in the tiles neighbouring to $v \in V$ since all the tiles are acute and the rays pass by a vertex and a circumcenter. Analogously to previous arguments, each of the sectors defined by these rays is foliated by sequences of periodic orbits with growing periods.  Each ray separately spirals non-linearly to infinity (in positive or negative time).

Finally, a singular trajectory $\delta$ passing by a curcumcenter of a tile can't pass by two vertices of the tiling since there are no rational relationships between the angles of the tile with $\rho_{\Delta} \in \mathcal{R}$.
\end{proof}

\begin{figure}
\centering
\includegraphics[scale=0.7]{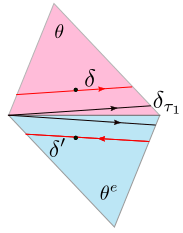}
\caption{Two neighbouring tiles $\theta$ and $\theta^e$ and trajectories $\delta$ and $\delta'$ passing by circumcenters of the tiles. The trajectory $\delta_{\tau_1}$ is a periodic loop containing $e$. }
\label{fig:explanation_escaping}
\end{figure}

Obviously, a trajectory passing by all points can't be linearly escaping.
Hence the Theorem \ref{thm:exceptional_trajectories} implies that all of the exceptional trajectories (singular and non-singular) are non-linearly escaping which proves our conjecture with P. Hubert from \cite{Olga}.

\smallskip

It can be interesting in the future to study the growing fractal forms to which the exceptional trajectories converge after reparametrization. We do it in the following for the family of exceptional trajectories corresponding to the Arnoux-Yoccoz map.

\subsection{A missing link: the Arnoux-Yoccoz map and the Rauzy fractal.}\label{subs:Arnoux-Yoccoz_Rauzy}

Consider the Arnoux-Yoccoz map $T^{\aaa} \in \AR$ defined in paragraph \ref{subs:ARfamily}. To this map, via Lemma \ref{lem:ARconnection}, we associate a map $F^{\aaa} \in \CETthreehalf$ with a triple $(l_1, l_2, l_3) \in {\Delta}_2$ of parameters defined by

\begin{equation}\label{eq:Arnoux-Yoccoz}
l_1:=\frac{1-\aaa}{2}, \; \; l_2:=\frac{1-\aaa^2}{2}, \; \;  l_3:=\frac{1-\aaa^3}{2}
\end{equation}

and a periodic triangle tiling (via the vocabularly established in paragraph \ref{subs:vocabularly}) with the angles of tiles defined by

\begin{equation}\label{eq:Tribonacci_triangle}
\alpha=\frac{\pi}{2} (1-\aaa), \beta=\frac{\pi}{2} (1-\aaa^2), \gamma=\frac{\pi}{2} (1-\aaa^3).
\end{equation}

In other words, $\rho_{\Delta}=(\aaa, \aaa^2, \aaa^3)$ with $\aaa$ defined by \eqref{eq:Arnoux-Yoccoz} and $\rho_{\Delta}$ defined by \eqref{eq:rho_delta}.

Here the angles approximatively are equal to $\alpha \approx 41^{\circ}, \beta \approx 63^{\circ}, \gamma \approx 76^{\circ}$. We call such a triangle a \textbf{Tribonacci triangle} and a corresponding billiard the \textbf{Tribonacci (triangle tiling) billiard}. In some sense, this billiard is the simplest one from all those that admit exceptional trajectories, because of its autosimilarity properties that we discuss in the following.

\smallskip

As discussed already in Section \ref{sec:arithmetic_orbits}, the symbolic dynamics of orbits of the Tribonacci billiard coincides with the arithmetic orbits of the Arnoux-Yoccoz map. In this paragraph we prove the convergence of such arithmetic orbits to the Rauzy fractal.  We first give some reminders about the classical Rauzy substitution and the Rauzy fractal.

\bigskip

A \textbf{Tribonacci substitution} $\sigma_R$ is a map on the words in the alphabet $\mathcal{N}_{\Delta}=\{1,2,3\}$ defined as the extension of the following map:
\begin{equation*}
\sigma_R: \; \;  \left\{
\begin{array}{cc}
1 \mapsto 12 \\
2 \mapsto 13 \\
3 \mapsto 1 
\end{array}
\right..
\end{equation*}

The substitution $\sigma_R$ has a unique fixed point $w_R \in {\mathcal{N}_{\Delta}}^{\N}$ (i.e. $w_R$ is an infinite word such that $\sigma_R(w_R)=w_R$), which starts as $w_R:=1213121121312 \ldots $.

\smallskip

We interpret the sequence $w_R$ as an infinite ladder in the space $\R^3$ with standard cartesian coordinates (we fix a standard basis $e_1=(1,0,0), e_2=(0,1,0)$ and $e_3=(0,0,1)$). Each subsequent symbol $(w_R)_j \in \mathcal{N}_{\Delta}, j\in \N$ is interpreted as an addition of the step $e_{(w_R)_j}$ to the growing ladder. We color an endpoint of each of these added vectors in one of three colors. The infinite ladder constructed in this way has a principal direction. After projecting on a plane orthogonal to this direction, we consider the image of the set of endpoints. This set is, by definition, the \textbf{Rauzy fractal}, a self-similar set defined by G. Rauzy in $1981$, see \cite{Rauzy}.

\smallskip

To a classic Tribonacci substitution $\sigma_R$ one also associates a sequence of \textbf{Tribonacci numbers}, i.e. the sequence of lengths of iterations of the word $123$ under the action of the substitution $\sigma_R$:

\begin{equation}\label{eq:tribonacci_substitution}
T_{n+4}:=\left|\sigma_R^n(123)\right|, n \in \N.
\end{equation}

It is standard (and trivial) that for all $n$, after setting $T_1=T_2=T_3:=1$ we have

\begin{equation}\label{eq:def_tribonacci}
T_{n+3}=T_{n+2}+T_{n+1}+T_{n},
\end{equation}

which can also be seen as the definition. The sequence $\{T_n\}$ is a generalization of the Fibonacci sequence (hence the name). The first $20$ terms of the Tribonacci sequence \footnote{Tribonacci sequence is the A000213 sequence of the on-line encyclopedia of integer sequences, see \url{https://oeis.org/A0000213} for more details.} are
\begin{equation*}
1, 1, 1, 3, 5, 9, 17, 31, 57, 105, 193, 355, 653, 1201, 2209, 4063, 7473, 13745, 25281, 46499,...
\end{equation*}

\bigskip

The following is based on an important self-similarity property of the number $\aaa \in \R$ defined by \eqref{eq:Arnoux-Yoccoz} which has been already used in many contexts. The point $\rho_{\Delta}=(\aaa, \aaa^2, \aaa^3) \in \Delta_2$ is a $3$-periodic point of the Rauzy subtractive algorithm (and a corresponding triple \eqref{eq:Arnoux-Yoccoz} a $3$-periodic point of the fully subtractive algorithm). Indeed, we have the following calculation:
\begin{equation*}
\left(\aaa, \aaa^2, \aaa^3\right) \mapsto \frac{1}{\aaa}\left(\aaa-\aaa^2-\aaa^3, \aaa^2, \aaa^3\right)=\left(\aaa^3, \aaa^2, \aaa\right).
\end{equation*}

In the context of the renormalization on $\CETthree$, we see that
 $R_3 R_2 R_1 F^{\aaa}=F^{\aaa}$. Actually, the map $R_1 F^{\aaa}$ is the same as $F^{\aaa}$ but with the labels of the intervals of continuity  changed. As an abstract tiling, the Tribonacci tiling is a fixed point of the renormalization process. From Theorem \ref{thm:minimality}, the map $F^{\aaa}$ is minimal. Hence, by Proposition \ref{prop:SYMBOLIC}, the symbolic dynamics of its generic point is an invariant point of a substitution $\sigma:=\sigma_1 \circ \sigma_2 \circ \sigma_3$ with $\sigma_j$ defined explicitely by \eqref{eq:sigmas} for all $j \in \mathcal{N}_{\Delta}$. All the corresponding (passing through circumcenters of tiles, in any direction) tiling billiard trajectories are escaping, as already has been noticed in \cite{BDFI18}.

\bigskip

In order to control the relabelling, define a following map 
$\upsilon_{\mathrm{rel}}$ on the alphabet $\mathcal{A}_{\Delta}$. We define $\upsilon_{\mathrm{rel}}$ on the letters by
\begin{equation*}
\upsilon_{\mathrm{rel}}: 
 \; \;  \left\{
 \begin{array}{cc}
 a \mapsto b, \\
 b \mapsto c, \\
 c \mapsto a,
 \end{array}
 \right.
\end{equation*}
and extend this definition to all words in $\mathcal{A}_{\Delta}$. 

Define now a substitution $\varsigma_R:=\upsilon_{\mathrm{rel}} \circ \sigma_3$ on periodic words in the alphabet $\mathcal{A}_{\Delta}$. We remind our reader that the substitution $\sigma_3$  has been defined in Proposition \ref{prop:SYMBOLIC}, see \eqref{eq:sigmas}, by
\begin{equation}\label{eq:sigma_1}
\sigma_3: \; \;  \left\{
\begin{array}{cc}
a \mapsto a \\
b \mapsto b \\
c \mapsto bac, \textit{if a precedent symbol was not}\; \;  b \\
c \mapsto bca, \textit{if a precedent symbol was not}\; \;  a. \\
\end{array}
\right.
\end{equation}

\smallskip

Since the words in the orbit $\{\varsigma_R^j(cba)\}_j$ do not contain two equal letters subsequently (since they all correspond to Tribonacci tiling billiard trajectories as discussed above), one can define $\varsigma_R$ as 
\begin{equation}\label{eq:varsigma}
\varsigma_R: \; \;  \left\{
\begin{array}{cc}
a \mapsto b\\
b \mapsto c \\
c \mapsto cba, \textit{if a precedent symbol is}\; \;  a \\
c \mapsto bca, \textit{if a precedent symbol is}\; \;  b\\
\end{array}
\right..
\end{equation}

We now define the \textbf{factorization map} $\upsilon_{\mathrm{fac}}$ on the words in $\mathcal{A}_{\Delta}^2$ (or, equivalently, on the even-length words in $\mathcal{A}_{\Delta}$).  Define $\upsilon_{\mathrm{fac}}: \mathcal{A}_{\Delta}^2 \rightarrow \mathcal{N}_{\Delta}$ as first, defining it explicitely on letters by 
\begin{align*}
\upsilon_{\mathrm{fac}} (ab)=\upsilon_{\mathrm{fac}}(ba):=3,\\
\upsilon_{\mathrm{fac}} (ac)=\upsilon_{\mathrm{fac}}(ca):=2,\\
\upsilon_{\mathrm{fac}} (cb)=\upsilon_{\mathrm{fac}}(cb):=1,
\end{align*}
and then extending it to words. 

\smallskip

Define for any map $\varphi: (\mathcal{A}_{\Delta}^2)^{\N} \rightarrow (\mathcal{A}_{\Delta}^2)^{\N}$ its \textbf{factorization} $\varphi^*: (\mathcal{N}_{\Delta})^{\N} \rightarrow (\mathcal{N}_{\Delta})^{\N}$
as the solution of the following commutative relationship: 
\begin{equation*}
\upsilon_{\mathrm{fac}} \circ \varphi =  \varphi^* \circ \upsilon_{\mathrm{fac}}.
\end{equation*}

The connection between $\sigma_R$ and $\varsigma_R$ is now apparent through factorization.

\begin{proposition}\label{prop:periods}
The following holds for the factorizations  $\sigma_j^*, j \in \mathcal{N}_{\Delta}$, $\upsilon_{\mathrm{rel}}^*$ and $\varsigma_R^*$ of the substitutions $\sigma_j,  j \in \mathcal{N}_{\Delta}$, $\upsilon_{\mathrm{rel}}$ and $\varsigma_R$:
\begin{enumerate}
\item[1.] these  factorizations  are well defined, 
\item[2. ] even though the substitutions $\sigma_j$ are defined only for periodic words, their factorizations $\sigma_j^*$ are well defined for non-periodic words,
\item[3.] $\varsigma_R^*=\sigma_R$.
\end{enumerate}
\end{proposition}

\begin{proof}
This is a simple verification. First, for the definition of $\sigma_1^*$, we study the action of $\sigma_1$ on two-letter words. Indeed, we have the following relations
\begin{align*}
\; \; \; \; \sigma_1 (ab) = \; \;  \left\{
\begin{array}{cc}
cbac \\
bcab \\
\end{array}
\right.,\\
\sigma_1 (ba) = bcba.
\end{align*}

These three equations factorize correctly into one equation $\sigma_1^*(3)=13$ which proves that $\sigma_1^*(3)$ is well-defined.

Similarly, $\sigma_1(ac)=cbac$ or $bcac$ and $\sigma_1(ca)=cbca$ and $\sigma_1^*(2)=12$ is well defined. Finally, since $\sigma_1(bc)=bc, \sigma_1(cb)=cb$ then  $\sigma_1^*(1)=1$. This defines the map $\sigma_1^*$ on the elements of the set $\mathcal{N}_{\Delta}$ and then on all words of this alphabet by extension. By analogously proceeding with $\sigma_2$ and $\sigma_3$ one obtains well-defined maps:
\begin{equation*}
\sigma_1^*: \; \;  \left\{
\begin{array}{cc}
1 \mapsto 1 \\
2 \mapsto 12 \\
3 \mapsto 13 
\end{array}
\right.,\;\;\;
\sigma_2^*: \; \;  \left\{
\begin{array}{cc}
1 \mapsto 21 \\
2 \mapsto 2 \\
3 \mapsto 23 
\end{array}
\right.,\;\;\;
\sigma_3^*: \; \;  \left\{
\begin{array}{cc}
1 \mapsto 31 \\
2 \mapsto 32  \\
3 \mapsto  3
\end{array}
\right..
\end{equation*}

Then obviously, the factorization of the map $\upsilon_{\mathrm{rel}}$ is given by 

$$
 \upsilon_{\mathrm{rel}}^*: \; \;  \left\{
\begin{array}{cc}
1 \mapsto 2 \\
2 \mapsto 3 \\
3 \mapsto 1 
\end{array}
\right..
$$

The final calculation gives that $\varsigma_R^*= \upsilon_{\mathrm{rel}}^* \circ \sigma_3^*= \sigma_R$. For $\sigma=(\sigma_1 \circ \sigma_2 \circ \sigma_3)$ we have $\sigma^*=\sigma_1^* \circ \sigma_2^* \circ \sigma_3^* = \sigma_R^3$.
\end{proof}

\begin{remark}
One can also associate
$1^+$ to $bc$, $1^-$ to $cb$,
$2^+$ to $ca$, $2^-$ to $ac$,
$3^+$ to $ab$, $3^-$ to $ba$. In this case one defines a substitution on \emph{cyclic} words in $\{1^+, 1^-, 2^+, 2^-, 3^+, 3^-\}$. By identifying $1^+=1^-, 2^+=2^-, 3^+=3^-$, we get the substitutions in Proposition above. 
\end{remark}

\begin{example}
The image $\sigma(\overline{cba})$ (as that of the periodic word $cbacba$) is calculated as follows:
\begin{equation*}
\overline{cba} \xmapsto{\sigma_3}  \overline{bacba} \xmapsto{\sigma_2} \overline{cabacacba} \xmapsto{\sigma_1} \overline{cbcabcbacbcacbcba}.
\end{equation*}
The corresponding relabelled sequence of images gives with $w_j=\varsigma_R^{j-1}\left(\overline{cba}\right), j \in \N^*$: 

\begin{equation*}
w_1=\overline{cba} \xmapsto{\varsigma_R} w_2=\overline{cbacb} \xmapsto{\varsigma_R}
w_3= \overline{bcacbcbac} \xmapsto{\varsigma_R}
w_4=\overline{cbcabcbacbcacbcba}
\end{equation*}

%


Since the word $\upsilon_{\mathrm{fac}} (cbacba)=123$, from the Proposition \ref{prop:periods} follows that $\upsilon_{\mathrm{fac}} (w_j)=\sigma_R^{j-1}(123)$. 

All of the words $w_j$ correspond to the symbolic trajectories in the Tribonacci billiard, and these words are the \emph{only} possible behaviors of trajectories on such tilings, see paragraph \ref{subs:vocabularly} and Theorem \ref{thm:renormalization_process}. \end{example}

\bigskip

The example above and Proposition \ref{prop:periods} already sketch the connection between the symbolic dynamics of the Tribonacci billiard and the Rauzy fractal. In the following, we make this connection precise. For this, we choose the markings of the periodic words $w_j$ in order to be able to treat them as strings of letters and not as cyclic words. Before doing so, we introduce the following notations: we write $U_1=U_2$ for two elements $U_1, U_2 \in \mathcal{A}_{\Delta}^{\N}$ if their corresponding cyclic words are equal and we write $U_1\equiv U_2$ if these two elements coincide symbol by symbol as elements in $\mathcal{A}_{\Delta}^{\N}$.

\smallskip

Define a sequence of words $\{s_j\}_{j=-2}^{\infty}$,   $ s_j \in \mathcal{A}_{\Delta}^{\N}$ with $s_j \equiv s_j^{1} \ldots s_j^{l_j}$ with $s_j^i \in \mathcal{N}_{\Delta}, i =1, \ldots, l_j$. Here $l_j=|s_j| \in \N$. First let $s_{-2}:\equiv a, s_{-1}: \equiv b, s_{0}: \equiv c, s_1:\equiv cba$.

Then, for any $j \in \N^*$ we define the word $s_{j+1}$ in a following reccurent way from the word $s_j$ already defined.
If $s_{j}^1 \neq c$, let $s_{j+1} :\equiv \varsigma_R (s_{j})$. Otherwise, if $s_{j}^1 = c$, $s_j= c \; s_j^2 \; \ldots \; s_j^{l_j}$ and we define a string $s_{j+1}$ (still equal to $\varsigma_R (s_{j})$ as a cyclic word) by shifting its beginning two steps to the right with respect to $\varsigma_R (s_{j})$. Indeed, we have 
\begin{equation*}
\varsigma_R (s_j) \equiv k_1^j \; k_2^j \; a \; \varsigma_R(s_j^2) \ldots \varsigma_R(s_j^{l_j}) = a \; \varsigma R(s_j^2)\; \ldots \varsigma_R(s_j^{l_j}) \; k_1^j \; k_2^j =: s_{j+1}.
\end{equation*}

Here $(k_1^j, k_2^j)=(b,c)$ if $s_j^{l_j}=b$ and $(k_1^j, k_2^j)=(c,b)$ if $s_j^{l_j}=a$, by the definition \eqref{eq:varsigma} of the substitution $\varsigma_R$. Define the cyclic words $w_j:=s_j^2$. Obviously, as cyclic words, they are as above, the images of the word $\overline{abc}$ under the substitution $\varsigma_R$. Denote by $P_j$ the length of the word $w_j$, i.e. $P_j=2 l_j$ with $l_j=|s_j|$. We also define the word $w_{\infty}$ as a fixed point of $\varsigma_R$.

Very importantly for the following, we consider $s_j$ as string words\footnote{The string words $s_j$ coincide with the symbolic codes of a singularity for the maps $F^{\boldsymbol{a}}_r$ in the family of real-lef deformations for the Arnoux-Yoccoz map $F^{\boldsymbol{a}}$, with the parameter $r \rightarrow 0$ as $j \rightarrow \infty$.}, and we define $w_j:=s_j^2$ as cyclic words. 

\smallskip

We introduce several additional notations. Let $W \subset \mathcal{A}_{\Delta}^{\N}$ be defined as $W:=\{w_j, w_j \in  \mathcal{A}_{\Delta}^{\N}\}$.

Moreover, for any word $w \in \mathcal{A}_{\Delta}^{\N}$, if this word finishes by a symbol or a word $\varkappa$, we denote by $w^{\varkappa}$ this same word without its last symbol or last word $\varkappa$.  By definition, $w= w^{\varkappa} \varkappa$.

\begin{example}
The next $4$ elements of the sequence $\{s_j\}, $ are 
\begin{align*}
s_2: \equiv acbcb,\\
s_3:  \equiv bcbacbcac,\\
s_4: \equiv cbcacbcbacbcabcba,\\
s_5: \equiv acbcabcbacbcacbcbacbcabcbcacbcb.
\end{align*}
\end{example}

Now we are ready to prove Theorem \ref{thm:convergence to the Rauzy fractal}. We actually prove a following (stronger) statement.

\begin{theorem}[Combinatorics of Tribonacci billiards]\label{thm:big}
Consider the Tribonacci billiard. Take an oriented trajectory $\delta_{AY}$\footnote{The subscript AY is here to remind that $\delta_{AY}$ corresponds to an arithmetic orbit of the Arnoux-Yoccoz map.} in this tiling passing by a circumcenter of some tile it crosses. Suppose first that $\delta_{AY}$ is not singular. Then the following holds for the trajectories in the parallel foliation $\mathcal{P}^{\delta_{AY}}$:
\begin{itemize}
\item[1.] all of the leaves (except for $\delta_{AY}$) in  $\mathcal{P}^{\delta_{AY}}$ are \emph{periodic} tiling billiard trajectories and $\delta_{AY}$ passes by all tiles,
\item[2.] for any periodic trajectory $\delta$ (once oriented as turning counterclockwise), there exists $j \in \N^*$ such that a word $w_j=\varsigma_R^{j-1}(\overline{acb}), w_j \in W$ caracterizes its symbolic dynamics, and $w_j=s_j^2$. Moreover, $\upsilon_{\mathrm{fac}}(w_j)=\sigma_R^{j-1}(123)$ and $\upsilon_{\mathrm{fac}}(w_{\infty})=w_R$.
The period of $\delta$ is then a doubled Tribonacci number $2 T_{j+3}$, see equation \eqref{eq:def_tribonacci},
\item[3.] any fixed periodic trajectory $\delta \in \mathcal{P}^{\delta_{AY}}$ with combinatorics $w_j$ defines a unique family  $\Gamma_{\delta}=\{\gamma_k, k \in \N\}$ of flowers $\gamma_k$ in $\mathcal{P}^{\delta_{AY}}$ (except for $\gamma_2$ which is not a flower but a petal of $\gamma_3$) with pistils in vertices $v_k \in V$ \footnote{For the exceptional case of $\gamma_2$, we define $v_2=v_3$ and the petal $\gamma_2$ is attached to $v_3$.} that satisfy the following properties:
\begin{itemize}
\item[a.] if $j \geq 3$, the trajectory $\delta$ with combinatorics $w_j$ is contracted onto the flower $\gamma_j$ with the same combinatorics, if $=1$ then $\delta$ contracts on the flower $\gamma_0$, and if $j=2$ it contracts on the flower $\gamma_1$ (inside), and outside onto the petal $\gamma_2$,
\item[b.] all of the flowers in $\Gamma$ pass by all of the the six tiles in $\theta \subset \Theta_{v_1}=\cup_{\theta \ni v} \theta$,
\item[c.] for all $k \in \N^*, k \neq 2$, a flower (petal) $\gamma_k$ has combinatorics $w_k$,
\item[d.] the flower $\gamma_0=v_0 \in V$ is a vertex, $\gamma_1$ is a one-petal flower, $\gamma_2$ is a petal of a two-petal flower $\gamma_3$, all $\gamma_k$ with $k \geq 4$ are flowers with three petals,  
\end{itemize}
\item[4.] the family $\Gamma_{\delta}$ has the following autosimilarity properties: 
\begin{itemize}
\item[4.1] for any $k \geq 4$, a flower $\gamma_k$ has three petals with combinatorics $w_{k-3}, w_{k-2}, w_{k-1}$, and is contained in one petal of the flower $\gamma_{k+1}$ with \emph{three} petals of combinatorics correspondingly $w_{k-2}, w_{k-1}, w_k$; a flower $\gamma_3$ has two petals of combinatorics $w_2, w_1$, and a flower $\gamma_1$ also has one petal of combinatorics $w_1$,
\item[4.2] the string symbolic words  $s_k^2, j \in \N \cup \{-2,-1\}$ of all $\gamma_k \in \Gamma_{\delta}$ satisfy the following relationships, depending on the value of $\varepsilon(k):= k\; \mathrm{mod}\;3$,
\begin{equation}\label{eq:formulation}
(s_{k-3} \cdot s_{k-3})^{\ast} \; \dagger \; (s_{k-2}^2)^{\star} \; \ast \; (s_{k-1}^2)^{\dagger} \; \star  \equiv s_{k-3} \cdot ((s_k)^2)^{s_{k-3}} =  s_k^2= w_k,
\end{equation}
where $(\ast, \star, \dagger): \N^* \rightarrow  \mathcal{A}_{\Delta}^3$ is defined explicitely by $(\ast, \star, \dagger)=(c,a,b)$ if $\varepsilon= 0$, $=(a,b,c)$ if $\varepsilon=1$, $=(b,c,a)$ if $\varepsilon=2$.
Moreover, the edges inserted in between for any flower $\gamma_k$ meet in the same point $v_k$ defined above (for all $k \neq 2$ this point is the pistil of the flower $\gamma_k \in \Gamma$). On each new step of the construction, the pistil $v_{k+1}$ is uniquely defined as a vertex such that first, $v_{k+1} \notin \Omega^{\gamma_k}$ and such that $v_{k+1}$  belongs to the edge $e$ defined as follows. The edge $e$ is crossed by the smallest of the three petals of the flower $\gamma_k$ on the half of its length (in the symbolic code \eqref{eq:formulation}, it corresponds to the middlepoint $\cdot$  marked in the code \eqref{eq:formulation},
\end{itemize}
\item[6.] for any flower $\gamma_k, k \geq 4$, we denote by $\Omega_k^1, \Omega_k^2, \Omega_k^3$ the unions of all the tiles by which pass its petals with combinatorics $w_{k-1}, w_{k-2}$ and $w_{k-3}$ correspondingly.\footnote{In other words, the biggest, the middle and the smallest petals of the flower $\gamma_k$.} Then  for a matrix $A=\begin{pmatrix}
-a&1\\
-1-a^2&-1
\end{pmatrix}$\footnote{Here $A=B^{-1}$ with $B$ defined on p. $151$ of the initial article \cite{Rauzy} by G. Rauzy, where he defines his fractal for the first time.} one has up to an isometry,
\begin{align*}
\Omega_{k+1}^1=\Omega_k^1 \cup \Omega_k^2 \cup \Omega_k^3,\\
\Omega_{k+1}^2=A \Omega_{k}^1,\\
\Omega_{k+1}^3=A \Omega_{k}^2.
\end{align*}

This implies that the sequence of curves $A^{-k} \gamma_k$ approximates the Arnoux-Rauzy curve, and the sets of all barycenters of tiles in the partition $A^{-k} \Omega_k^1 \cup A^{-k}\Omega_k^2 \cup A^{-k}\Omega_k^3$ of $A^{-k} \Omega_{k+1}^1$, give better and better approximations of the Rauzy fractal  with its standard partition into $3$ cells. Finally, a sequence of curves $\{A^{-k} \delta_{AY}\}_{k \in \N^*}$ on the plane converges to the Arnoux-Rauzy curve, in restriction to the fundamental domain which is a limit set of the sets $A^{-k} \Omega_{k+1}$ in the Haudorff topology. The distance $d (\theta_n, \theta_0)$ between the triangle $\theta_n$ that $\delta_{AY}$ visits at its $n$th iteration and its initial triangle $\theta_0$ verifies 
$$
d(\theta_n, \theta_0) \sim C \cdot \sqrt{n},  n \rightarrow \infty.
$$
\end{itemize}

Moreover, if $\delta_{AY}$ is singular (in some point $v \in V$) then the corresponding foliation  $\mathcal{P}^{\delta_{AY}}$ has $5$ additional singular entering the tiles in $\Theta_v$. Each of the sectors defined by these rays is foliated by sequences of periodic orbits with growing periods that approach Rauzy fractal, up to the reparametrization described above.
\end{theorem}


\begin{proof}
We consider the parallel foliation $\mathcal{P}^{\delta_{AY}}$. We now show how the renormalization process defined in Section \ref{sec:trop_cool} translates to a construction of a growing sequence of flowers in this foliation, and completely describes the dynamics of $\delta_{AY}$.

The point 1. has been already proven in Theorem \ref{thm:exceptional_trajectories}. The point 2. follows from the renormalization (Theorem \ref{thm:renormalization_process}) and Proposition \ref{prop:SYMBOLIC} as explained before, the symbolic dynamics of the map $F^{\boldsymbol{a}}$ in the alphabet $\mathcal{A}_{\Delta}$ is given by the sequence $w_j=\varsigma_R^{j-1}(\overline{acb})$ and $w_{\infty}$. By Proposition \ref{prop:periods}, since $\varsigma^*_R=\sigma_R$ we obviously have that $|w_j|=2 T_{j+3}$ with the Tribonacci sequence defined by \eqref{eq:def_tribonacci}. 

By Theorem \ref{thm:bounded_flower_conjecture}, and since all of the trajectories $\delta \in \mathcal{P}^{\delta_{AY}}, \delta \neq \delta_AY$ are periodic, then the trajectories on one side of $\delta_{AY}$ have the same winding. 

Fix a trajectory $\delta$ with symbolic dynamics $w_j$. In order to construct the family $\Gamma_{\delta}$, we proceed in the following way. We contract  $\delta$ inside onto some flower, then choose the biggest petal of this flower, and a periodic trajectory approaching this petal from inside. This periodic trajectory contracts on another flower etc. Thus we construct a sequence of flowers $\gamma_k$ with diminishing periods till $\gamma_0=\{v_0\}$. Moreover, we already know that the only periods possible belong to the set of doubled Tribonacci numbers. Since the set of Tribonacci numbers is a number system, this implies that for $k \geq 4$, the petals of the flower $\gamma_k$  have combinatorics $w_{k-1}, w_{k-2}, w_{k-3}$ (since their lengths are $2T_{k+2}, 2T_{k+1}$ and $2 T_k$), for any step $k \geq 4$.\footnote{In the process of contraction, if one doesn't choose the biggest flower, not all of the periods (and combinatorics) will be realized.} The combinatorics of a sequence of flowers $\{\gamma_j\}$, with small indices (for $j \leq 4$), follows from explicit calculation, see Figure \ref{fig:NICE_FIGURE}. This finishes the proof of point 3. Indeed, by construction, all of the curves $\gamma_k$ pass by the six tiles in $\Theta_{v_0}$.

\smallskip

The family $\Gamma_{\delta}$ splits up the plane into open domains of trajectories with the same symbolic dynamics. The more these zones approach $\delta_{AY}$, the more the corresponding period grows. The point 4.1 follows from point 3.

\begin{figure}
\centering
\includegraphics[scale=0.18]{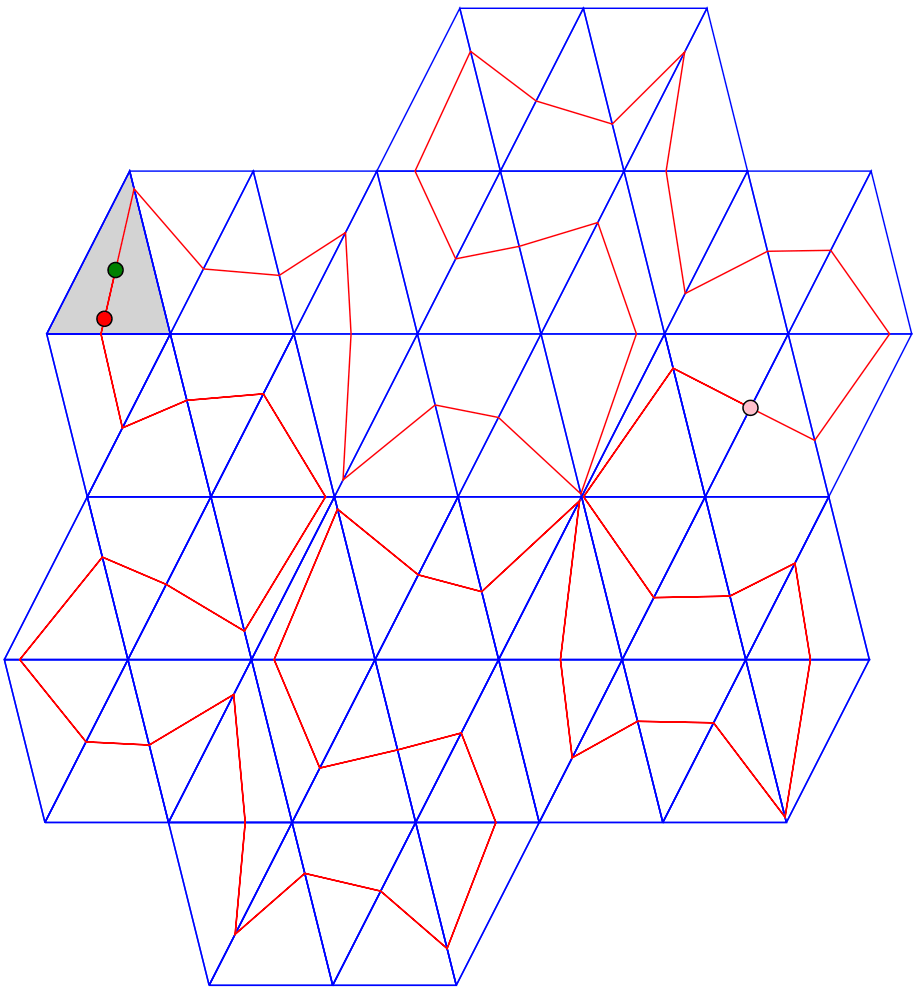}
\includegraphics[scale=0.05]{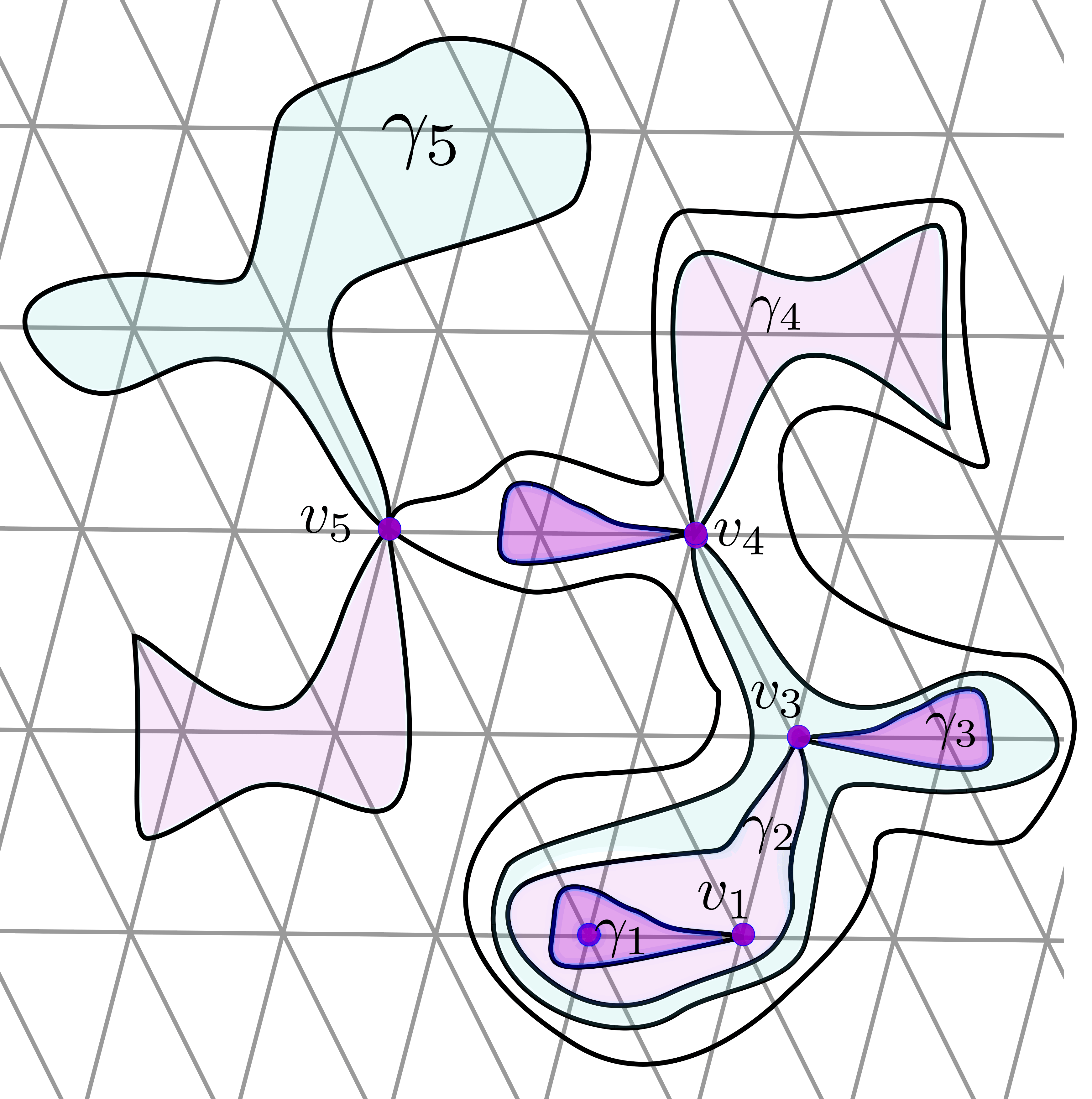}
\caption{\emph{Flowers} $\gamma_0, \ldots, \gamma_5$ \emph{in the foliation} $\mathcal{P}^{\delta_{AY}}$. On the left, we take any periodic trajectory $\delta$ of period $62 = 2 \cdot 31$. Then consider the foliation $\mathcal{P}^{\delta_{AY}}$ inside the domain $\Omega^{\delta}$. On the right is presented the structure of the flowers obtained by contraction of $\delta$. The vertices $v_j, j \in \{1,2,3,4,5\}$ are marked. Moreover, one can see a flower $\gamma_1$ (one petal), a petal $\gamma_2$ of the flower $\gamma_3$ (with two petals), and finally, the flowers $\gamma_4$ and $\gamma_5$ with three petals. The zones of equal symbolic behavior are colored in the same colors. Such a renormalization construction expands to all the plane. On the right, we didn't draw the exact trajectories but curves with equal symbolic codes. We note that all $\gamma_j$ pass by all of the tiles $\theta \subset \Theta_{v_0}$ (those crossed by $\gamma_1$).}
\label{fig:NICE_FIGURE} 
\end{figure}

One easily checks the statement 4.2. explicitely for all $k \leq 3$. Now take a flower $\gamma_4$ with three petals, and a vertex $v_4$. We look at the combinatorics of this flower by fixing the initial starting position to make a turn of the smallest of its petals (with combinatorics $w_1$). Then, the two sides of the equation \eqref{eq:formulation} correspond to the symbolic code of this flower. By the symbolic code of a flower we understand a symbolic code of a family of periodic trajectories approaching it from\emph{ outside. }

Indeed, the word $s_1 \cdot (s_4^2)^{s_1}$ obviously coincides cyclically with $w_4$. The left-hand side of \eqref{eq:formulation} coincides with $w_4$ as well since a flower is a union of three petals, in the presented order, which can be verified explicitely. The junctions $\ast, \star, \dagger$ correspond to the three edges that are crossed by a close periodic trajectory (and not contained in the flower itself). These are three edges such that $\ast \cap \star \cap \dagger=\{v_4\}$. 
Here $\ast=a, \star=b, \dagger=c$.

Now, the equation \eqref{eq:formulation} for any $k$ follows analogously. It suffices to say that the flower $\gamma_k$ is mapped 
to the next flower $\gamma_{k+1}$ va renormalization\footnote{The renormalization process has been defined on the orbits of the maps in $\CETthree$ but via the vocabularly in paragraph \ref{subs:vocabularly} it transfers to tiling billiard foliations.} and the pistil $v_k$ is mapped to the pistil $v_{k+1}$. The vertices $\{v_k\}$ are related to the symbolic dynamics in a following way. For any $k$ there exists a unique edge $e_k$ which is crossed by a smallest petal of $\gamma_k$ in the middle of its symbolic dynamics (starting from the vertex $v_k$)\footnote{This edge corresponds to a break point $\cdot$ in the equation \eqref{eq:formulation}.}. The vertex of this edge contained outside $\Omega^{\gamma_k}$ is exactly $v_{k+1}$, via renormalization.

Concerning point 5., the relationships between the sets $\Omega_k^j$ follow obviously from above. Indeed, at each new step of construction of a flower $\gamma_{k+1}$, its biggest petal "eats up" the flower $\gamma_k$, and the smaller petals of $\gamma_{k+1}$ are obtained from two biggest petals of $\gamma_k$ via renormalization. Then, since the square of the renormalization is the Rauzy substitution, the reparametrization matrix is the same as that in \cite{Rauzy}. All of the rest follows from standard results and arguments.

For now we didn't use the fact that $\delta_{AY}$ is non-singular. The difference between the non-singular and singular cases, is that in the first case $\delta_{AY}$ passes by all triangles in the tiling. In the second case, is is stopped in a vertex. But the limit set is also a Rauzy fractal but in this case this fractal grows only in some sector bounded by separatrix rays in the same flower as $\delta_{AY}$.
\end{proof}

This proof can be generalized to a more general case in order to prove the results on the convergence of other exceptional trajectories to fractals, at least for the periodic points in the Rauzy gasket. It is an interesting question to study convergence for all $\rho \in \mathcal{R}$.

\newpage
\begin{center}
{\textbf{ Part III.-- Generalizations and open questions.}}
\end{center}
\section{Dynamics of quadrilateral triangle tiling billiards.}\label{subs:Q}
The theory of tiling billiards in cyclic quadrilateral tilings is in many ways analogous to that of triangle tiling billiards. Indeed, a folding map into a disk is well defined (see Section \ref{sec:locally foldable tilings and associated billiard foliations}) as well as tiling billiard foliations (see Section \ref{sec:foliations}). Moreover, the connection with a family of fully flipped maps on the circle persists (see paragraph \ref{subs:CETandTriangles}). Although, the renormalization process we define for $\CETthree$ in Section \ref{sec:trop_cool} doesn't seem to extend (at least, in a straightforward way) to the family $\CETfour$. In this Section, we discuss the challenges and open questions. 

\subsection{Tree conjecture for quadrilateral tiling billiards.}\label{subs:tree_quadrilaterals}

Analogously to Conjecture \ref{conj:tree} for triangle tiling, we formulate 

\begin{conjecture}[Tree conjecture for cyclic quadrilateral tilings]\label{conj:tree_quadrilaterals}
Take any periodic trajectory $\delta$ of a cyclic quadrilateral billiard. Then the set $G_{\square}^{\delta}:={\Omega}^{\delta} \cap \Lambda_{\square}$ is a \emph{tree} (as a subgraph of $\Lambda_{\square}$).
\end{conjecture}



\smallskip

The Tree Conjecture for cyclic quadrilateral tilings seems to hold, based on simulations of dynamics.  Although, we didn't yet manage to prove it. By Proposition \ref{prop:treetoflower}, it is sufficient to prove the Bounded
Flower Conjecture for cyclic quadrilateral tilings. Even though one can prove easily the analogue of Proposition \ref{prop:list of possible local behaviours PTT}, the global symbolic behavior of quadrilateral tiling billiards seems to be much more complicated than that of triangle tilings.The trajectories in quadrilateral tilings are not symmetric, e.g. their symbolic codes do not necessarily belong to the set $\left\{4n+2, n \in \N^*\right\}$ since already on the square tilings there exist  $4$-periodic orbits. This is far to be an only example: there exist higly asymmetric trajectories, see for example that on Figure \ref{fig:Qasym}.

\begin{figure}
\centering
\includegraphics[scale=0.3]{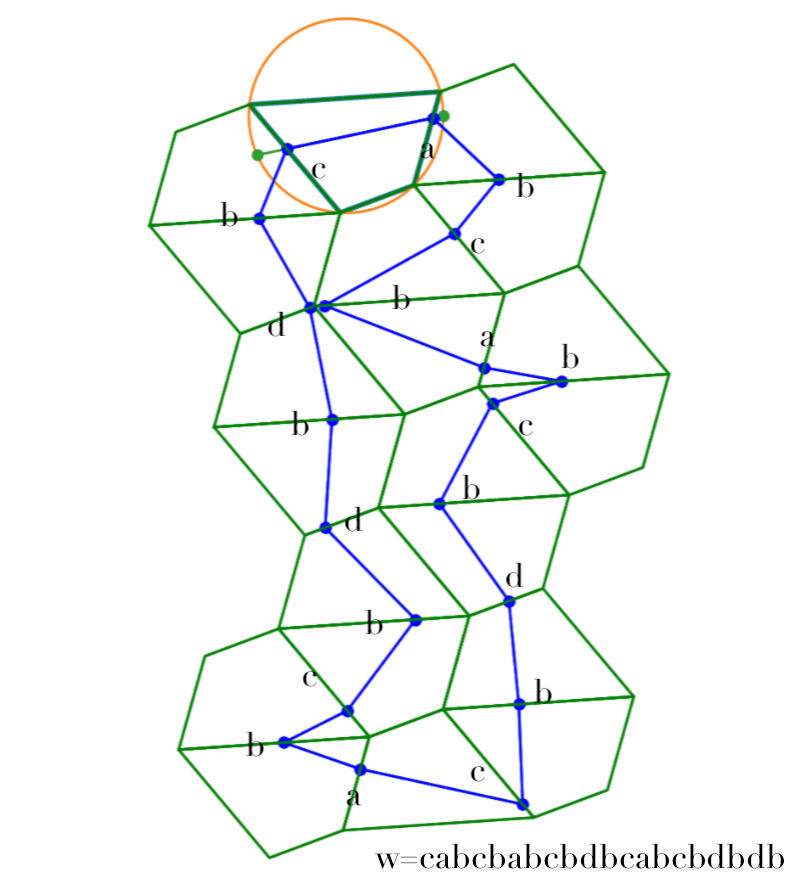}
\caption{Quadrilateral billiard trajectory with a symbolic code $w=cabcbabcbdbcabcbdbdb$.}
\label{fig:Qasym}
\end{figure}

\bigskip

We suspect that the analogue of the renormalization process that we have defined in Section \ref{sec:trop_cool} for $\CETthree$ still can be defined for $\CETfour$, even though with a more complicated combinatorics. This process should correspond to the contraction of flowers in the parallel foliation that has been disucssed in the proof of Proposition \ref{prop:treetoflower} and in the proof of Theorem \ref{thm:big}. We hope to obtain this process by contraction of measured foliations on the projective plane onto traintracks. 

\subsection{Density property for triangle and quadrilateral tiling billiards.}\label{subs:Density conjecture}
Periodic trajectories of triangle tiling billiards pass by all of the triangles they contour, by Theorem \ref{thm:main}. This behaviour may be generalized to hold for any, not necessarily periodic trajectory. Indeed, every non-periodic trajectory \emph{constructs dynamically} two graphs, both of which are trees.

\bigskip

Consider a (not necessarily periodic) trajectory $\delta$ in a triangle (or cyclic quadrilateral) tiling billiard. Define a subset $V(\delta) \subset V$ as $V(\delta):=\left\{
v \in V \mid \exists e \in E, e \ni v, \delta \cap e \neq \emptyset
\right\}$ and a coloring map $\mathcal{L}_{\delta}: V(\delta) \rightarrow \{0, 1\}$ step by step, as follows. 

\smallskip

First, pick some edge $e \in E$ that is crossed by $\delta$. Denote its extremities $w_0$ and $b_0$, in any arbitrary order. Add $w_0 \in \mathcal{L}^{-1}_{\delta}(0), b_0 \in \mathcal{L}^{-1}_{\delta}(1)$. To pass from step $j$ to the step $j+1$, we add  $b_{j+1} \in \mathcal{L}^{-1}_{\delta}(1), w_{j+1} \in \mathcal{L}^{-1}_{\delta}(0)$ in such a way that the following conditions hold:

\begin{itemize}
\item[•] either  $b_j=b_{j+1}$ or $w_j=w_{j+1}$,
\item[•] $b_j b_{j+1} \cap \delta = w_j w_{j+1} \cap \delta = \emptyset$,
\item[•]  $b_j w_{j+1} \cap \delta \neq \emptyset$,  $w_j b_{j+1}\cap \delta \neq \emptyset$.\footnote{Some of the edges here may be empty (degenerate into vertices). It may also happen for some $k<j-1$ that $b_j=b_k, k<j-1$.}
\end{itemize}

Define a subgraph $G_k^{\delta}$ of $\Lambda$ ($\Lambda=\Lambda_{\Delta}, \Lambda_{\square}$), $k \in \{0,1\}$ as a graph with the set of vertices coinciding with $\mathcal{L}^{-1}_{\delta}(k)$ and two vertices are connected by an edge of  $\Lambda$, if such an edge exists. 

\begin{theorem}[Density property]\label{thm:forest}
For any nonsingular triangle tiling billiard trajectory $\delta$, at least one of the graphs $G_k^{\delta}$ is a tree (say, $G_0^{\delta}$). A trajectory is periodic if and only if $G_1^{\delta}$ has a unique cycle in it.\footnote{In this case, of course, $G_0^{\delta}=G^{\delta}_{\Delta}$ defined in Theorem \ref{thm:main} and $G_0^{\delta} \subset \Omega^{\delta}$.} A trajectory $\delta$ is not periodic if and only if both of the graphs $G_j^{\delta}$ are trees, $j=0,1$.
\end{theorem}

The proof of the Density property follows the same strategy as the proof of Theorem \ref{thm:main}, we give here a sketch of its proof. Consider the parallel foliation $\mathcal{P}^{\delta}$ and perturb $\delta$ in it onto singular trajectories.

\smallskip

If $\delta$ is periodic, the two singular trajectories $\gamma_+, \gamma_-$ approaching $\delta$ are well defined (there are no accumulating trajectories in the neighbourhood of $\delta$ in $\mathcal{P}^{\delta}$). One of them (say, $\gamma_-$) is a bounded flower inside $\Omega^{\delta}$, and another one is a petal of a bigger (not necessarily bounded) flower. In this case, the statement of the Density conjecture follows directly from Theorem \ref{thm:main}. Indeed, since the graph $G_1^{\delta}$ uniquely defined by $G_0^{\delta}$ as the set of vertices at distance $1$ from $G_0^{\delta}$, and $G_0^{\delta}$ is a tree, $G_1^{\delta}$ has a unique cycle in it.

\smallskip

Now suppose that $\delta$ is escaping to infinity. If $\delta$ is exceptional then the Density property follows from Theorem \ref{thm:exceptional_trajectories}. Indeed, $\delta$ is an only non-bounded leaf in $\mathcal{P}^{\delta}$. In this case, the sets $G_k^{\delta}, k=0,1$ are the spanning trees of the initial graph $\Lambda_{\Delta}$. 

\smallskip

Finally, suppose $\delta$ is linearly escaping. In order to finalize the proof now, one classifies possible topological behaviours of unbounded flowers. The Proposition below finishes the proof.

\begin{proposition}\label{prop:UNBOUNDED} 
Consider an unbounded flower $\gamma$ in $v \in V$ with $s$ separatrix segments in $\Theta_v$. Suppose that at least one of these segments defines an escaping ray. Then, up to change of orientation, $\gamma$ has one of the types listed on Figure \ref{fig:unboundedflowers}. More precisely, for $s=2$ there are possible: two behaviours; for $s=4$: two behaviours with $0$ or $1$ bounded petals; for $s=6$: three behaviours  with $0,1$ or $2$ bounded petals.
\end{proposition}

\begin{proof}
The starting point of the proof is the Proposition \ref{prop:list of possible local behaviours PTT} that lists possible local behaviours of separatrix segments. It is left to exclude the following two obstructions for the behaviour of some unbounded flower $\gamma$. 
\begin{itemize}
\item[•] $s=4$ and there exists a closed petal in $\gamma$ passing by two opposite triangles,
\item[• ]$s=6$ and there exist two unbounded separatrix rays passing by neighbouring triangles, and the two bounded petals of the flower have different orientations. 
\end{itemize}

Both of these cases are excluded by a common symmetry argument. In both of the obstructions above, there exists a tile $\theta_0, \theta_0 \ni v$ such that $\gamma \cap \theta_0$ defines an unbounded separatrix ray and $\theta_0^v$ is contained inside some petal of $\gamma$. By point 1. in Proposition \ref{prop:nohungry}, one considers a symmetric flower $\gamma^v$ in the ray foliation. Then $\gamma$ and $\gamma^v$ necessarily intersect outside $v$ which gives a contradiction. 
\end{proof}

\begin{figure}
\centering
\includegraphics[scale=0.06]{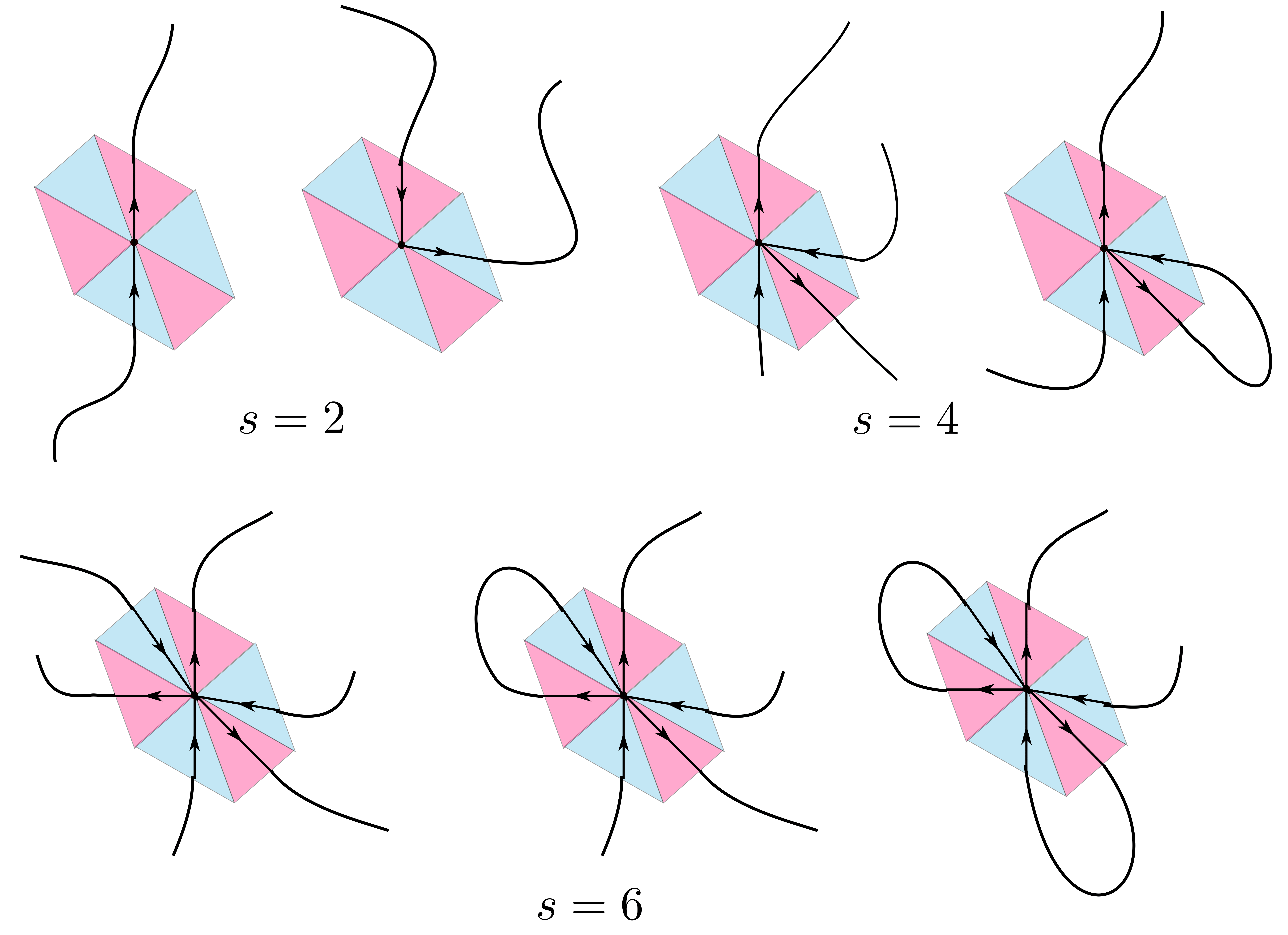}
\caption{Possible behaviors of unbounded flowers in parallel triangle tiling billiard foliations.}
\label{fig:unboundedflowers}
\end{figure}

\bigskip

\begin{conjecture}
Density property holds for quadrilateral tiling billiards.
\end{conjecture}

This Conjecture is a stronger form of Conjecture \ref{conj:tree_quadrilaterals}. The Density property for tiling billiards can be reformulated in terms of scissor cuts.

\bigskip

\textbf{Reformulated Density property.}
Consider a periodic (triangle or cyclic quadrilateral) tiling of the plane and fold the plane into a bellow. Then, cut along some line in the bellow. Then, the plane "falls into" an infinite number of connected components. The Density property is equivalent to the fact that \emph{none} of these components contain a full triangle.

\smallskip

Does the Density property (and hence, the Tree Conjecture) have a simpler proof based on this interpretation? 

\begin{remark}
A difficulty in proving such a statement directly is that when one makes a cut of the bellow, one does not cut out one trajectory but an infinite number of them. Moreover, the Density property doesn't follow purely from folding since there exist locally foldable tilings on which the Tree Conjecture is false.
\end{remark}

The next statement follows obviously from Theorem \ref{thm:main}. We present its proof in relation to the Reformulated Density property.

\begin{proposition}\label{prop:tree_trianglesimplecase}
There is no triangle tiling billiard trajectory $\delta$ that crosses the tiles $\theta^e, e=a,b,c$ and doesn't cross the tile $\theta$, surrounded by them.
\end{proposition}

\begin{proof}
Take any trajectory $\delta$, and consider its folding into a chord $l$ in the disk $\mathcal{D}$. We color each vertex $v \in V$ of the plane in one of the two colors depending on what side the vertex $\mathcal{F}(v)$ is with respect to the oriented chord $l$.\footnote{This coloring coincides with $\mathcal{L}_{\delta}$ on the set $V(\delta)$.}

Suppose now that $\delta$ as in the assumption exists. Then all of the vertices of $\theta$ are colored in the same color. Although, the vertices $A', B',C' $ of the tiles $\theta^e$ with $e=a,b,c$ that do not belong to $\theta$ are all colored in the opposite color. This is impossible since at least one of these three vertices lies on the same side of the chord $l$ as $A,B$ and $C$, by folding.
\end{proof}

\subsection{Symbolic dynamics of maps in  $\CETn$.}\label{subs:dynamics_CETn}
Even though there exists no periodic tiling by $n$-gones with $n \geq 5$, a geometric interpretation of the dynamics of maps in $\CETn$ exists.

\smallskip

Consider some cyclic polygon $P$ with $n$ sides and take $\tau \in \Sph^1$.  This data defines a map $F$ of \textbf{reflections in the circumcircle} as follows. Consider a chord in the disk bounded by the unit circle and connecting $0$ to $\tau$. Denote the sides of $\mathcal{P}$ by reading the boundary in a counterclockwise order, by $a_1, a_2, \ldots, a_n$. We put $a_{n+1}:=a_1$. For any $X \in \Sph^1$ we inscribe the polygon $\mathcal{P}$ in its circumcircle in such a way that the vertex $A:=a_1 \cap a_n$ is placed exactly in $X$. The map $F$ sends a polygon into a congruent polygon of different orientation sharing one side with ${P}$. A label of the side is defined by a positive intersection of $\mathcal{P}$ with a chord defined by $\tau$, see Figure \ref{fig:reflection_in_the_circumcircle}.\footnote{The dynamics of this system is equivalent  to that of the dynamics of a triangle (or cyclic quadrilateral) tiling billiard for $n=3$ (or $4$).} For any $n$, the data $({P}, \tau)$ defines  a map $F\in \CETn$. See \cite{Olga} for more details.

\begin{figure}
\centering
\includegraphics[scale=0.8]{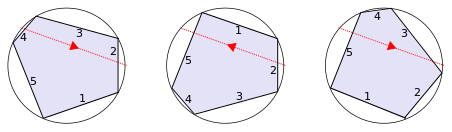}
\caption{Two iterations of a system of reflections of a pentagon in its circumcircle. After each iteration, one changes the direction of the chord defined by the parameter $\tau$ to its opposite. The symbolic code of the orbit $(X, F(X), F^2(X))$ in this case is $w= a_4a_2 \;\; a_2a_5$ for some $F \in \CETfive$.}
\label{fig:reflection_in_the_circumcircle}
\end{figure}

\bigskip

The following definition is inspired by our discussion with Pierre Dehornoy. 

Consider an alphabet $\mathcal{A}_n^2:=\{a_i a_j \mid i,j=1, \ldots n, i \neq j\}$.\footnote{Of course, $\mathcal{A}_n^2$ coincides with $\mathcal{A}_{\Delta}$ ($\mathcal{A}_{\square}$) for $n=3$($n=4$).} 
We say that the \textbf{winding} is a map $\wind: \mathcal{A}_n^2 \rightarrow \{0,1,-1\}$ defined on the letters by $\wind(a_i a_j)=1$ if $j=i+1$, $\wind(a_i a_j)=-1$ if $i=j+1$, and $\wind(a_i a_j)=0$ otherwise. Then the winding map is enlarged to all the words in the alphabet $\mathcal{A}_n^2$ by additivity.

\begin{remark}
The winding map is a generalization of the sign map defined in paragraph \ref{subs:MAIN}.
\end{remark}

\smallskip
 
For this paragraph, we define \textbf{periodic trajectories} in the system of reflections in the circumcircle as those  trajectories that are stable under a small perturbation of the polygon $\mathcal{P}$.\footnote{We give such a definition since the the drift-periodic trajectories of tiling billiards in triangles and quadrilaterals correspond to periodic trajectories in reflections in the circumcircle system and we want to exclude this case. Another way of approaching periodic trajectories is to exclude the cases of polygons with rationally dependent angles.}

One defines a \textbf{ winding of a periodic trajectory} of the system of reflections in the circumcircle as the winding of its symbolic code, and a \textbf{winding of a periodic trajectory} as a winding of a minimal period of its symbolic code.

\begin{example}
For a trajectory going clockwise around a vertex in a triangle (cyclic quadrilateral) tiling, its winding is equal to $6$ (or $4$). 
\end{example}

\begin{lemma}
The following holds :
\begin{enumerate}
\item[1.] a winding of a simple closed curve that doesn't touch the vertices in the triangle (cyclic quadrilateral) tiling is equal to $\pm 6$ ($\pm 4$) depending on its orientation,
\item[2.] 
for any $n \in \N, \geq 3$, the winding of a periodic trajectory in a system of reflections in the circumcircle is well-defined, i.e. doesn't depend on the string representation of the period. This winding is equal to $\pm 2n$ if $n$ is odd, and to $\pm n$ if $n$ is even. 
\end{enumerate}
\end{lemma}
 
 \begin{proof}
For point 1., consider a vector $v^{\perp}_{\delta}$ orthogonal to the curve $\delta$ and count the (algebraic) number of turns this vector makes when it moves along $\delta$. One can easily see that this number is exactly $\frac{1}{6} \wind(\delta)$ for the triangle tiling billiard and $\frac{1}{4} \wind(\delta)$ for cyclic quadrilateral tiling billiard by decomposing $\delta$ into the sum of loops.

For point 2., we first observe that the winding of a periodic trajectory is well-defined. Second, the only change in winding is done by the words that use subsequent letters. Even though for $n>4$ the corresponding tiling doesn't exist, one still can unfold the trajectory to some broken trajectory in a tiling with self-intersections. When one comes back to the same tile in the system of reflections, one comes back to the same tile on this unfolding as well.
 \end{proof}
 
\begin{conjecture}[Winding Conjecture]\label{conj:winding}
For any map $F \in \CETn$, a winding number of its periodic trajectory doesn't depend on a trajectory and is an invariant $\wind(F)$ of the map. Moreover, $\wind(F)=\pm 2n$ (if $n$ is odd), or to  $\wind(F)=\pm n$ (if $n$ is even).
\end{conjecture}

In terms of tiling billiards, this conjecture states that periodic orbits obtained by the same scissor cut have the same orientation with respect to infinity. The Winding Conjecture is our attempt to generalize the Tree Conjecture for any family $\CETn$, for all $n \geq 3$.  

From Theorems \ref{thm:main} and \ref{thm:forest} it follows, that the Winding Conjecture holds $n=3$. We believe that the Winding Conjecture holds for $n=4$, and we have no idea for $n>4$. Winding Conjecture concerns the behavior of the asymptotic cycle for families of translation surfaces. The difficulty is that these families are not generic and the maps in them are typically not minimal, so classical results do not apply.

\bigskip

\textbf{Problem.} Give an explicit description of minimal maps in $\CETn$ for any $n \geq 3$.

\smallskip

This Problem is answered for $n=3$ (see Theorem \ref{thm:minimality} above). Already for $n=4$ this question is open. In \cite{Olga} it has been shown that for $n=3$ and $n=4$ minimal maps in $\CETn$ belong to the hyperplane $\tau=\frac{1}{2}$. We wonder if one can provide a homological argument to prove this statement. Can one calculate a Hausdorff dimension of the set of minimal maps in $\CETfour$ and describe a possible analogue of the Rauzy gasket in this next dimension?

For $n \geq 5$ one may exhibit the examples of minimal maps in $\CETn$ outside the hyperplane $\{tau=1/2\}$. One could speculate that such a behavior of the family $\CETn$ (minimality implying $\tau=\frac{1}{2}$ for $n=3,4$ but only for these $n$) is related to the famous Novikov's conjecture on the chaotic sections of $3$-periodic surfaces, i.e. genus $3$ subsurfaces of a $3$-torus. Indeed, the squares of the maps in $\CETn$ for $n=3,4$ are interval exchange transformations corresponding to genus $3$ flat surfaces. 

\smallskip

To conclude, in this last Section we made an attempt to clarify the relationship between the Tree Conjecture (Density property, Winding Conjecture) for tiling billiards and the existence of renormalization in the family $\CETn$. For $n=4$, both of the questions are open. 

\newpage
\begin{center}
\textbf{Appendix: some comments on \emph{Triangle tiling billiards and the exceptional family of their escaping trajectories: circumcenters and the Rauzy gasket}.}
\end{center}

As pointed out before, while working on this article we have found a mistake in one of the proofs of our previous work with P. Hubert \cite{Olga}. This mistake does not influence the principal results of \cite{Olga} for triangle tiling billiards, except for the proof of $4n+2$ Conjecture. It is important for us to mention what was wrong in our arguments.

The present work gives a new set of tools for the study of triangle tiling billiards, and reproves all of the results in \cite{Olga}, in a simpler way. 
We remind our reader that the work \cite{Olga} approached the maps in the family $\CETthree$ with a tool of a standard Rauzy-Nogueira induction.

\smallskip


Here we revisit the proof of the [Proposition 6, \cite{Olga}] which is the key statement for the proof of the $4n+2$ Conjecture (4. and 5. in Theorem \ref{thm:triangle_tiling_billiards_info}) announced in \cite{Olga}. The statement in itself is correct, but a proof we propose in \cite{Olga} has a hole in it. We remind the statement as well as the idea of the initial proof, and then point out where the mistake is hiding.

\begin{proposition}\cite{Olga}\label{prop:prop6}
Take $F=F_{\tau}^{l_1, l_2, l_3} \in \CETthree$ such that $\frac{l_i}{l_j} \notin \Q$ for any $i \neq j$. Suppose that the Rauzy-Nogueira induction stops for $F$ at some $4$-interval exchange transformation $F'$. Then for any interval $Y \subset I$ of continuity for $F'$ such that $F'(Y)=Y$, the restriction $F' \mid_Y$ is an involution.
\end{proposition}

A proof of this proposition we propose in \cite{Olga} goes as follows. Take some map $F'$ with  $F' \mid_Y=\mathrm{id}$, follow  backwards the Rauzy-Nogueira induction and prove that such a path can't end up on a map in $\CETthree$. The argument is correct for all possible outcomes except for the case when $
F'=\begin{pmatrix}
Y&*&*&\bar{X}\\
Y&*&*&\bar{X}
\end{pmatrix}.
$ In the argument of the proof, one argues that the back-ward path has to go up into $Y$ losing to some (flipped) $Z$, as in

\begin{equation}\label{eq:path_back}
\begin{pmatrix}
Y&*&*&\bar{X}\\
Y&*&*&\bar{X}
\end{pmatrix} 
\leftarrow
\begin{pmatrix}
\bar{Y}&\ldots &\bar{Z}\\
\bar{Z}&\ldots &\bar{Y}
\end{pmatrix}.
\end{equation}

Then one concludes $Z=X$, and such combinatorics is indeed not possible for a map in the family $\CETthree$. A mistake in this reasoning is that for a matrix represented by the right-hand side of \eqref{eq:path_back} its number of columns may potentially be smaller than $4$, i.e. $Z$ is not necessarily equal to $X$. Indeed, there exist fully flipped $4$-interval exchange transformations with periodic orbits 
for which this happens, here is an example of the the Rauzy-Nogueira path for one of those maps:
 
\begin{equation}\label{eq:chain_quadrrilaterals}
\begin{pmatrix}
\bar{Z}&\bar{W}&\bar{X}&\bar{Y}\\
\bar{Y}&\bar{Z}&\bar{W}&\bar{X}
\end{pmatrix} 
 \xrightarrow[]{X>Y}
\begin{pmatrix}
\bar{Z}&\bar{W}&{Y}&\bar{X}\\
Y&\bar{Z}&\bar{W}&\bar{X}
\end{pmatrix} 
 \xrightarrow[]{W>Y}
 \begin{pmatrix}
\bar{Z}&\bar{Y}&\bar{W}&\bar{X}\\
\bar{Y}&\bar{Z}&\bar{W}&\bar{X}
\end{pmatrix} 
 \xrightarrow[]{Z>Y}
 \begin{pmatrix}
{Y}&\bar{Z}&\bar{W}&\bar{X}\\
{Y}&\bar{Z}&\bar{W}&\bar{X}
\end{pmatrix}.
\end{equation} 
 
Once the Rauzy-Nogueira induction stops, one re-iterates this induction on a smaller interval. Even though the statement of Proposition \ref{prop:prop6} still holds,  additional arguments have to be applied. By following the methods of \cite{Olga}, one can finish the proof with a more precise study of Rauzy-Nogueira graphs but our proof loses its interest since it becomes a case-by-case study of a big graph. 

\smallskip
Finally, a chain given in \eqref{eq:chain_quadrrilaterals} can be modified in order to construct a counterexample to Proposition \ref{prop:prop6} for the maps in the family $\CETfour$. Indeed, it suffices to add a fifth column $\begin{pmatrix}
\bar{V}\\
\bar{V}
\end{pmatrix}
$ to every matrix in a chain. Moreover, a matrix
\begin{equation*}
\begin{pmatrix}
\bar{Z}&\bar{W}&\bar{X}&\bar{Y}&\bar{V}\\
\bar{Y}&\bar{Z}&\bar{W}&\bar{X}&\bar{V}
\end{pmatrix}
\end{equation*}
corresponds to the dynamics of a map in $\CETfour$. It suffices to define $I_1:=Z, I_2:=W, I_3:=X, I_4:=Y \cap V$. This illustrates how the orbits of periods different from $4n+2$ may appear in cyclic quadrilateral tiling billiards, see paragraph \ref{subs:tree_quadrilaterals}. Moreover, the proof of the integrability result for $\CETfour$ (Proposition 9 in \cite{Olga}) is not finished because of the problem above. It also can be finished by an explicit graph study but we hope to find a simpler proof in the future.

\bigskip

To conclude, all of the statements of \cite{Olga} for triangle tiling billiards are correct, even though the proof of the $4n+2$ Conjecture is not finished in \cite{Olga}, but finished in the present work. Although, the work \cite{Olga} doesn't provide any understanding on the dynamics of quadrilateral tilings. The statement concerning the symbolic dynamics of trajectories (point 2, of Theorem 7 in \cite{Olga}) is false, and the proof of the integrability (Proposition 9) is not finished. Although, we strongly believe that the integrability property does hold for almost all quadrilateral tiling billiards, and reflects an interesting subcase of some version of Novikov's conjecture, see discussion at the end of Section \ref{subs:Q}.

\bigskip

\begin{center}
{ACKNOWLEDGEMENTS}
\end{center}
I am grateful to Pierre Arnoux, Dmitry Chelkak, Charles Fougeron, Pascal Hubert, Victor Kleptsyn, Paul Mercat, Julien Lavauzelle, Pierre Dehornoy, Valente Ramirez, Ferrán Valdez and Barak Weiss for friendly discussions on different stages of my creative process. I am especially grateful to Ofir David for his animations illustrating billiard foliations. The research for this paper has been done in many places and many countries (anyway, trees are everywhere...). I am especially thankful to IRMAR at University of Rennes 1, where I have spent the last two years, as well as to CIRM in Marseille where I have fallen in love with the Tree Conjecture presented in an amazing talk\footnote{One can see a video of the talk in \cite{Davis_talk}.} by Diana Davis, on the 14th February 2017.

\Addresses
\end{document}